\documentclass[11pt,a4paper]{amsart}

\usepackage{hyperref}
\usepackage{amsfonts}
\usepackage{amsmath}
\usepackage{amsthm}
\usepackage{latexsym}
\usepackage{amscd}
\usepackage{xypic}
\usepackage[all,cmtip]{xy}
\usepackage{mathrsfs}
\usepackage{graphicx}
\usepackage{enumitem} 
\usepackage{braket}
\usepackage[margin=1in]{geometry}
\usepackage{tikz}
\usetikzlibrary{calc,fadings,decorations.pathreplacing}
\usepackage{verbatim}

\allowdisplaybreaks

\numberwithin{equation}{section}

\newtheorem{prop}{Proposition}[section]
\newtheorem{theo}[prop]{Theorem}
\newtheorem{lemm}[prop]{Lemma}
\newtheorem{coro}[prop]{Corollary}

\newtheorem{defi}[prop]{Definition}
\newtheorem{conj}[prop]{Conjecture}

\newtheorem{quest}[prop]{Question}

\newcommand{\DD}{\mathbb{D}}
\newcommand{\EE}{\mathbb{E}}

\newcommand{\MM}{\mathbb{M}}
\newcommand{\NN}{\mathbb{N}}
\newcommand{\PP}{\mathbb{P}}
\newcommand{\QQ}{\mathbb{Q}}
\newcommand{\RR}{\mathbb{R}}

\newcommand{\ZZ}{\mathbb{Z}}

\newcommand{\cC}{\mathcal C}
\renewcommand{\cD}{\mathcal D}

\newcommand{\cG}{\mathcal G}
\renewcommand{\cH}{\mathcal H}

\newcommand{\cJ}{\mathcal J}

\newcommand{\cP}{\mathcal P}

\newcommand{\cS}{\mathcal S}

\def\fS{\mathfrak{S}}

\newcommand\addlabel[1]{\refstepcounter{equation}\tag{\theequation}\label{#1}}

\DeclareMathOperator{\tr}{tr}

\DeclareMathOperator{\supp}{supp}

\newcommand{\bangle}[1]{\left\langle #1 \right\rangle}
 
\DeclareMathOperator{\diam}{diam}
\DeclareMathOperator{\Ric}{Ric}

\newcommand{\pa}[2]{\frac{\partial #1}{\partial #2}}
\DeclareMathOperator{\Hess}{Hess}
\DeclareMathOperator{\Ent}{Ent}
\DeclareMathOperator{\Jac}{Jac}
\DeclareMathOperator{\vol}{vol}

\DeclareMathOperator{\proj}{proj}
\newcommand{\cost}{\text{\textrm{cost}}}
\DeclareMathOperator{\Lip}{Lip}
\newcommand{\Leb}{\text{\textnormal{Leb}}}
\newcommand{\Id}{\text{Id}}
\DeclareMathOperator{\Riem}{Riem}
\DeclareMathOperator{\Graph}{Graph}
\DeclareMathOperator{\Var}{Var}

\title{ Optimal Transport and Ricci Curvature: \linebreak Wasserstein space over the Interval}
\author{Otis Chodosh}
\thanks{This document was written in fulfillment of the essay option of Part III of the Mathematical Tripos at Cambridge University for the 2010-2011 academic year. I would like to thank Dr. Cl\'ement Mouhot for agreeing to set and mark this essay, assisting me in learning the material contained within, his extensive editing help, as well as for suggesting the problem of Ricci bounds on $(\cP(X),d^{W})$.  All of the figures are created using the TikZ package and I would like to thank Thomas Trzeciak for his spherical drawing code that I modified to create Figure \ref{fig:spherical-triangle}, which can be found at \url{http://www.latex-community.org/forum/viewtopic.php?f=4&t=2111}. Finally, I would like to thank the Cambridge Gates Trust for their financial support this year.}
\date{\today} 
\begin{document}

%
%

\begin{abstract}
In this essay, we discuss the notion of optimal transport on (geodesic) metric spaces, and the associated ($2$-)Wasserstein distance. We then examine displacement convexity of the entropy functional on $\cP(X)$ and associated synthetic Ricci lower bounds. In particular, we prove that the Lott-Villani-Sturm notion of generalized $\Ric\geq K$ agrees with the classical notion of $\Ric \geq K$ on smooth manifolds and that it is preserved under Gromov--Hausdorff convergence. We examine in detail the space of probability measures on the unit interval $[0,1]$, denoted $\cP_{0}$, equipped with the ($2$-)Wasserstein metric. We show that it is isometric to the space of nonnegative, nondecreasing functions on $[0,1]$ with the metric induced from the $L^{2}$ norm, and using this we show that $(\cP_{0},d^{W})$ has vanishing Alexandrov curvature, and give a direct proof that the entropy functional is displacement convex on $\cP_{0}$. We additionally examine finite dimensional Gromov--Hausdorff aproximations of $\cG_{0}$, and use these to construct the entropic measure on $\cG_{0}$ first considered by Von Renesse and Sturm, denoted $\QQ^{\beta}_{0}$. In addition, we examine properties of $\QQ^{\beta}_{0}$ and explain why one might expect that it has generalized lower Ricci bounds, which we show not to be the case. Finally, we discuss the possibility of finding a measure $\MM \in \cP(\cP_{0})$ such that $(\cP_{0},d^{W},\MM)$ has generalized Ricci lower bounds. 
\end{abstract}

\maketitle 
\tableofcontents

\section{Introduction} Lott and Villani in \cite{LottVillani:RicciViaTspt} and independently Sturm in \cite{Sturm:CurvDimCond, Sturm:MetMeas1, Sturm:MetMeas2} have developed a synthetic notion of lower Ricci bounds on geodesic metric spaces by considering displacement convexity of certain entropy functionals along geodesics in the space of probability measures on the underlying metric space. Their notion agrees with the classical condition of lower Ricci bounds on Riemannian manifolds, and is preserved under Gromov--Hausdorff convergence. In addition, Lott--Villani have used this notion to examine the relation between this displacement convexity and log Sobolev, Poincar\'e, and Talagrand inequalities. In this essay, we give an overview of this notion, discussing Wasserstein distance on the space of probability measures on a metric space, as well as proving the agreement of the synthetic Ricci bounds and classical Ricci bounds on Riemannian manifolds and the preservation of synthetic bounds under Gromov--Hausdorff convergence. We also give a detailed examination of the space of probability measures on the unit interval $\cP_{0} : =\cP([0,1])$, equipped with the Wasserstein metric. In this case, the one dimensionality of $[0,1]$ allows us to give more explicit formulas for the dynamics of optimal transport, and provides some insight into the more general case. We also examine the analytic consequences of a measure on $\cP_{0}$ with lower bounds on Ricci in the Lott--Villani--Sturm sense, and show that the entropic measure constructed by Von Renesse and Sturm in \cite{RenesseSturm:EntMeasDiff} does not admit lower Ricci bounds, even though there are good reasons to suspect that it does. 

The structure of this essay is as follows. In Section \ref{sect:geo-met-sp}, we discuss the notion of \emph{geodesic spaces}, which are metric spaces where the distance between two points is the same as the minimum of the lengths of all continuous paths between the two points. We will call curves achieving this minimum \emph{geodesics}. Then, in Section \ref{sect:opt-trans-was-dist}, we examine \emph{optimal transport} and \emph{Wasserstein distance}. In particular for a compact metric space $(X,d)$, this will allow us to define a metric on $\cP(X)$, the space of probability measures on $X$ by
\[ d^{W}(\mu,\nu)^{2} := \inf_{\pi \in \Pi(\mu,\nu)} \int_{X\times X} d(x,y)^{2}d\pi(x,y)\]
where $\Pi(\mu,\nu)$ is the set of \emph{transport plans}, i.e.\ probability measures on $X\times X$ with first and second marginals $\mu$ and $\nu$, respectively. We  show that this is a metric, turning $\cP(X)$ into a compact metric space, with the weak* topology. Finally, if $(X,d)$ is a geodesic space, then we will show that $(\cP(X),d^{W})$ is as well and we will discuss geodesics in this space. In Section \ref{sect:met-struct-cP0}, we discuss the space of probability measures on the unit interval, $\cP([0,1])$, which we view as a metric space with the Wasserstein metric. We show that this metric space, which we denote $(\cP_{0},d^{W})$ is isometric to $(\cG_{0},d^{L^{2}})$ where $\cG_{0} \subset L^{2}([0,1])$ is the set of nonnegative, nondecreasing right continuous functions on $[0,1]$ with $f(1) = 1$ and $d^{L^{2}}$ is the metric induced by $L^{2}$ distance. This isometry (which we denote $\Psi$) is given by the inverse distribution map in one direction, and the pushforward map applied to Lebesgue measure in the other direction. The isometry $\Psi$ will be very important to our understanding of $\cP_{0}$, because it allows us to work with $\cG_{0}$, which is a totally geodesic, compact subset of a Hilbert space, and thus has reasonably simple properties. For example, we show that geodesics are unique, and in fact are straight lines between their endpoints.

In Section \ref{sect:synth-sect-curv}, we describe a synthetic notion of sectional curvature. We say that a geodesic space has nonnegative sectional curvature in the Alexandrov sense if geodesic triangles in the space are ``fatter'' than corresponding triangles in the plane, $\RR^{2}$. In particular, we state Toponogov's theorem, which says that this is equivalent to nonnegative sectional curvature in the classical sense. Then, we go on to discuss the relation between the sectional curvature of $M$ and Alexandrov curvature $\cP(M)$ for $M$ a Riemannian manifold. As a special case of this consideration, we prove that that $(\cP_{0},d^{W})$ has ``vanishing Alexandrov curvature.'' In Section \ref{sect:synth-ricci-curv}, we define the entropy functional on $\cP(X)$ by 
\[ \Ent(\mu|m) = \begin{cases}
\int_{X} \rho \log \rho\  dm & \text{ for } \mu \ll m \text { and } \mu = \rho m\\
+ \infty & \text{ otherwise}
\end{cases}\]
which allows us to discuss displacement convexity for geodesic measure spaces $(X,d,m)$. We define the notion of \emph{weak a.c. displacement convexity} (cf.\ Definition \ref{defi:k-convex-ent}), which is the condition that for $\mu_{0},\mu_{1} \in \cP(X)$ with $\mu_{i} \ll m$, then there exists some geodesic in $\cP(X)$ between them, $\mu_{t}$, such that 
\[
\Ent(\mu_{t}|m) \leq t \Ent(\mu_{1}|m)  + (1-t) \Ent(\mu_{0}|m) - \frac K2 t (1-t) d^{W}(\mu_{0},\mu_{1})^{2}. 
\]
We then prove that for a compact Riemannian manifold, choosing the reference measure to be the normalized volume measure, i.e.\ $m = \frac{\vol_{M}}{\vol(M)}$, the triple $(M,d,m)$ has lower Ricci bounds $\Ric \geq K$ if and only if $\Ent(\cdot | m)$ is weakly a.c. displacement $K$-convex. In proving this, we discuss particulars about optimal transport on manifolds. Because of this result, we will often refer to weak a.c. displacement $K$-convexity as generalized $\Ric\geq K$. Additionally in this section, we explicitly investigate displacement convexity over the interval, proving directly that
\[
\Ent(\Psi(f) |{\Leb}) = - \int_{0}^{1} \log f'(x) dx.
\] 
for $f \in \cG_{0}$ (i.e.\ a nondecreasing function on $[0,1]$) where $\Psi(f) = f_{*}\Leb$ is the image of $f$ under the isometry $\Psi:(\cG_{0},d^{L^{2}}) \to (\cP_{0},d^{W})$. By concavity of $\log$, this allows us to directly prove that $\Ent(\cdot|\Leb)$ is displacement $0$-convex in this case. 

In Section \ref{sect:GH-conv}, we define the notion of \emph{Gromov--Hausdorff convergence} of metric measure spaces and give an overview of the proof that weak a.c. displacement convexity is preserved under this type of convergence. We say that a sequence of (compact) metric measure spaces $(X_{n},d_{n},\mu_{n})$  converges in the \emph{(measured) Gromov--Hausdorff topology} to $(X,d,\mu)$ if there is a sequence of $\epsilon_{n} >0$, with $\epsilon_{n}\to 0$ and associated maps $f_{n}:X_{n}\to X$ (not necessarily continuous) which are $\epsilon_{n}$-isometries, i.e.  
\begin{enumerate}
\item for $x,x' \in X$, $|d_X(x,x') - d_Y(f_{n}(x),f_{n}(x'))| \leq \epsilon$ 
\item for all $y \in Y$ there is $x\in X$ with $d_Y(y,f_{n}(x)) \leq \epsilon$
\end{enumerate}
and furthermore, that $(f_{n})_{*}\mu_{n} \to \mu$ in the weak* topology. We then show that generalized Ricci lower bounds are preserved under such convergence. We discuss the relation of this with Gromov's compactness theorem, which says that the set of Riemannian manifolds with a fixed dimension, diameter upper bound and Ricci curvature lower bound is precompact in this topology. Finally, we show how to approximate the space $(\cG_{0},d^{L^{2}})$ by finite dimensional metric spaces. Then, in Section \ref{sect:meas-cP0-ricci-curv}, we discuss reference measures on $(\cG_{0},d^{L^{2}})$ (or equivalently on $(\cP_{0},d^{W})$, but because $\cG_{0}$ is a totally geodesic subset of a Hilbert space, it is often easier to work with). First, we discuss the analytic consequences of such a measure having generalized Ricci lower bounds, in particular discussing the resulting log-Sobolev and Poincar\'e inequalities. Then, we will discuss the entropic measure $\QQ^{\beta}_{0}$, on $\cG_{0}$, constructed by Von Renesse and Sturm in \cite{RenesseSturm:EntMeasDiff}, defined by requiring that
 \begin{multline*}
 {\int_{\cG_{0}} u(g(t_{1}),\dots, g(t_{N}) )d \QQ^{\beta}_{0}} \\
 = \frac{\Gamma(\beta)}{\prod_{i=0}^{N} \Gamma(\beta(t_{i+1}-t_{i}))} \int_{\Sigma_{N}} u(x_{1},\dots,x_{N}) \prod_{i=0}^{N} (x_{i+1} - x_{i})^{\beta(t_{i+1}-t_{i})-1} dx_{1}\cdots dx_{N}.
 \end{multline*}
 We give a concrete proof of the existence of $\QQ^{\beta}_{0}$, which does not require the Kolmogorov extension theorem, as used by Von Renesse and Sturm. Then, we discuss why one might expect that $(\cG_{0},d^{L^{2}},\QQ^{\beta}_{0})$ would have generalized $\Ric \geq 0$, or at least bounded below by some real number, including the work by D\"oring and Stannat in \cite{Stannat:logsobWass}, showing that $(\cG_{0},d^{L^{2}},\QQ^{\beta}_{0})$ admits log-Sobolev and Poincar\'e inequalities. However, we go on to prove that this space does not have generalized $\Ric \geq K$ for any real $K$. We show this by using a generalization of the proof of log-concavity of displacement convex measures on a Hilbert space (cf.\ \cite{AGS:GradFlow}, Chapter 9). Finally, we discuss the possibilities of a measure $\MM \in \cP(\cG_{0})$ having generalized Ricci lower bounds, suggest methods for constructing such an $\MM$, as well as necessary conditions that could be used to rule out such a property for a given $\MM$, and give a brief overview of the theory of optimal transport in Hilbert spaces. 
 
 
We also include Appendix \ref{app:non-smooth-cov}, in which we discuss nonsmooth changes of variables, and formulas for the density of nonsmooth pushforward measures, results which we make use of in various places throughout the text, particularly in Section \ref{sect:synth-ricci-curv}. In Appendix \ref{app:riem-mfld-intro}, we give a very brief introduction to Riemannian manifolds and curvature. Then in Appendix \ref{app:mfd-lower-ric}, we give a brief overview of the theory of manifolds with lower Ricci bounds.
 
 A few words on notation and terminology are in order. We denote the space of probability measures on a topological space $X$ by $\cP(X)$, and for a map $f:X\to Y$ will write the \emph{pushforward} of a (probability) measure $\mu \in \cP(X)$ by $f_{*}\mu \in \cP(Y)$, which is the measure such that for $A \subset Y$, measurable, $f_{*}\mu(A) = \mu(f^{-1}(A))$. For a measure $m$ on $X$, we will write $\mu \ll m$ if $\mu(A) = 0$ for all measurable sets with $m(A) = 0$. If this holds, then we will also say that $\mu$ is \emph{absolutely continuous} with respect to $m$. It is then standard by the Radon--Nikodym theorem that  there is a measurable function $\rho$ on $X$ such that $\mu = \rho m$.  We will also make frequent use of Prokhorov's theorem, which in the case we need simply says that if $X$ is a compact space then so is $\cP(X)$ with the weak* topology (for a proof of a slightly more general version, cf.\ \cite[Theorem~6.5]{Billingsley:ConvProb}). We further remind the reader that convergence $\mu_{k}\to \mu$ in the weak* (also known as weak and narrow convergence) topology on $X$ means that $\int_{X} f \mu_{k} \to\int_{X} f \mu$ for all continuous functions $f$ on $X$. We write $C(X)$ for all continuous functions, and recall that by one version of the Riesz representation theorem, we can identify measures on $X$ with positive linear functionals on $C(X)$.\footnote{N.B.\ in our statements of the Radon-Nikodym theorem, Prokhorov's theorem and the Riesz representation theorem, we have always been assuming that $X$ is a compact metric space and we will always only consider Borel measures.}

We must emphasize that the majority of the material below is not in any way original work of the author, and should be considered as an expository work. We have tried to present the material in the Lott-Villani paper \cite{LottVillani:RicciViaTspt} in a simplified form. In particular, we have only discussed what they call $\infty$-Ricci curvature, and have not discussed convexity of general functionals on $\cP(X)$, just $\Ent(\cdot |m)$. We have also made use of the Sturm papers \cite{Sturm:MetMeas1,Sturm:MetMeas2}, which present much of the same information, but take a slightly different viewpoint, and have the advantage (from the point of view of our exposition) that the first paper only considers ``$\infty$-Ricci curvature,'' which is what we have focused on. We have expanded on the background material   on optimal transport given in the Lott-Villani paper, and following Villani's two books on optimal transport, \cite{Villani:TopOptTspt03,Villani:OptTspt} have given (mostly) self contained proofs of the basic facts about the theory. For metric geometry, we have mainly relied on Gromov's book \cite{Gromov:MetStruct} and Burago--Burago--Ivanov's textbook \cite{BBI:MetGeo}. The reader looking for further information about any of the material presented within will find the above references excellent places to start, and should hopefully be aided by the in-text citations as well. 
 
 Finally, we remark on the parts of the essay which we believe to be original. We believe our proof (but certainly not the result) of Proposition \ref{prop:G0-P0-iso}, showing that $(\cG_{0},d^{L^{2}})$ and $(\cP_{0},d^{W})$ are isometric is original. We have presented it this form, which relies on knowledge of optimal transport on compact Riemannian manifolds, rather than giving a more concrete proof, because we wanted to avoid discussing optimality conditions for transport plans. Secondly, the Gromov--Hausdorff approximations of $(\cG_{0},d^{L^{2}})$ in Proposition \ref{prop:gh-conv-to-g0} are probably original, but we note that similar ideas are examined by Andres and Von Renesse in \cite{AndresRenesse:PartApproxWass} in the context of Mosco- and $\Gamma$- convergence of Dirichlet forms. It seems that our proof (but not the result) of the existence of the entropic measure, $\QQ^{\beta}_{0}$ without relying on the Kolmogorov extension theorem, in Proposition \ref{prop:exist-entrop-meas} is original. We believe that the approximation of $\QQ^{\beta}_{0}$ by measures on finite dimensional spaces in the sense of (measured) Gromov--Hausdorff convergence has not been studied in this exact form before (cf. \cite{AndresRenesse:PartApproxWass} for related ideas, however). Finally, we believe that the results of Theorem \ref{theo:QQbeta-no-ric-geqK}, which says that $(\cG_{0},d^{L^{2}},\QQ^{\beta}_{0})$ does not have generalized Ricci lower bounds has not been shown before.

\section{Geodesic Metric Spaces} \label{sect:geo-met-sp}

For $(X,d)$ a metric space and $\gamma:[a,b]\to X$ a continuous curve, we define the length of $\gamma$ to be
\begin{equation} \label{eq:Ldefi} L(\gamma) = \sup_{a =t_0<t_1<\dots<t_{N+1} = b} \sum_{i=0}^N d(\gamma(t_i),\gamma(t_{i+1})).  \end{equation} 
Certainly this does not have to be finite. Curves $\gamma$ with $L(\gamma) < \infty$ are known as \emph{rectifiable curves}. As an example, recall that if $(X,d)$ is the metric space associated to a Riemannian manifold, then \eqref{eq:Ldefi} is the same as the usual definition of length for smooth curves, so in this case there are clearly plenty of rectifiable curves. Notice that by the triangle inequality applied to each term in the supremum in \eqref{eq:Ldefi}, we have that 
\[\addlabel{eq:dist-less-than-inf-length} d(\gamma(a),\gamma(b)) \leq L(\gamma) . \] 
We say that $(X,d)$ is a \emph{length space} if we can measure the distance between two points from knowledge of the lengths of curves between them, i.e.\ for $p,q \in X$
\[ d(p,q) = \inf_{\gamma} L(\gamma) \] where the infimum is taken over all paths from $p$ to $q$. We say that $(X,d)$ is a \emph{geodesic space} if the infimum is actually attained, that is for all $p,q\in X$ there is a curve $\gamma$ from $p$ to $q$ such that $L(\gamma) = d(p,q)$. We will call such a curve a \emph{geodesic}. It is important to note that our terminology slightly differs from the classical notion of a geodesic on a Riemannian manifold in that we require a geodesic to be globally length minimizing.

For points $p,q \in (X,d)$ a metric space, we call $r$ a \emph{midpoint} of $p$ and $q$ if $d(p,r) = d(q,r) = \frac 12 d(p,q)$. This provides an often convenient characterization of geodesic spaces.

\begin{lemm}\label{lemm:geo-sp-iff-midpts}
A metric space $(X,d)$ is a geodesic space if and only if midpoints exist.
\end{lemm}
\begin{proof}

It is not hard to show that the map $t\mapsto L(\gamma|_{[a,t]})$ is a continuous function of $t$ as long as $\gamma$ is rectifiable (see \cite{BBI:MetGeo} Proposition 2.3.4), so if $\gamma$ from $p$ to $q$ is a geodesic, then  there is $t$ with $L(\gamma|_{[a,t]}) = \frac 12 L(\gamma) = \frac 12 d(p,q)$. Thus, for this $t$, $d(p,\gamma(t)) = d(\gamma(t),q) = \frac 12 d(p,q)$, so $\gamma(t)$ is a ``midpoint'' between $p$ and $q$. 

Conversely supposing that midpoints exist, for points $p,q$, there is a midpoint $x_{1/2}$ between $p$ and $q$. Continuing this, there are midpoints $x_{1/4}$ of $p$ and $x_{1/2}$ and $x_{3/4}$ of $x_{1/2}$ and $q$. Thus, by induction, there are points $x_{k/2^n}$ such that $x_0 = p$, $x_1= q$ and $x_{(2k+1)/2^{n+1}}$ is the midpoint of $x_{k/2^n}$ and $x_{(k+1)/2^n}$. Furthermore induction, it is clear that this gives a map $\gamma$ from the dyadic rationals to $X$ such that \[d(\gamma(k/2^n), \gamma(k'/2^{n'})) = |k/2^n - k'/2^{n'}| d(p,q)\] and because of this, $\gamma$ extends by continuity to a continuous map $\gamma: [0,1] \to X$, with $L(\gamma) = d(p,q)$. 
\end{proof}
\begin{lemm}
Any geodesic $\gamma :[a,b] \to X$ can be continuously reparametrized so that it is a map $\gamma:[0,1] \to X$ and 
\[ d(\gamma(t),\gamma(t')) = |t-t'| d(\gamma(0),\gamma(1)).\]
\end{lemm}
Further details and a proof of this can be found in \cite[Proposition 2.5.9]{BBI:MetGeo}. From now on, we will assume that all geodesics are so parametrized.

We will often use the term \emph{induced length metric} or \emph{intrinsic length metric} in the following situation. Suppose that we have a compact geodesic space $(X,d)$ and a (path connected) closed subset $Y\subset X$. We can define a metric $\tilde d$ on $Y$ by 
\[ \tilde d(y,y') = \inf_{\gamma \subset Y} L(\gamma), \] where the infimum is taken over all paths $\gamma$ from $y$ to $y'$ which are contained entirely inside $Y$ but $L(\gamma)$ is the length as measured in $X$ with the metric $d$. It is possible to show that under our assumptions on $X$ and $Y$, $(Y,\tilde d)$ is a geodesic space; cf.\ \cite[Chapter 2]{BBI:MetGeo}. As such, we call $\tilde d$ the \emph{intrinsic/induced metric}. 

\section{Optimal Transport and Wasserstein Distance} \label{sect:opt-trans-was-dist}
Let $(X,d)$ be a compact geodesic space. We will denote the space of probability measures on $X$ by $\cP(X)$. For $k=1,2$, we define $\proj_{k}:X\times X \to X$ to be projection onto the $k$-th factor. For two probability measures $\mu,\nu \in \cP(X)$, we define the set of \emph{admissible transport plans} to be 
\begin{equation} \Pi(\mu,\nu) =\{ \pi \in \cP(X\times X) : (\proj_1)_*\pi = \mu,(\proj_2)_*\pi = \nu\}\end{equation}
where for $k=1,2$, $(\proj_{k})_{*} \pi \in \cP(X)$ is the $k$-th marginal, $(\proj_{k})_{*}(A) = \pi[(\proj_{k}^{-1}(A)]$. This is certainly nonempty, because the measure $\mu\otimes \nu$ which is defined by \[ \int_{X\times X} f(x,y) d\mu\otimes \nu(x,y) : = \int _{X}\int_{ X} f(x,y) d\mu(x)  d\nu(y)\] for continuous $f\in C(X\times X)$, is clearly in $\Pi(\mu,\nu)$. Loosely speaking, a transport plan is a proposal for how to move the mass of $\mu$ around so as to assemble the distribution of mass prescribed by $\nu$. Given the above definition, we define the $2$-\emph{Wasserstein distance} between $\mu$ and $\nu$ to be 
\begin{equation}\label{eq:wass-dist} d^{W}(\mu,\nu)^{2} := \inf_{\pi\in\Pi(\mu,\nu)}\int_{X\times X} d(x,y)^{2} d\pi(x,y). \end{equation}
\begin{lemm}\label{lemm:exist-opt-trans}
The infimum in \eqref{eq:wass-dist} is always achieved.
\end{lemm}
\begin{proof}
Letting $\pi_{k} \in \Pi(\mu,\nu)$ be a minimizing sequence,
\[ \lim_{k\to\infty} \int_{X\times X} d(x,y)^{2}d\pi_{k} = d^{W}(\mu,\nu)^{2}.\]
Because $X\times X$ is compact, by Prokhorov's theorem, extracting a subsequence, we can assume that there is some $\pi \in \cP(X\times X)$ such that $\pi_{k}\to \pi$ in the weak* topology. Because $d(x,y)^{2}$ is continuous on $X\times X$, we thus have that 
\[ \int_{X\times X}d(x,y)^{2}d\pi = d^{W}(\mu,\nu)^{2}.\]
It remains to show that $\pi \in \Pi(\mu,\nu)$. For continuous functions $f,g\in C(X)$ 
\begin{align*}
 \int_{X\times X}( f(x) + g(y)) d\pi(x,y) &  = \lim_{k\to\infty} \int_{X\times X} (f(x) + g(y)) d\pi_{k}(x,y)\\
 & = \lim_{k\to \infty} \left( \int_{X} f(x) d\mu (x)+ \int_{X} g(y) d\nu(y)\right)\\
 & =  \int_{X} f(x) d\mu(x) + \int_{X} g(y) d\nu(y)
 \end{align*}
 Taking $g$ zero, because $f$ was arbitrary, this shows that $(\proj_{1})_{*}\pi = \mu$ and similarly setting $f$ to zero gives $(\proj_{2})_{*}\pi = \nu$, so we have that $\pi \in \Pi(\mu,\nu)$.
\end{proof}

We call such a minimizer an \emph{optimal transport plan}. We remark that for a fixed $\mu,\nu\in \cP(X)$ it could be true that there is more than one optimal transport plan.

\subsection{Metric and Topological Properties of Wasserstein Distance} In the following proof, we give a simplified version of the exposition in \cite[Chapter 7]{Villani:TopOptTspt03}. 
\begin{prop}
The Wasserstein distance, $d^{W}$, defined by \eqref{eq:wass-dist} is a metric on $\cP(X)$. 
\end{prop}

\begin{proof}
Clearly $d^{W}$ is symmetric, nonnegative and finite on $\cP(X)$. Furthermore, $d^{W}(\mu,\mu) = 0$ for all $\mu \in \cP(X)$. If $d^{W}(\mu,\nu) = 0$, then by Lemma \ref{lemm:exist-opt-trans}, there exists an optimal plan $\pi \in \Pi(\mu,\nu)$ with 
\begin{equation*}\label{eq:opt-plan-zero-int}
\int_{X\times X} d(x,y)^{2} d\pi(x,y) = 0.
\end{equation*}
This shows that $\pi$ must be concentrated on the diagonal $\Delta_{X} \subset X\times X$, and thus is invariant under $x\leftrightarrow y$. Thus, for $u\in C(X)$
\[
\int_{X} u(x) d\mu(x) = \int_{X\times X} u(x) d\pi(x,y) = \int_{X\times X} u(y) d\pi(x,y) = \int_{X} u(y) d\nu(y).
\]
which shows thate $\mu = \nu$. Thus, it remains to prove the triangle inequality. This requires 
\begin{lemm}[Gluing Lemma]\label{lemm:gluing}
For probability measures $\mu_{1},\mu_{2},\mu_{3}\in \cP(X)$ and associated transport plans $\pi_{12}\in \Pi(\mu_{1},\mu_{2}), \pi_{23} \in \Pi(\mu_{2},\mu_{3})$, there exists a probability measure $\pi \in \cP(X\times X\times X)$ with marginal $\pi_{12}$ on the first two factors and $\pi_{23}$ on the second two factors. 
\end{lemm}

We will prove this after we show how it establishes the triangle inequality. For $\mu_{1},\mu_{2},\mu_{3} \in \cP(X)$, let $\pi_{12}\in \Pi(\mu_{1},\mu_{2})$ and $\pi_{23}\in \Pi(\mu_{2},\mu_{3})$ be optimal transport plans. The gluing lemma says that there is $\pi \in \cP(X\times X\times X)$ with marginals $\pi_{12}$ and $\pi_{23}$. We let $\pi_{13}$ be the marginal on the first and third factors of $X$. It is clear that this is a transport plan between $\mu_{1}$ and $\mu_{3}$ (it is not necessarily optimal), giving
\begin{align*}
d^{W}(\mu_{1},\mu_{3})^{2} & \leq \int_{X\times X} d(x_{1},x_{3})^{2} d\pi_{13}(x_{1},x_{3}) \\
& =\int_{X\times X\times X} d(x_{1},x_{3})^{2} d\pi(x_{1},x_{2},x_{3})\\
& \leq \int_{X\times X\times X} \left[ d(x_{1},x_{2}) + d(x_{2},x_{3})\right]^{2} d\pi(x_{1},x_{2},x_{3})\\
& = \int_{X\times X\times X} \left[d(x_{1},x_{2})^{2}    +  d(x_{2},x_{3})^{2}  + 2d(x_{1},x_{2}) d(x_{2},x_{3}) \right] d\pi(x_{1},x_{2},x_{3})\\
& = \int_{X\times X} d(x_{1},x_{2})^{2} d\pi_{12}(x_{1},x_{2})  +  \int_{X\times X} d(x_{2},x_{3})^{2} d\pi_{23}(x_{2},x_{3})  \\
 & \qquad \qquad+  2 \int_{X\times X\times X}  d(x_{1},x_{2}) d(x_{2},x_{3}) d\pi(x_{1},x_{2},x_{3}) \\
 & = d^{W}(\mu_{1},\mu_{2})^{2} + d^{W}(\mu_{2},\mu_{3})^{2} +  2 \int_{X\times X\times X}  d(x_{1},x_{2}) d(x_{2},x_{3}) d\pi(x_{1},x_{2},x_{3})\\
  &  \leq 2 \left( \int_{X\times X\times X}  d(x_{1},x_{2})^{2 }d\pi(x_{1},x_{2},x_{3})\right)^{1/2} \left( \int_{X\times X\times X}d(x_{2},x_{3})^{2} d\pi(x_{1},x_{2},x_{3})\right)^{1/2}\\
 &    \qquad \qquad + d^{W}(\mu_{1},\mu_{2})^{2} + d^{W}(\mu_{2},\mu_{3})^{2}\\
 & = d^{W}(\mu_{1},\mu_{2})^{2} + 2 d^{W}(\mu_{1},\mu_{2})d^{W}(\mu_{2},\mu_{3}) + d(\mu_{2},\mu_{3})^{2}\\
 & = \left[ d^{W}(\mu_{1},\mu_{2}) + d^{W}(\mu_{2},\mu_{3})\right]^{2}
\end{align*}
proving the triangle inequality. Thus, it remains to prove the gluing lemma (Lemma \ref{lemm:gluing}). 
\begin{proof}[Proof of the Gluing Lemma]
Let $V \subset C(X\times X\times X)$ be the vector subspace
\[ V: = \{ \varphi_{12}(x_{1},x_{2}) + \varphi_{23}(x_{2},x_{3}) : \varphi_{12},\varphi_{23} \in C(X\times X) \}\] and define a functional $G: V\to \RR$ by 
\[ G(\varphi_{12} + \varphi_{23} ) : = \int_{X\times X} \varphi_{12}(x_{1},x_{2})d\pi_{12}(x_{1},x_{2}) + \int_{X\times X} \varphi_{23} (x_{2},x_{3}) d\pi_{23}(x_{2},x_{3}).\] It is not a priori clear that this is well defined. If $\varphi_{12} + \varphi_{23} = \tilde \varphi_{12} + \tilde \varphi_{23}$ as functions, then 
\[ \varphi_{12}(x_{1},x_{2}) -\tilde \varphi_{12}(x_{1},x_{2}) = \tilde\varphi_{23}(x_{2},x_{3}) - \varphi_{23}(x_{2},x_{3}),\] 
which shows that both sides are functions of $x_{2}$ only. Thus
\begin{align*}
 \int_{X\times X} [\varphi_{12}(x_{1},x_{2}) - \tilde\varphi_{12}(x_{1},x_{2}) ] d\pi_{12}(x_{1},x_{2}) & = \int_{X} [ \varphi_{12}(x_{1},x_{2}) - \tilde\varphi_{12}(x_{1},x_{2})] d\mu_{2}(x_{2})\\
 &= \int_{X} [ \tilde \varphi_{23}(x_{2},x_{3}) - \varphi_{23}(x_{2},x_{3})] d\mu_{2}(x_{2})\\ 
 &= \int_{X\times X} [ \tilde \varphi_{23}(x_{2},x_{3}) - \varphi_{23}(x_{2},x_{3})] d\pi_{23}(x_{2},x_{3})
 \end{align*}
 and rearranging the terms shows that 
 \[ G(\varphi_{12} + \varphi_{23}) = G(\tilde \varphi_{12} + \tilde\varphi_{23}) \] 
 so $V$ is in fact well defined. Furthermore, $G$ is clearly bounded and linear. Thus Hahn-Banach tells us that there exists an extension to a positive functional\footnote{The usual version of Hahn--Banach does not tell us that the extension need be positive. The reader could try to modify the proof of Hahn--Banach to show this (some caution must be taken here: the fact that there is always an element in $V$ dominating any element in $C(X\times X\times X)$ (e.g.\ a constant function) should be used in the proof). Alternatively, cf.\ \cite[ Section 8.11]{AliprantisBorder:InfDimAnalysis}. } $\hat G: C(X\times X\times X) \to \RR$, and by the Riesz representation theorem, $\hat G$ is represented by some (Borel) measure $\pi \in \cP(X\times X\times X)$. Because 
 \begin{align*}
  \int_{X\times X\times X} [\varphi_{12}(x_{1},x_{2}) + \varphi_{23}(x_{2},x_{3})] d\pi(x_{1},x_{2},x_{3}) & = \hat G(\varphi_{12}+ \varphi_{23}) \\
  & = G(\varphi_{12}+ \varphi_{23})\\
  & = \int_{X\times X} \varphi_{12}(x_{1},x_{2})d\pi_{12}(x_{1},x_{2})\\
  & \qquad  + \int_{X\times X} \varphi_{23} (x_{2},x_{3}) d\pi_{23}(x_{2},x_{3}) 
  \end{align*}
  it is clear that $\pi$ has marginals $\pi_{12}$ and $\pi_{23}$, as desired. 
\end{proof}
This completes the proof that $d^{W}$ satisfies the triangle inequality, and thus the proof that it is a metric on $\cP(X)$. 
\end{proof}

The next proof is a simplification of the proof of \cite[Theorem 6.9]{Villani:OptTspt}, and we will rely on a result characterizing the support of optimal plans known as $c$-monotonicity, references for which we give below. 
\begin{prop}
Equipped with the Wasserstein distance, $(\cP(X),d^{W})$ is a compact metric space with the same topology as that on $\cP(X)$ defined by weak* convergence. 
\end{prop}
\begin{proof}
It is enough to show that $\mu_{k}\to \mu$ in the weak* topology if and only if we have that $d^{W}(\mu_{k},\mu) \to 0$, because then we have compactness by Prokhorov's theorem. To begin with, suppose that we have the convergence $\mu_{k}\to \mu$ in the weak* topology. Let $\pi_{k} \in \Pi(\mu_{k},\mu)$ be optimal transport plans. By compactness of $\cP(X\times X)$, we may extract a subsequence $\pi_{k'}$ such that $\pi_{k'} \to \pi$ in the weak* sense. First, we claim that $\pi$ is a transport plan, i.e.\ $\pi \in \Pi(\mu,\mu)$. To see this, notice that 
\begin{align*}
\int_{X\times X} [\varphi(x) + \psi(y)] d\pi (x,y)& = \lim_{k'\to\infty} \int_{X\times X} [\varphi(x) + \psi(y)] d\pi_{k'}(x,y) \\
& = \lim_{k'\to\infty} \int_{X} \varphi(x) d\mu_{k'}(x) + \int_{X}\psi(y) d\mu(y)\\
& =\int_{X} \varphi(x) d\mu(x) + \int_{X} \psi(y) d\mu(y),
\end{align*}
showing that $\pi \in \Pi(\mu,\mu)$. We now rely on \cite[Theorem 5.20]{Villani:OptTspt}, which says that the weak* limit of optimal plans is optimal (which follows from the discussion of $c$-monotonicity of supports of optimal plans, which can be found in various places, including the cited theorem, or \cite[Theorem 6.1.4]{AGS:GradFlow}). Thus,
\begin{equation*}
\lim_{k'\to \infty} d^{W}(\mu_{k'},\mu)^{2}  = \lim_{k'\to\infty } \int_{X\times X} d(x,y)^{2} d\pi_{k'}(x,y) = \int_{X\times X} d(x,y)^{2} d\pi(x,y)  = 0,
\end{equation*}
where the last equality follows from the fact that $\pi \in \Pi(\mu,\mu)$ is optimal, so its cost must be zero. 

Conversely, suppose that $d^{W}(\mu_{k},\mu)\to 0$. Because $\cP(X)$ is compact, extracting a subsequence there is some $\tilde \mu \in\cP(X)$ with $\mu_{k'}\to \tilde\mu$ in the weak* topology. The above part of the proof shows that $d^{W}(\mu_{k'}, \tilde \mu) \to 0$. By uniqueness of limits in metric spaces, this shows that $\mu = \tilde \mu$, and because we can repeat this argument for any subsequence, we have that $\mu_{k}\to \mu$ in the weak* topology. 
\end{proof}

\subsection{Dynamical Transport Plans and Geodesics} From now on, we will assume that $(X,d)$ is a (compact) geodesic space. We define $\Lip([0,1];X)$ to be Lipschitz continuous maps from $[0,1] \to X$ with the uniform topology, and let \[\Gamma \subset \Lip([0,1];X)\] be the (compact by Arzel\'a-Ascoli) subset of unit speed parametrized geodesics. For $t \in [0,1]$, the evaluation maps $e_{t}:\Gamma \to X$, given by $e_{t}:\gamma \mapsto \gamma(t)$ are clearly continuous. We call a Borel measure $\Theta$ on $\Gamma$ a \emph{dynamical transport plan} between $\mu,\nu \in \cP(X)$ if $(e_{0},e_{1})_{*} \Theta \in \Pi(\mu,\nu)$. Intuitively, $\Theta$ is a choice of which geodesic to transport the mass along, instead of just a initial and final endpoint. If $(e_{0},e_{1})_{*}\Theta$ is an optimal transport plan then, reasonably enough, we call $\Theta$ an \emph{optimal dynamical transport plan}. Again, even for a fixed optimal transport plan $\pi$, there might be multiple dynamic plans $\Theta$ with $(e_{0},e_{1})_{*}\Theta = \pi$. For example, if \[\addlabel{eq:defnX-non-unique-dynamic-plans}X = ([-2,-1]\times\{0\}) \cup \{(x,y) : x^{2} + y^{2} = 1\} \cup ([1,2]\times\{0\}) \subset \RR^{2}\] 
is equipped with the ``induced length metric,'' (i.e. the distance between two points is the minimum of the lengths of curves between the two points, where length is measured in the usual way for a curve in $\RR^{2}$), and $\mu = \delta_{(-2,0)}$ and $\nu = \delta_{(2,0)}$, it is clear that $\pi = \delta_{(-2,0)\times (2,0)}$ is the unique optimal transport plan, but there are two geodesics between $(-2,0)$ and $(2,0)$, say $\gamma_{1}$ which goes along the upper hemisphere of the circle and $\gamma_{2}$ which goes along the lower hemisphere, as in Figure \ref{fig:non-unique-dynamical-tspt}. Then, for any $s\in[0,1]$, defining \[ \Theta_{s} : = s \delta_{\gamma_{1}} + (1-s) \delta_{\gamma_{2}}\] we clearly have $(e_{0},e_{1})_{*}\Theta_{s}=\pi$. 
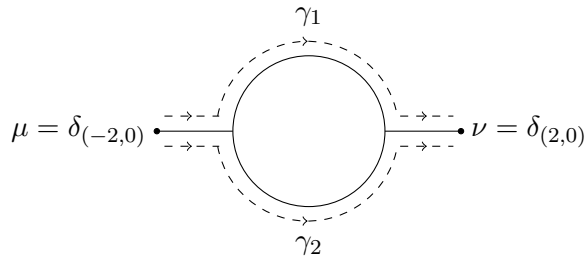
\begin{figure}[hb]
\usetikzlibrary{decorations.markings}
\begin{tikzpicture}
\begin{scope}[decoration={
    markings,
    mark=at position 0.5 with {\arrow{>}}}]
\draw (0,0) circle (1);
\draw (-2,0) -- (-1,0);
\draw (1,0) -- (2,0);
\filldraw (-2,0) circle (1pt) node [left] {$\mu = \delta_{(-2,0)}$};
\filldraw (2,0) circle (1pt) node [right] {$\nu = \delta_{(2,0)}$};
\draw [dashed,postaction={decorate}] (-1.2,.2) arc (170:10:1.2) node [midway,above=36] {$\gamma_{1}$};
\draw [dashed,postaction={decorate}] (1.2,.2) -- (1.9,.2);
\draw[dashed,postaction={decorate}]  (-1.9,.2) -- (-1.18,.2);
\draw [dashed,postaction={decorate}] (-1.2,-.2) arc (190:350:1.2) node [midway,below=36] {$\gamma_{2}$};
\draw [dashed,postaction={decorate}] (1.2,-.2) -- (1.9,-.2);
\draw [dashed,postaction={decorate}] (-1.9,-.2) -- (-1.18,-.2);
\end{scope}
\end{tikzpicture}
\caption{An example of a unique optimal transport plan on the space $X$ defined in \eqref{eq:defnX-non-unique-dynamic-plans} (transporting all of the mass at $(-2,0)$ to $(2,0)$ is the \emph{only} transport plan, and is thus the optimal one) with multiple associated dynamical transport plans.}
\label{fig:non-unique-dynamical-tspt}
\end{figure}

\begin{prop}\label{prop:dynamic-trans-plan}
Assuming, as above, that $(X,d)$ is a compact geodesic space, for any $\mu,\nu \in \cP(X)$, and any optimal transport plan $\pi$, there exists an optimal dynamical transport plan between $\mu$ and $\nu$, $\Theta$ with $(e_{0},e_{1})_{*}\Theta = \pi$. 
\end{prop}

\begin{proof}
If $(e_{0},e_{1})$ admits a Borel right inverse $S: X\times X \to \Gamma$ with \[(e_{0},e_{1}) \circ S = \Id_{X\times X},\] then we are finished, because this implies that
\[ \pi = \Id_{*} \pi = (e_{0},e_{1})_{*}S_{*} \pi\] so we can define $\Theta = S_{*}\pi$. The map $(e_{0},e_{1}): \Gamma \to X\times X$ is certainly surjective, as $X$ is a geodesic space, so a right inverse exists by the axiom of choice. However, there is no reason that it should be measurable, and we must turn to a technical result known as a measurable selection theorem, one version of which says that because $X\times X$ and $\Gamma$ are separable, compact, and metrizable, there is a Borel measurable right inverse to $(e_{0},e_{1})$. In \cite[Corollary A.6]{Zimmer:ErgoSemiSimp}, this version of measurable selection is shown to follow from \cite[Proposition 3.1]{Kallman:QuotCountSep3}. 
\end{proof}

\begin{coro}
If $(X,d)$ is a geodesic space, then so is $(\cP(X),d^{W})$.
\end{coro}
\begin{proof}
For $\mu,\nu \in \cP(X)$, by Lemma \ref{lemm:exist-opt-trans}, there exists an optimal transport plan $\pi \in \Pi(\mu,\nu)$, and by Proposition \ref{prop:dynamic-trans-plan} there is a corresponding optimal dynamical transport plan $\Theta$. We claim that $\mu_{t} : = (e_{t})_{*}\Theta$, where $e_{t}:\Gamma \to X, \gamma \mapsto \gamma(t)$ is a geodesic between $\mu = \mu_{0}$ and $\nu=\mu_{1}$. To see this, we calculate
\begin{align*}
d^{W}(\mu_{t},\mu_{t'})^{2} & \leq \int_{X\times X} d(x,y)^{2} d[(e_{t},e_{t'})_{*}\Theta](x,y)\\
& = \int_{\Gamma} d(\gamma(t),\gamma(t'))^{2} d\Theta(\gamma)\\
& = |t-t'|^{2} \int_{\Gamma} d(\gamma(0),\gamma(1))^{2} d\Theta(\gamma)\\
& = |t-t'|^{2 }\int_{X\times X} d(x,y)^{2} d[(e_{0},e_{1})_{*}\Theta](x,y)\\
& = |t-t'|^{2}\int_{X\times X} d(x,y) d\pi(x,y)\\
& = |t-t'|^{2} d^{W}(\mu_{0},\mu_{1})^{2}.
\end{align*}
This implies that $\mu_{t}$ is continuous and moreover that $L(\mu_{t}) \leq d^{W}(\mu_{0},\mu_{1})$. Thus by \eqref{eq:dist-less-than-inf-length}, $L(\mu_{t}) = d^{W}(\mu_{0},\mu_{1})$, so $\mu_{t}$ is a geodesic between $\mu$ and $\nu$, because we must have equality in the first line above. 
\end{proof}

\begin{prop}
The map $i : X \to \cP(X)$ defined by $x\mapsto \delta_{x}$ is an isometric embedding.
\end{prop}
\begin{proof}
For $x,y\in X$, notice that $\Pi(\delta_{x},\delta_{y}) = \{ \delta_{(x,y)} \}$, so 
\[ d^{W}(i(x),i(y))^{2} = \int_{X\times X} d(\tilde x,\tilde y)^{2} \delta_{(x,y)}(\tilde x,\tilde y) = d(x,y)^{2},\] showing that $i$ is an isometric embedding, as claimed. 
\end{proof}

Notice that $i(X)$ need not be a totally geodesic subset of $\cP(X)$. Certainly for $x,y \in X$, if $\gamma(t)$ is a geodesic between $x$ and $y$, it is easy to see that $\delta_{\gamma(t)}$ is a geodesic between $\delta_{x}$ and $\delta_{y}$ that lies entirely in $i(X)$. However, if there is more than one geodesic between $x$ and $y$, an (infinite) convex combination of these geodesics will also be a geodesic, which is disjoint from $i(X)$ except for at its endpoints. For example, on $(S^{2},d_{S^{2}})$, if we take $x$ to be the north pole and $y$ the south pole, then clearly if $\gamma(t) : [0,1] \to S^{2}$ is the constant speed geodesic along any great circle from $x$ to $y$ then $\delta_{\gamma(t)}$ is a geodesic in $\cP(S^{2})$ from $\delta_{x}$ to $\delta_{y}$ and lies entirely in $i(S^{2})$, but for example in coordinates $(\phi,\theta)$ on the sphere, where $\phi$ is the angle from the $z$-axis and $\theta$ is the azimuthal angle of the projection of the point onto the $xy$-plane in polar coordinates, then if we define \[A_{t} : = \{(\pi t,\theta): \theta \in[0,2\pi)\}\] to be the circle at $\phi=\pi t$, then if we write $\cH^{1}$ for the $1$-dimensional Hausdorff measure on $S^{2}$, it is not hard to see that 
\[ \mu_{t} := 
\begin{cases}  
(2 \pi \sin \pi t)^{-1} \cH^{1}|_{A_{t}}  & \text{ for } t \in (0,1)\\
\delta_{x} & \text{ for } t = 0\\
\delta_{y} & \text{ for } t = 1
\end{cases}  \]
is a geodesic between $\delta_{x}$ and $\delta_{y}$ in $\cP(S^{2})$ which is most certainly not in $i(S^{2})$. This is illustrated in Figure \ref{fig:nonuniquedispinterpol-sphere}.
\begin{figure}[h]
\centering
\begin{tikzpicture}[scale = 2.8]
\draw (0,0) circle (1);
\begin{scope}[line width = 1.4]
	\begin{scope}
		\clip (-1,-.5) rectangle (1,1);
		\draw [dashed] (0,-.5) ellipse (.86 and .2);
	\end{scope}
	\begin{scope}
		\clip (-1,-.5) rectangle (1,-1);
		\draw (0,-.5) ellipse (.86 and .2);
	\end{scope}
\end{scope}

\draw (.5,-.55) node {$A_{t}$};

\draw (0,1) to (0,-1);
\draw (0,0) to (.85,-.51);
\draw (0,.1) arc (90:-30:.1) node [above = 12, right = 1] {$\phi = \pi t$};

\filldraw (0,1) circle (.5pt) node [above=2] {${x}$}; 
\filldraw (0,-1) circle (.5pt) node [below=2] {${y}$}; 

\end{tikzpicture}
\caption{This figure shows that the image of $X$ in $\cP(X)$ under the map $i: x \mapsto \delta_{x}$ is not totally geodesic. There is a geodesic between $\delta_{x}$ and $\delta_{y}$ which at time $t$ is supported on $A_{t}$, as described in the text, and is thus clearly not in $i(X)$.}
\label{fig:nonuniquedispinterpol-sphere}
\end{figure}
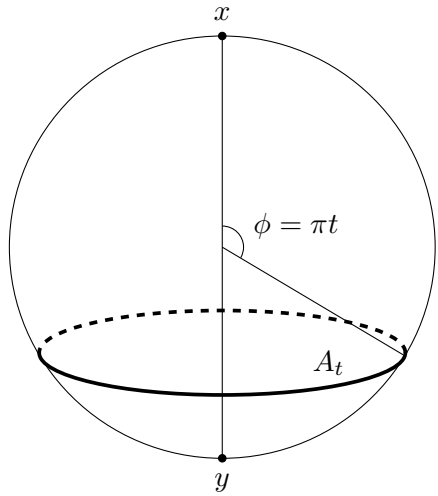

\section{Metric Structure of $(\cP_{0},d^{W})$} \label{sect:met-struct-cP0}
We now specialize to studying $(\cP([0,1]),d^{W})$, which will often write $(\cP_{0},d^{W})$.

\begin{prop}\label{prop:G0-P0-iso}
Letting $\cG_{0} \subset L^{2}([0,1])$ be the subset of right continuous, nondecreasing maps $g:[0,1]\to [0,1]$, and $d^{L^{2}}$ the metric induced on $\cG_{0}$ from $L^{2}$, then the map \[ \addlabel{eq:psi-defi} \Psi : (\cG_{0},d^{L^{2}}) \to (\cP_{0},d^{W})\] 
\[ g \mapsto g_{*}\Leb \] is an isometry. The inverse $\Psi^{-1}$ is given by 
\[ \addlabel{eq:psi-inverse} \Psi^{-1} : \mu \mapsto g_{\mu}\] where $g_{\mu}$ is the inverse distribution function defined
\[ \addlabel{eq:inverse-dist-defi} g_{\mu}(t) : = \inf\{s \in[0,1]: \mu([0,s] )> t\}\] with the convention that $\inf \emptyset :=1$.
\end{prop} 

\begin{proof}
It is clear that $\Psi$ as defined is a bijection with stated inverse. Thus, it remains to show that it is an isometry. For $f,g \in \cG_{0}$, notice that $(f,g)_{*}\Leb_{[0,1]} \in \Pi(\Psi(f),\Psi(g))$, and that 
\[ \int_{[0,1]\times[0,1]} |x-y|^{2} d[(f,g)_{*}\Leb](x,y) =\int_{0}^{1} |f(s) -g(s)|^{2} d\Leb(s) .\] Thus\[ d^{W}(\Psi(f),\Psi(g)) \leq d^{L^{2}}(f,g).\]

There are various ways to show the opposite inequality. One way is to use the Knott--Smith optimality condition, which says that a plan $\pi$ between two measures on $\RR^{n}$ is optimal if and only if there is a convex, lower semi-continuous function $\varphi: \RR^{n}\to \RR$ with $\supp(\pi) \subset \Graph(\underline{\partial} \varphi)$ (where $\underline{\partial}\varphi$ is the subdifferential of $\varphi$). This approach is taken in the proof of \cite[Theorem 2.18]{Villani:TopOptTspt03}(which is stated for $\cP(\RR)$, but the same proof works without change for $\cP_{0}$). Alternatively, a more probabilistic proof, using the Hoeffding--Fr\'echet theorem which relates the cumulative distribution function of a measure on $\RR^{2}$ to the cumulative distributions of its marginals, is given in \cite{Rachev:MassTsptProblems1}. However, in order to avoid distracting ourselves with these notions, we will rely on Theorem \ref{theo:polar-fact-riem-mflds}. All that we will need is that on a compact manifold $(M,g)$, if $\mu,\nu \in \cP(M)$ with $\mu,\nu \ll m = ({\vol(M)})^{-1}{\vol_{M}}$ then there is an optimal transport plan $\pi$ of the form $(\Id,F)_{*}\mu$ for some measurable $F:M\to M$. Such a plan is called a \emph{Monge transport plan}. Furthermore, we use that $F(x) = \exp_{x}(-\nabla\phi)$ for a $\frac{d^{2}}{2}$-concave function $\phi$, but all we will really need is the special case of the circle $M= S^{1}$, where this implies that $F$ is ``nondecreasing,'' in the sense that it lifts to a nondecreasing function $\RR\to \RR$. 

Given this fact, we will finish the proof by isometrically embedding $[0,1]$ into $(S^{1},g)$, the standard circle of circumference $3$, $\RR/3\ZZ$. For $f,g \in \cG_{0}$ with $f,g$ continuous and strictly increasing, defining $\mu= \Psi(f),\nu=\Psi(g) \in \cP_{0}$, it is clear that our assumptions imply that $\mu,\nu \ll \Leb$ and $\supp(\mu) = \supp(\nu) = [0,1]$. Furthermore, they give rise to measures $\tilde \mu,\tilde \nu \in \cP(S^{1})$ supported on the image of $[0,1]$, and that $\tilde\mu,\tilde\nu \ll m$, so by Theorem \ref{theo:polar-fact-riem-mflds}, there is an optimal Monge plan $\tilde \pi = (\Id,\tilde F)_{*}\tilde \mu$ between $\tilde\mu$ and $\tilde\nu$. This implies that $\tilde F_{*}\tilde \mu = \tilde \nu$, so $\tilde F(\supp ( \tilde \mu)) \subset \supp( \tilde \nu)$, so defining a new function $F: [0,1] \to [0,1]$ by  $F := \tilde F|_{[0,1]}$, we claim that $\pi:= (\Id,F)_{*} \mu$, this is an optimal Monge transport plan. It is clearly a transport plan, and if it were not optimal, by considering  $[0,1]\times[0,1]$ as a subset of $S^{1}\times S^{1}$ we could use an optimal plan to contradict $\tilde\pi$'s optimality. Because $\tilde F$ is nondecreasing, it is clear that $F$ is nondecreasing. Furthermore, because $\nu \ll \Leb$, we may assume without loss of generality that $F$ is right continuous, because this is only a modification on a set of $\mu$-measure zero. 

Now, combining all of this, we see that $g_{*} \Leb =  F_{*}f_{*}\Leb = (F\circ f)_{*}\Leb$, and because of our above observations, $F\circ f \in \cG_{0}$, so bijectivity of $\Psi$ gives that $g = F\circ f$. Thus, we see that \[\pi = (\Id, F)_{*} \mu = (Id,F)_{*} f_{*} \Leb = (f, F\circ f)_{*} \Leb = (f,g)_{*}\Leb\] is optimal, so for $f,g\in \cG_{0}$ continuous and strictly increasing
\begin{equation}
d^{W}(\Psi(f),\Psi(g)) = d^{L^{2}}(f,g).
\end{equation}
To finish the proof, it remains to notice that continuous, strictly increasing functions are dense in $(\cG_{0},d^{L^{2}})$. To show this, for example, for an arbitrary $h \in \cG_{0}$, define $\tilde h_{n}$ to be the linear interpolation between the points $\{(i/n, f(i/n)\}_{i=0}^{n}$, and notice that 
\begin{align*} \int_{0}^{1} |h(x) - \tilde h_{n}(x)|^{2} dx & = \sum_{i=0}^{n-1} \int_{i/n}^{(i+1)/n} | h(x) - \tilde h_{n}(x) |^{2} dx\\
&\leq \frac {1}{n}\sum_{i=0}^{n-1}\left | h\left( \frac{i+1}{n} \right)- h\left( \frac i n\right) \right|^{2}  \\
& \leq \frac {1}{n}\sum_{i=0}^{n-1}\left ( h\left( \frac{i+1}{n} \right)- h\left( \frac i n\right) \right)\\
& = \frac 1 n \left( h (1) - h(0) \right) \\
& \leq \frac 1n.
\end{align*} 
We're not quite finished; $\tilde h_{n}$ is certainly nondecreasing, but might not be strictly increasing. To fix this, define 
\[ h_{n}(x) := \frac{n-1}{n} \tilde h_{n}(x) +\frac x n \]
and notice that 
\[ \Vert h_{n} - \tilde h_{n} \Vert_{L^{2}}  \leq \frac 1 n \left( \Vert \tilde h_{n} \|_{L^{2}} +  \| x\|_{L^{2}} \right) \to 0\]
so, we have that $h_{n}$ is a sequence of continuous, strictly increasing functions converging to $h$ in $\cG_{0}$, finishing the proof. 
\end{proof}

We must remark that as far as we know, Lemma \ref{prop:G0-P0-iso} is highly dependent on the one dimensionality of $[0,1]$. A similar statement holds for  the circle $\cS^{1}$, but it is not clear how to embed $\cP(X)$ into a Hilbert space in some reasonable manner. In fact, as we will see in the next lemma, the image of the embedding is totally convex, allowing us to give a very concrete description of the geodesics in $\cP_{0}$.

\begin{lemm}\label{lemm-g0-geo}
The space $\cG_{0}$ is a totally convex subset of $L^{2}([0,1])$ (that is, any geodesic between two elements in $\cG_{0}$ lies entirely in $\cG_{0}$). In fact, for $f,g\in \cG_{0}$, the unique geodesic between them is given by the linear combination
\[ \gamma(t) := (1-t) f + t g .\]
\end{lemm}
\begin{proof}
For $f,g \in L^2$, we claim that the convex combination $\gamma(t) = tf + (1-t)g$ is a $L^{2}$-geodesic from $f$ to $g$. Notice that
\[ \Vert \gamma(t) - \gamma(s) \Vert_{L^2} = \Vert (t-s) f + (s-t) g\Vert_{L^2} = |t-s| \Vert f-g\Vert_{L^2}\] so it is thus a geodesic. Now we claim that these are the unique geodesics. Assume that for $f\neq g$, there is a $h \in L^2$ which is not in $\gamma([0,1])$. This gives 
\begin{align*} (\Vert f-h\Vert_{L^2} + \Vert h-g\Vert_{L^2})^2 
& = \Vert f \Vert_{L^2}^2 + \Vert g \Vert_{L^2}^2 + 2 \Vert h \Vert_{L^2}^2 - 2\bangle{f,h} - 2\bangle{g,h}\\ & \qquad + 2 \Vert f-h\Vert_{L^2} \Vert g-h\Vert_{L^2}\\
& \geq \Vert f \Vert_{L^2}^2 + \Vert g \Vert_{L^2}^2 + 2 \Vert h \Vert_{L^2}^2 - 2\bangle{f+g,h} + 2\bangle{ f-h,  h-g}\\
& = \Vert f \Vert_{L^2}^2 + \Vert g \Vert_{L^2}^2 + 2 \Vert h \Vert_{L^2}^2 - 2\bangle{f+g,h} - 2\bangle{ f,g} + 2\bangle{ h,g} \\
&\qquad + 2\bangle{ f,h} - 2 \Vert h \Vert_{L^2}^2\\
& = \Vert f \Vert_{L^2}^2 + \Vert g \Vert_{L^2}^2- 2\bangle{f+g,h} - 2\bangle{ f,g} + 2\bangle{ h,g} + 2\bangle{ f,h} \\
&= \Vert f \Vert_{L^2}^2 + \Vert g \Vert_{L^2}^2- 2\bangle{ f,g}\\
& = \Vert f - g\Vert_{L^2} ^2
\end{align*}
with equality in the application of the Cauchy--Schwartz inequality if and only if there is $\lambda \in \RR$ with $f-h = \lambda(h-g)$ (because neither term can be zero). This implies that $(\lambda + 1) h = f + \lambda g$. Thus if $\lambda \neq  -1$, this implies that 
\[ h = \frac{1}{1+\lambda} f + \frac{\lambda}{1+\lambda} g \in \gamma([0,1]).\] Furthermore, we cannot have that $\lambda = -1$, because that would imply that $f=g$. Thus, we see that for any $h \not\in \gamma([0,1])$ a path from $f$ to $g$ going through $h$ is strictly longer than $\gamma$, so we see that geodesics in $L^2$ are unique, as claimed, which clearly implies that $\cG_0$ is a totally convex subset of $L^2$, with the claimed geodesics.
\end{proof}

\section{Metric Sectional Curvature}\label{sect:synth-sect-curv}

\subsection{Toponogov's Theorem and Alexandrov Curvature} One of the oldest geometric notions which has been generalized to geodesic metric spaces is that of sectional curvature bounds. Because it is relatively simple (and we hope instructive) we sketch the relevant ideas in this section, and then discuss the relationship of sectional curvature of a Riemannian manifold $(M,g)$ with the geometry of $(\cP(M),d^{W})$. 

Suppose that $(M,g)$ is a complete Riemannian manifold (we give a brief introduction to Riemannian geometry, along with many references in Appendix \ref{app:riem-mfld-intro}) with nonnegative sectional curvature, which we write $K(M) \geq 0$. We will see that we can characterize non-negative curvature, $K(M) \geq 0$, by comparing the shape of ``geodesic triangles'' in $M$ to appropriate ones in euclidean space, and this will turn out to be a condition that is only given in terms of the distance metric on $M$ and not the Riemannian metric.\footnote{In fact, for a general $K \in \RR$, a similar ``metric'' version of $K(M) \geq K$ and $K(M) \leq K$ can be given by comparing to the simply connected spaces of constant curvature $K$ in a similar fashion to that described in this section. This more general Toponogov theorem is described in many places, for example in \cite[Chapter 11]{Petersen:RiemGeo} or \cite[Chapter 2]{CheegerEbin:Comp}. The resulting condition on metric spaces leads to the notion of Alexandrov spaces in the case of $K(M) \geq K$ and Hadamard spaces for $K(M) \leq K$. Further details can be found in  \cite{BBI:MetGeo}. } If $p,q,r \in M$ are three points, and $\gamma_{pq},\gamma_{qr},\gamma_{pr}$ are length minimizing geodesics between $p$ and $q$, $q$ and $r$ and $p$ and $r$, respectively, then we will call this a \emph{geodesic triangle}. Clearly for any three points there is at least one such triangle with vertices at those points, but as the example of the sphere shows, there can be more than one such geodesic triangle with vertices at a give triple of points. In spite of this ambiguity, we will usually denote such a geodesic triangle by $\triangle pqr$.

\begin{figure}[h]
\centering
\begin{tikzpicture}[scale =3]
\draw (0,0) arc (260:290:2);
\draw (0,0) arc (150:120:2);
\draw (.73,.73) arc (45:3:1);
\draw (-.1,-.05) node {$p$};
\draw (1.1,0) node {$r$};
\draw (.75,.85) node {$q$};
\draw (0,.6) node {$\triangle pqr$};

\draw (2,0) -- (3.1,.1);
\draw (2,0) -- (2.7,.8);
\draw (3.1,.1) -- (2.7,.8);
\draw (1.9,-.05) node {$\overline p$};
\draw (2.7,.9) node {$\overline q$};
\draw (3.2,.1) node {$\overline r$};
\draw (3.3,.6) node {$\overline{\triangle pqr}$};
\end{tikzpicture}
\caption{A geodesic triangle and corresponding comparison triangle.  }
\label{fig:comptri}
\end{figure}
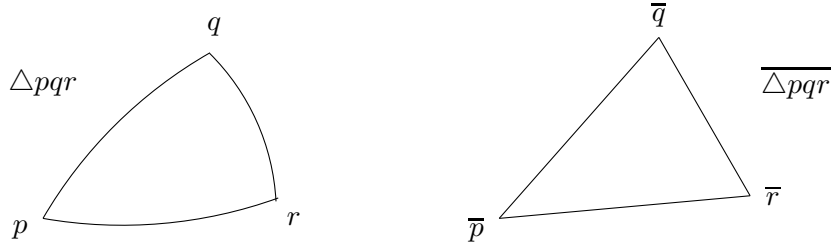

Given a geodesic triangle $\triangle pqr$ in $M$, we consider a \emph{comparison triangle} in $\RR^2$, which is a set of points $\overline p ,\overline q, \overline r$ of the same distance apart, i.e. \[|\overline p \overline q| = d_M(p,q) = L(\gamma_{pq}) , |\overline p \overline r| = d_M(p,r) = L(\gamma_{pr}) \text{ and } |\overline q \overline r| = d_M(q,r) = L(\gamma_{qr}).\] We will denote this comparison triangle by $\overline{\triangle pqr}$. These definitions are illustrated in Figure \ref{fig:comptri}. By the triangle inequality for $d_M$, such a triangle always exists and from basic euclidean geometry such a triangle is unique up to rigid isometry once the side lengths are fixed. Now, let us consider the example of a standard sphere, $(S^2,d_{S^2})$. Here, as illustrated in Figure \ref{fig:spherical-triangle}, a geodesic triangle is ``fatter'' than the corresponding comparison triangle. 
\begin{figure}[h!]
\centering

\newcommand\pgfmathsinandcos[3]{%
  \pgfmathsetmacro#1{sin(#3)}%
  \pgfmathsetmacro#2{cos(#3)}%
}
\newcommand\LongitudePlane[3][current plane]{%
  \pgfmathsinandcos\sinEl\cosEl{#2} 
  \pgfmathsinandcos\sint\cost{#3} 
  \tikzset{#1/.estyle={cm={\cost,\sint*\sinEl,0,\cosEl,(0,0)}}}
}
\newcommand\LatitudePlane[3][current plane]{%
  \pgfmathsinandcos\sinEl\cosEl{#2} 
  \pgfmathsinandcos\sint\cost{#3} 
  \pgfmathsetmacro\yshift{\cosEl*\sint}
  \tikzset{#1/.estyle={cm={\cost,0,0,\cost*\sinEl,(0,\yshift)}}} %
}
\newcommand\DrawLongitudeCircle[2][1]{
  \LongitudePlane{\angEl}{#2}
  \tikzset{current plane/.prefix style={scale=#1}}
  \pgfmathsetmacro\angVis{atan(sin(#2)*cos(\angEl)/sin(\angEl))} %
  \draw[current plane] (\angVis:1) arc (\angVis:\angVis+180:1);
  \draw[current plane,dashed] (\angVis-180:1) arc (\angVis-180:\angVis:1);
}
\newcommand\DrawLatitudeCircle[2][1]{
  \LatitudePlane{\angEl}{#2}
  \tikzset{current plane/.prefix style={scale=#1}}
  \pgfmathsetmacro\sinVis{sin(#2)/cos(#2)*sin(\angEl)/cos(\angEl)}
  \pgfmathsetmacro\angVis{asin(min(1,max(\sinVis,-1)))}
  \draw[current plane] (\angVis:1) arc (\angVis:-\angVis-180:1);
  \draw[current plane,dashed] (180-\angVis:1) arc (180-\angVis:\angVis:1);
}


\tikzset{%
  >=latex, 
  inner sep=0pt,%
  outer sep=2pt,%
  mark coordinate/.style={inner sep=0pt,outer sep=0pt,minimum size=3pt,
    fill=black,circle}%
}

\begin{tikzpicture}[scale=.8] 


\def\R{3} 
\def\angEl{25} 
\def\angAz{-100} 
\def\angPhiOne{-30} 
\def\angPhiTwo{-110} 
\def\angBeta{33} 


\pgfmathsetmacro\H{\R*cos(\angEl)} 
\LongitudePlane[xzplane]{\angEl}{\angAz}
\LongitudePlane[pzplane]{\angEl}{\angPhiOne}
\LongitudePlane[qzplane]{\angEl}{\angPhiTwo}
\LatitudePlane[equator]{\angEl}{0}


\fill[white] (0,0) circle (\R); 
\draw (0,0) circle (\R);


\coordinate (O) at (0,0);
\coordinate (N) at (0,\H);
\path[xzplane] (\R,0) coordinate (XE);
\path[pzplane] (\angBeta:\R) coordinate (P);
\path[pzplane] (\R,0) coordinate (PE);
\path[qzplane] (\angBeta:\R) coordinate (Q);
\path[qzplane] (\R,0) coordinate (QE);


\DrawLatitudeCircle[\R]{0} 

\draw[pzplane,thick] (0,\R) arc (90:0:\R) node [ below = 1] {$r$};
\draw[qzplane,thick] (0,\R) arc (90:0:\R) node [ below = 1] {$q$};
\node [left = 2] at (N) {$p$};
\end{tikzpicture}
\caption{A spherical geodesic triangle. }\label{fig:spherical-triangle}
\end{figure}
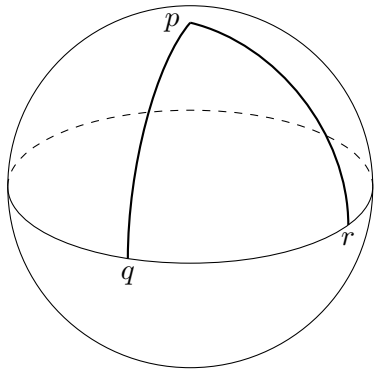
By this, we mean that, for example a point on $\gamma_{qr}$ is farther (as measured by $d_{S^2}$) from $p$ than the corresponding point on $\overline{qr}$ in $\RR^2$.

It turns out that this ``fatness'' property holds in all nonnegatively curved manifolds, as proved by Alexandrov for surfaces, and Toponogov for general manifolds
\begin{theo}[Toponogov]\label{theo:toponogov}
Given a complete Riemannian manifold $(M,g)$ with $K(M)\geq0$, and a geodesic triangle $\triangle pqr \subset M$, let $\overline {\triangle pqr} \subset \RR^2$ be a corresponding comparison triangle. For a point $x \in \gamma_{qr}$, there clearly is a unique $\overline x \in \overline{qr}$ with $|\overline {qx}| = d_M(q,x)$ and $|\overline{xr}| = d_M(x,r)$. For this $\overline x$, we have that 
\begin{equation} \label{eq:top-comp-tri} d_M(p, x) \geq |\overline{px}|.\end{equation}

If $(M,g)$ is an arbitrary Riemannian manifold such that every geodesic triangle $\triangle pqr$ satisfies the above comparison inequality, then $K(M) \geq 0$. 
\end{theo}

See Figure \ref{fig:comptriPosCurv} for an illustration of such a triangle. We do not give proof, as it would take us too far astray (a proof can be found, for example, in \cite[Chapter 11]{Petersen:RiemGeo}). Notice that Toponogov's Theorem tells us that if we know the (metric) distance between any four points of $M$, we can tell whether or not $K(M) \geq 0$ (in fact, all we really need is the corresponding distances for any three points $p,q,r$ and a fourth $x$ lying between $q$ and $r$ in the sense that $d(q,x) + d(x,r) = d(q,r)$). Thus, we have a characterization of a manifold $(M,g)$ having $K(M) \geq 0$ \emph{purely} in terms of the length metric $d_M$!

\begin{figure}[h!]
\centering
\begin{tikzpicture}[scale =3]
\draw (0,0) arc (260:290:2);
\draw (0,0) arc (150:120:2);
\draw (.73,.73) arc (45:3:1);
\draw (-.1,-.05) node {$p$};
\draw (1.1,0) node {$r$};
\draw (.75,.85) node {$q$};
\draw (0,.6) node {$\triangle pqr$};

\draw [style = dashed] (0,0) .. controls (.4,.4) .. (.85,.6);
\draw (.95,.65) node {$x$};

\draw (2,0) -- (3.1,.1);
\draw (2,0) -- (2.7,.8);
\draw (3.1,.1) -- (2.7,.8);
\draw (1.9,-.05) node {$\overline p$};
\draw (2.7,.9) node {$\overline q$};
\draw (3.2,.1) node {$\overline r$};
\draw (3.3,.6) node {$\overline{\triangle pqr}$};

\draw [style = dashed] (2,0) -- (2.82,.63);
\draw (2.9,.67) node {$\overline x$};
\end{tikzpicture}
\caption{Comparing the lengths $d_M(p,x)$ and $|\overline {px}|$ as in Toponogov's Theorem. Here $d_M(p,x) \geq |\overline{px}|$, so this triangle satisfies the conclusion of the theorem.}
\label{fig:comptriPosCurv}
\end{figure}
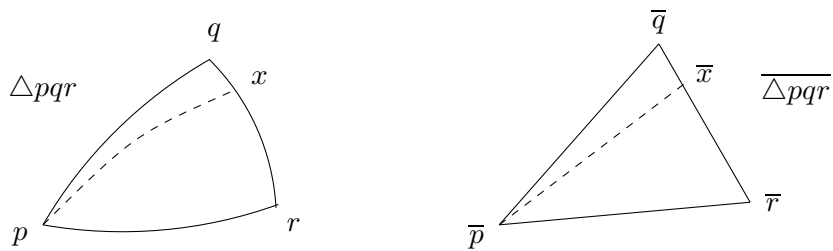

As indicated above, this allows us to generalize $K(M) \geq 0$ to some metric spaces. If $(X,d)$ is a geodesic space it is clear that the notion of geodesic triangles and comparison triangles still makes perfect sense, and we can turn Theorem \ref{theo:toponogov} into a definition, generalizing the notion of nonnegative sectional curvature in such a way that the new definition is equivalent to the standard one when $(X,d)$ is secretly a Riemannian manifold in disguise as a metric space. As such, we will say that a metric space $(X,d)$ has \emph{nonnegative curvature (in the Alexandrov sense)} in this case. 

\subsection{Alexandrov Sectional Curvature of Wasserstein Spaces}
It turns out that the sectional curvature of the base space is related to the Alexandrov curvature of the associated Wasserstein space
\begin{theo} \label{theo:nneg-sect-curv-wass-sp}
For $(M,g)$ a smooth compact connected Riemannian manifold, $M$ has nonnegative sectional curvature if and only if $\cP(M)$ has nonnegative curvature in the Alexandrov sense.\footnote{It is important to note that even though many of the results stated in this essay are simplified (for example, we could easily generalize Toponogov's theorem to sectional curvature bounded from below by $K\in \RR$), the reasonable looking generalization of Theorem \ref{theo:nneg-sect-curv-wass-sp} to say that if $M$ has sectional curvature bounded below by $k$ then $(\cP(M),d^{W})$ has curvature bounded below by $k$ in the Alexandrov sense is \emph{false}. See, for example, \cite[Proposition 2.10]{Sturm:MetMeas1}. }
\end{theo}
See, for example \cite[Theorem A.8]{LottVillani:RicciViaTspt} or \cite[Proposition 2.10]{Sturm:MetMeas1}. The following is the same as \cite[Proposition 4.1]{Kloeckner:GeoWassSpEuc}, except that the given proposition is for the interval, instead of the real line.
\begin{prop}\label{prop:P0-vanish-alex-curv}
The space $(\cP_{0},d^{W})$ has ``vanishing Alexandrov cuvature'' in the sense that equality always holds in \eqref{eq:top-comp-tri} when comparing a geodesic triangle to a comparison triangle.
\end{prop}
\begin{proof}
Vanishing Alexandrov curvature is equivalent to the following property: for $\mu,\nu,\eta \in \cP_{0}$, denoting by $\gamma:[0,1]  \to \cP_{0}$ the unique geodesic between $\nu$ and $\eta$, we have that 
\begin{equation}\label{eq:vanishing-alex-curv} d^{W}(\mu,\gamma(t))^{2} = (1-t) d^{W}(\mu,\gamma(0))^{2} + t d^{W}(\mu,\gamma(1))^{2} -t(1-t) d^{W}(\gamma(0),\gamma(1))^{2}.\end{equation} To see this, we take a comparison triangle $\overline{\triangle pqr}$ as in Figure \ref{fig:van-curv-comp-tri}. We choose the triangle so that $\overline p$ corresponds to $\mu$, $\overline q$ corresponds to $\nu$, and $\overline r$ corresponds to $\eta$. As in Figure \ref{fig:van-curv-comp-tri}, we label the angle $\angle \overline pqr$ by $\theta$. 
\begin{figure}[h!]
\centering
\begin{tikzpicture}[scale =3]
\draw (0,0) -- (.7,1);
\draw (0,0) -- (1,0);
\draw (.7,1) -- (1,0);
\draw (-.1,-.01) node {$\overline q$};
\draw (.7,1.1) node {$\overline p$};
\draw (1.1,0) node {$\overline r$};
\draw [dashed] (.7,1) -- (.6,0);
\draw (.6,-.08) node {$\overline x$};
\draw (.1,0) arc (0:54:.1);
\draw (.18,.1) node {$\theta$};
\end{tikzpicture}
\caption{Example comparison triangle $\overline {\triangle pqr}$ for the calculation of $|\overline{px}|$ in order to show that \eqref{eq:vanishing-alex-curv} is equivalent to $|\overline {px}|$ being \emph{equal} to the corresponding length in $\cP_{0}$.}
\label{fig:van-curv-comp-tri}
\end{figure}
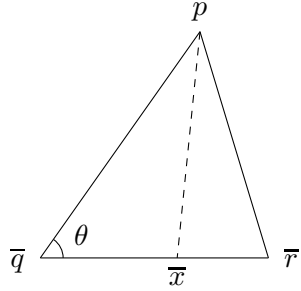
If we choose coordinates so that $\overline q = (0,0)$, $\overline r = (|\overline{qr}|,0)$, then we can take $\overline p = (|\overline{pq}|\cos \theta,|\overline{pq}|\sin \theta)$. Then, $x = (t|\overline{qr}|,0)$, so \[ |\overline {px}|^{2} = (|\overline{pq}|\cos \theta - t|\overline {qr}|)^{2} + (|\overline{pq}|\sin \theta)^{2} =  |\overline{pq}|^{2} + t^{2} |\overline{qr}|^{2 } - 2 t |\overline{pq}| | \overline{qr}| \cos \theta.\]
However, by the law of cosines, we have that 
\[ 2 |\overline{pq}| |\overline{qr}| \cos \theta = - |\overline{pr}|^{2} + |\overline{pq}|^{2} + |\overline{qr}|^{2},\] 
so combining these two gives 
\[ |\overline {px}|^{2} = (1-t) |\overline {pq}|^{2} + t|\overline{qr}|^{2} - t(1-t)|\overline{pr}|^{2}.\] Because $\overline{\triangle pqr}$ is a comparison triangle, it is immediate that $|\overline{px}| = d^{W}(\mu,\gamma(t))$ if and only if \eqref{eq:vanishing-alex-curv} holds. 

To show that \eqref{eq:vanishing-alex-curv} holds, we will show the equivalent identity in $\cG_{0}$, which follows easily from the fact that $L^{2}([0,1])$ is a Hilbert space. For $f,g,h \in \cG_{0}$, by Lemma \ref{lemm-g0-geo} $\gamma(t) : = (1-t) g + t h$ is the unique geodesic between $g$ and $h$. Thus, it is equivalent to show that 
\[ \| f - \gamma(t)\|_{L^{2}}^{2} = (1-t) \| f - g \|_{L^{2}}^{2} + t\| f - h\|_{L^{2}}^{2}  - t(1-t) \| g-h\|_{L^{2}}^{2}. \] To do this, we compute
\begin{align*}
(1-t) &  \|f-g\|_{L^{2}}^{2}  + t\|f-h\|_{L^{2}}^{2} -t(1-t)\|g-h\|_{L^{2}}^{2} \\
&  = \|f\|_{L^{2}}^{2}  - 2(1-t) \bangle{f,g}_{L^{2}} + (1-t)^{2}\|g\|_{L^{2}}^{2} + t^{2} \|h\|_{L^{2}}^{2}  - 2 t \bangle{f,h}_{L^{2}} - 2t(1-t) \bangle{g,h}_{L^{2}}\\
& = \| f - (1-t) g\|_{L^{2}}^{2}  - 2t \bangle{f+(1-t)h,h}_{L^{2}} + t^{2}\|h\|_{L^{2}}^{2}\\
& = \| f- (1-t)g - t h\|_{L^{2}}^{2}\\
& = \|f-\gamma(t)\|_{L^{2}}^{2},
\end{align*}
showing that $(\cP_{0},d^{W})$ has vanishing Alexandrov curvature, as desired. 
\end{proof}


\section{Metric Ricci Curvature} \label{sect:synth-ricci-curv}

Now that we have seen the notion of Alexandrov curvature, we turn to a metric notion of Ricci curvature. It will turn out that convexity of certain functionals along geodesics in $(\cP(X),d^{W})$ provides a good generalization of lower bounds of Ricci curvature. To do this, we fix a reference Borel probability measure $m \in \cP(X)$, and we will consider the geodesic measure triple $(X, d, m)$. We define the entropy functional\footnote{N.B., what we call ``entropy''  is actually Boltzmann's H-functional.} $ \Ent(\cdot| m) : \cP(X) \to \overline \RR = \RR \cup \{+ \infty\} $ by  \[ \mu \mapsto \begin{cases}
\int_{X} \rho \log \rho\  dm & \text{ for } \mu \ll m \text { and } \mu = \rho m\\
+ \infty & \text{ otherwise}
\end{cases}\]
Intuitively, $\Ent(\mu|m)$ measures the nonuniformity of $\mu$ with respect to $m$.  In some sense, we can think of it as a kind of distance (without symmetry or a triangle inequality), which is ``extensive'' in the sense that if $\mu,\mu',m\in \cP(X)$ then $\Ent(\mu\otimes\mu'|m\otimes m) = \Ent(\mu|m) + \Ent(\mu'|m)$.

\begin{defi}\label{defi:k-convex-ent}
We say that for the geodesic measure space $(X,d,m)$, $\Ent(\cdot | m)$ is weakly a.c.\ $K$-displacement convex if for any probability measures $\mu_{0},\mu_{1}\ll m$, there exists a geodesic in $\cP(X)$, $\mu_{t}$ from $\mu_{0}$ to $\mu_{1}$ so that $\Ent(\mu_{t}|m)$ is $K$-convex, in the sense that \begin{equation}\label{eq:k-disp-covex}
\Ent(\mu_{t}|m) \leq t \Ent(\mu_{1}|m)  + (1-t) \Ent(\mu_{0}|m) - \frac K2 t (1-t) d^{W}(\mu_{0},\mu_{1})^{2}. 
\end{equation}
\end{defi}
The ``a.c.''\ corresponds to requiring $\mu_{0},\mu_{1}$ absolutely continuous with respect to $m$ and the ``weakly'' corresponds to requiring \eqref{eq:k-disp-covex} to hold along only one geodesic, not all of them. By \cite[Proposition 3.21]{LottVillani:RicciViaTspt}, we could actually drop the a.c.\ requirement, but it will be more convenient to keep it. Furthermore, on Riemannian manifolds, it turns out (cf.\ \cite[Lemma 3.25]{LottVillani:RicciViaTspt}) that the ``weakly'' condition is unnecessary, but for general geodesic measure spaces it will prove to be important when we consider stability under Gromov--Hausdorff convergence in the next section. One indication this is a good definition is given by
\begin{theo}\label{theo:riem-mfld-equiv}
For a compact Riemannian manifold $(M,g)$, regarding it as a geodesic measure space $\left(M, d, m\right)$, with $m=({\vol(M)})^{-1}{\vol_{M}}$, we have that $\Ent( \cdot | m)$ is weakly a.c.\ $K$-convex if and only if $\Ric \geq K$ on $M$.\footnote{By this, we mean as a bilinear form, or in other words $\Ric \geq K$ if and only if $\Ric(\xi,\xi') \geq Kg(\xi,\xi')$ for all $p\in M$ and $\xi,\xi' \in T_{p}M$. See Appendix \ref{app:riem-mfld-intro}.} 
\end{theo}
As such, we will often refer to a geodesic measure space $(X,d,\mu)$ on which $\Ent(\cdot|\mu)$ is weakly a.c.\ $K$-convex as a space with generalized $\Ric \geq K$.

\subsection{Displacement convexity in $\cP_{0}$} Before proving the above theorem, we will first deal with our special case, $M = [0,1]$, which does not quite fit into the statement of Theorem \ref{theo:riem-mfld-equiv}, because we will not discuss manifolds with boundaries. We could deduce Corollary \ref{coro:ricci-curv-01} from Theorem \ref{theo:riem-mfld-equiv}, by considering $[0,1]$ as a small segment in $S^{1}$ and then showing that a convex subset inherits Ricci bounds. However, it is more instructive to do things explicitly in this case, which is what we will do. 
\begin{prop}\label{prop:ent-01}
We have the following formula for the entropy functional on $\cP_{0}$. For $f \in \cG_{0}$, we have that 
\begin{equation}\label{eq:G0-ent-formula}
\Ent(\Psi(f) |{\Leb}) = - \int_{0}^{1} \log f'(x) dx.
\end{equation}
\end{prop}
By Proposition \ref{prop:G0-P0-iso} and Lemma \ref{lemm-g0-geo}, which relate geodesics in $\cP_{0}$ to those in $\cG_{0}$ (which are simply linear interpolation) this allows us show that 
\begin{coro}\label{coro:ricci-curv-01}
For $\mu_{0},\mu_{1}\in \cP_{0}$, letting $\mu_{t}$ be the unique geodesic between them, $\Ent(\mu_{t}|\emph{\Leb})$ is a convex function (possibly taking the value $+\infty$). Thus, $\Ent(\cdot | \Leb)$ is a.c.\ displacement $0$-convex on $[0,1]$. 
\end{coro}
\begin{proof}[Proof of Proposition \ref{prop:ent-01}]
Because $f$ is monotonically increasing, it is differentiable almost everywhere. First of all, suppose that on some set of full measure, $A \subset [0,1]$, we have that $f|_{A}$ is injective and for all $x\in A$, $f$ is differentiable at $x$ with $f'(x) > 0$. In this case, we can see that by a nonsmooth change of variables (Lemma \ref{lemm:change-of-var}) that $f_{*}\Leb$ is absolutely continuous with respect to $\Leb$, and furthermore $f_{*}\Leb = \rho \Leb$ with 
\[ \rho(f(x)) f'(x) = 1\] for a.e. $x \in[0,1]$. Thus, this gives that 
\[ \Ent(\Psi(f)|\Leb) = \int_{0}^{1} \rho(y) \log \rho(y) dy = \int_{0}^{1} \rho(f(x)) \log\rho(f(x)) f'(x) dx = - \int_{0}^{1} \log f'(x) dx\] as desired (the nonsmooth change of variables $y= f(x)$ is justified by the area formula \eqref{eq:area-formula}).

On the other hand, suppose that such an $A$ does not exist. We will show that both sides of \eqref{eq:G0-ent-formula} are equal to $+\infty$. If $f$ is constant on some interval, it is not hard to see this, because implies that $f_{*}\Leb$ is not absolutely continuous and $f'(x) = 0$ on a set of nonzero measure. On the other hand, if there is $A$ of full measure such that $f|_{A}$ is injective, but it does not hold that $f'(x) >0$ for a.e. $x$, then by Lemma \ref{lemm:change-of-var} we again have that $f_{*}\Leb$ is not absolutely continuous with respect to Lebesgue measure, so again both sides of \eqref{eq:G0-ent-formula} are $+\infty$.  \end{proof}
\begin{proof}[Proof of Corollary \ref{coro:ricci-curv-01}]
For $f,g \in \cG_{0}$, the unique geodesic between them is given by
\[ h_{t}:= (1-t) f + t g,\] 
by Lemma \ref{lemm-g0-geo}. Thus, 
\begin{align*} 
\Ent(\Psi(h_{t})| \Leb ) &  = -\int_{0}^{1} \log( (1-t) f'(x) + tg'(x) ) dx \\
& \leq - \int_{0}^{1} ((1-t) \log (f'(x)) + t \log(g'(x)) ) dx\\
& = (1-t) \Ent(\Psi(f)|\Leb) + t \Ent(\Psi(g)| \Leb) 
\end{align*}
where we have used the concavity of $\log$ in the second step.
\end{proof}

\subsection{Optimal Transport and Displacement Convexity in Riemannian Manifolds} To prove Theorem \ref{theo:riem-mfld-equiv}, we must first understand optimal transport on Riemannian manifolds. With this in mind, we fix a smooth compact  Riemannian manifold $(M,g)$, and write $d$ for the induced length metric.

\begin{defi}[$d^{2}/2$-concavity]
A function $\phi:M\to \RR$ is called \emph{${d^{2}}/{2}$-concave} if there exists a function $\psi: M\to \RR$ so that 
\[ \phi(y) =  \inf_{x\in M}\left[ \frac 12 d(x,y)^{2} - \psi(x) \right] \] for all $y \in M$. 
For a function $\psi: M\to \RR$, we define its \emph{${d^{2}}/{2}$-transform} $\phi^{\frac{d^{2}}{2}}$ by the same formula
\[ \phi^{\frac{d^{2}}{2}}(y) : = \inf_{x\in M} [\frac 12 d(x,y)^{2} - \phi(x)] .\]
\end{defi}

We will need the following properties of ${d^{2}}/{2}$-concave functions
\begin{prop}\label{prop:c-conc-funct-prop}
A ${d^{2}}/{2}$-concave function $\phi$ is Lipschitz (and thus its gradient exists almost everywhere) and has an almost everywhere defined Hessian in the sense of Alexandrov. 

Here, we say that $\phi$ admits a Hessian at $x$ in the sense of Alexandrov if $\phi$ is differentiable at $x$ and there is a self-adjoint operator $H: T_{x}M \to T_{x}M$ satisfying 
the property that for all $v \in T_{x}M$ 
\[ \nabla_{v}\nabla\phi|_{x} = Hv |_{x} \] 
where $\nabla_{v}$ denotes covariant derivation in the $v$ direction. 
\end{prop}
This is proven in \cite{CMS:RiemInterpolationIneq} or in \cite[Theorem 14.1]{Villani:OptTspt}. The authors of \cite{CMS:RiemInterpolationIneq} go on to prove in Corollary 5.2 the following characterization of Wasserstein geodesics between absolutely continuous measures on Riemannian manifolds.
\begin{theo}\label{theo:polar-fact-riem-mflds}
For $\mu,\nu \in \cP(M)$ with $\mu \ll \vol_{M}$, there is a ${d^{2}}/{2}$-concave function $\phi:M\to \RR$ so that the map \[ F_{t} (x) := \exp_{x}(-t\nabla\phi)\] gives $\mu_{t} : = (F_{t})_{*}\mu : [0,1] \to \cP(M)$, which is the unique geodesic between $\mu$ and $\nu$. Furthermore, $(\Id,F_{1})_{*}\mu_{0}$ is an optimal transport plan between $\mu_{0}$ and $\mu_{1}$. For $t \in[0,1)$ $\mu_{t}\ll \vol_{M}$ and if, in addition, $\nu \ll \vol_{M}$, then we have that for all $t\in[0,1]$, $\mu_{t}\ll \vol_{M}$. 
\end{theo}

This was originally proven by McCann in \cite{McCann:PolarFactor}. In particular, if we let $\cP^{\text{ac}}(M) \subset\cP(M)$ denote the subset of probability measures which are absolutely continuous with respect to $\vol_{M}$, then  $\cP^{\text{ac}}(M)$ is a dense, totally convex subset of $\cP(M)$. 


\begin{proof}[Proof of Theorem \ref{theo:riem-mfld-equiv}] \emph{Assuming $\Ric \geq K$:} 
\begin{lemm}\label{lemm:diff-opt-map}
For $\phi : M \to \RR$ a ${d^{2}}/{2}$-concave function, define 
\[
F_{t}(x) : = \exp_{x}(-t\nabla\phi )
\]
 and assume that we have chosen $y \in M$ such that 
\begin{enumerate}
\item $\phi$ admits a Hessian at $y$ (in the sense of Alexandrov)
\item $F_{t}$ is differentiable at $y$ for all $t \in[0,1)$
\item $dF_{t}(y)$ is nonsingular for all $t\in[0,1)$.
\end{enumerate}
then $\cD(t) : = \det(dF_{t}(y))^{1/n}$ satisfies
\[ 
\frac{\ddot \cD(t)}{\cD(t)} \leq - \frac 1n \Ric(\dot F_{t}(y),\dot F_{t}(y)).
\]
\end{lemm}

It will be important for later to notice that by Proposition \ref{prop:c-conc-funct-prop} and \cite[Theorem 4.2]{CMS:RiemInterpolationIneq}\footnote{This theorem provides the a.e. nonsingularity of $dF_{t}$, i.e.\ assumption (3)} , assumptions (1), (2), and (3) are satisfied for a.e. $y \in M$.

\begin{proof}
We fix a $y$ with the above properties, and define a geodesic $\gamma:[0,1]\to M$ by
\[ \gamma(t) := \exp_{y}(-t\nabla\phi(y)).\]
Choosing an orthonormal basis $e_{1},\dots,e_{n}$ for $T_{y}M$ with $e_{1} = \dot\gamma(0)/|\dot\gamma(0)| = - \nabla\phi(y)/|\nabla \phi(y)|$, we parallel transport the basis to a orthonormal frame $e_{1}(t),\dots,e_{n}(t)$ along $\gamma(t)$. 

Defining \[ J_{i}(t) := (dF_{t})_{y}(e_{i}) = \frac{d}{d\delta} \Big|_{\delta=0} \exp_{y + \exp_{y}(\delta e_{i}) }(-t\nabla\phi(y+\exp_{y}(\delta e_{i})), \] this is a vector field along $\gamma(t)$ coming from a variation through geodesics, and is thus a Jacobi field.\footnote{A classical reference for Jacobi fields is Milnor's \cite{Milnor:MorseTheory}, but \cite[Chapter 5]{Petersen:RiemGeo} and \cite[Section 3.C]{GHL:RiemGeo} also contain useful introductions to the topic. We give a brief overview of Riemannian geometry, including Jacobi fields in Appendix \ref{app:riem-mfld-intro}.} Thus, for $t\in[0,1)$, we can define a matrix $J_{ij} (t) := \bangle{J_{i}(t),e_{j}(t)}$ and because the $J_{i}$ are Jacobi fields, we have that by the Jacobi equation, \eqref{eq:jacobi-eqns}
%
%
\[ \ddot J(t) + R(t) J(t) = 0 \]
where \[ R_{kj}(t) : = \bangle{\Riem(e_{k}(t),\dot\gamma(t) ) \dot\gamma(t), e_{j}(t)}. \]
%
Furthermore, notice that 
\[ J_{i}(0) = \frac{d}{d\delta}\Big|_{\delta=0} F_{0}(y+\exp_{y}(\delta e_{i})),\] so 
\[ J(0) = \Id \]
and 
\begin{align*}
 \dot J_{i}(0) &  = \frac{d}{dt} \Big|_{t=0}  \frac{d}{d\delta}\Big|_{\delta=0}  F_{t}(y+\exp_{y}(\delta e_{i}))\\
 &  =  \frac{d}{d\delta}\Big|_{\delta=0}  \frac{d}{dt} \Big|_{t=0}  F_{t}(y+\exp_{y}(\delta e_{i}))  \\
 & =  -  \frac{d}{d\delta}\Big|_{\delta=0} \nabla \phi(y+\exp_{y}(\delta e_{i})) \\
 & = -\nabla_{e_{i}} \nabla\phi|_{y}\\
 & =- \Hess ( \phi)|_{y} e_{i}.
 \end{align*}
 Here we have written $\nabla_{e_{i}}$ for the covariant derivative in the $i$-th direction. Thus, 
 \[\dot J_{ij}(0) = \frac{d}{dt}\Big|_{t=0} \bangle{J_{i}(t), e_{j}(t)} = \bangle{\dot J_{i}(t) , e_{j}(t)} = -  \bangle{\nabla_{e_{i}}(\nabla\phi),e_{j}} \] so we can write
 \[ \dot J(0) = - \nabla(\nabla \phi). \]
 
 Letting $U (t) := \dot J(t) J(t)^{-1}$, we thus have that
 \begin{equation} \label{eq:U-diffeq}\dot U(t) = \ddot J(t) J(t)^{-1} - \dot J(t)J(t)^{-1} \dot J (t)J(t) ^{-1} = -R(t) - U(t)^{2} . \end{equation}
 Taking traces, we thus have 
 \begin{equation}\label{eq:trU-diffeq} \frac {d}{dt} \tr (U(t)) + \tr( U(t)^{2} ) + \tr R(t) = 0\end{equation} and it is clear that 
 \[\addlabel{eq:trRisRicci} \tr R(t) = \Ric_{\gamma(t)} (\dot\gamma(t),\dot\gamma(t)). \]
 
We claim that $U(t)$ is symmetric. To see this, notice that $U(0) =\dot J(0) J(0)^{-1} = - \nabla(\nabla \phi) $, which is a symmetric matrix, by definition, and both $U$ and its transpose satisfy the first order system of ODE's \eqref{eq:U-diffeq}, with the same initial conditions, so they must be equal. This allows us to note that
\[ \tr (U(t)^{2}) = \sum_{i,j=1}^{n} U_{ij}(t) U_{ji}(t) = \sum_{i,j=1}^{n} (U_{ij}(t))^{2} \geq \frac 1n \left( \sum_{i=1}^{n} U_{ii}(t)\right)^{2} = \frac 1 n \tr(U(t))^{2}\] so, combined with \eqref{eq:trU-diffeq}, this gives a differential inequality for $\tr U$
 \begin{equation}\label{eq:trU-diff-rel} \frac {d}{dt} \tr (U(t)) + \frac 1n \tr( U(t) )^{2} + \tr R(t) \leq 0.\end{equation} 
 
 Now, letting \[ \cJ(t) : = \det J(t), \] we have that by assumption (3) in the statement of the lemma, $J$ is invertible for all $t$, so it is a standard fact that 
 \[\addlabel{eq:diff-cJ-U} \dot \cJ(t) = \cJ(t) \tr (\dot J(t) J(t)^{-1}) = \cJ(t) \tr U(t),\] so 
 \[  \frac{d}{dt} \tr(U(t)) = \frac{d}{dt} \left( \frac{\dot \cJ(t)}{\cJ(t)}\right) = \frac{ \cJ(t) \ddot \cJ(t) - (\dot \cJ(t))^{2} }{(\cJ(t))^{2}}\]
 and 
 \[ \frac 1n (\tr (U(t))^{2} = \frac{(\dot \cJ(t))^{2}}{n (\cJ(t))^{2}}\]
 so these combine with \eqref{eq:trRisRicci} and \eqref{eq:trU-diff-rel} to give
 \begin{equation} \label{eq:diff-rel-cJ} \frac{\ddot \cJ(t)}{\cJ(t)}-\left( 1 - \frac 1n\right) \frac{(\dot \cJ(t))^{2}}{ (\cJ(t))^{2}}  +  \Ric_{\gamma(t)} (\dot\gamma(t),\dot\gamma(t)) \leq 0.  \end{equation}
 
 Now, to finish the proof, notice that $\cD(t) = \cJ(t)^{1/n}$, so 
 \[ \addlabel{eq:cD-cJ-reln} n \frac{\dot \cD(t)}{\cD(t)} = \frac{\dot \cJ(t)}{\cJ(t)}\]
 \[ n \frac{\cD(t) \ddot \cD(t)  - (\dot \cD(t))^{2}}{(\cD(t))^{2}} = \frac{\cJ(t) \ddot \cJ(t) - (\dot \cJ(t))^{2}}{(\cJ(t))^{2}} .  \] This implies that 
 \[ n  \frac{\ddot \cD(t)}{\cD(t) } = \frac{\ddot \cJ(t)}{\cJ(t)}-\left( 1 - \frac 1n\right) \frac{(\dot \cJ(t))^{2}}{ (\cJ(t))^{2}}  \leq -   \Ric_{\gamma(t)} (\dot\gamma(t),\dot\gamma(t)) \] and noticing that $\dot \gamma(t) = \dot F_{t}(y)$, this finishes the proof. 
\end{proof}

For $\mu_{0},\mu_{1} \in \cP(M)$ with $\mu_{0},\mu_{1} \ll m = ({\vol(M)})^{-1}{\vol_{M}}$, let $\phi$ be a $\frac{d^{2}}{2}$-concave function such that the map \[ F_{t} (x) := \exp_{x}(-t\nabla\phi)\]  gives $\mu_{t} : = (F_{t})_{*}\mu_{0} : [0,1] \to M$, which is the unique geodesic between $\mu_{0}$ and $\mu_{1}$. By Theorem \ref{theo:polar-fact-riem-mflds}, $\mu_{t} \ll m$, so we can define $\eta_{t}$ by $\mu_{t} = \eta_{t} m$. Then, by the nonsmooth change of variables, Lemma \ref{lemm:change-of-var-mfld}, we have that the densities obey\footnote{The cited lemma requires almost everywhere injectivity of $F_{t}$, which is established in \cite[Lemma 5.3]{CMS:RiemInterpolationIneq}. }
\begin{equation}\label{eq:chang-var-density} \eta_{0}(x) = \eta_{t}(F_{t}(x)) \det(dF_{t}(x)) . \end{equation}
Using a nonsmooth change of variables again, combined with \eqref{eq:chang-var-density} gives
\begin{align*}
 \Ent(\mu_{t}|m) & = \int_{M} \eta_{t}(x) \log \eta_{t}(x) dm(x) \\
 & = \int_{M}  \log(\eta_{t}(F_{t}(x))) \eta_{t}(F_{t}(x)) \det(dF_{t}(x)) dm(x)\\
  & = \int_{M} \log \left( \frac{\eta_{0}(x)}{\det(dF_{t}(x))} \right) \eta_{0}(x) dm(x).
  \end{align*}
Thus, defining 
\[ \cC(x,t) := \log\det(dF_{t})(x) = n  \log \cD(t), \]
we can rewrite the previous equation as
\begin{equation}\label{eq:ent-C-eqn} \Ent(\mu_{t}|m) = \int_{M} \left( \log (\eta_{0}(x)) - C(x,t)\right) \eta_{0}(x) dm(x) .\end{equation}
Because we've assumed that $\Ric \geq K$, for almost every $x \in M$ we have that
\[\ddot \cC(x,t) = n  \frac{\ddot \cD(t) \cD(t) - \dot \cD(t)^{2} }{\cD(t)^{2}} \leq   n  \frac{\ddot\cD(t)}{\cD(t)} \leq - \Ric(\dot F_{t}(x),\dot F_{t}(x)) \leq - K |\dot F_{t}(x)|^{2} = -K |\nabla\phi|^{2}(x) \]
 which implies that 
 \[ \left( C(x,t) + \frac{t^{2}}{2}K|\nabla\phi|^{2}(x) \right)''  \leq 0,\] so it is concave. Thus
 \[ \cC(x,t) +  \frac{t^{2}}{2}K|\nabla\phi|^{2}(x)  \geq t \left( \cC(x,1) + \frac{1}{2}K|\nabla\phi|^{2}(x) \right) + (1-t) \left( \cC(x,0)  \right) \]
 \[ \cC(x,t) \geq t\cC(x,1) + (1-t) \cC(x,0)+  \frac K 2 t(1-t) |\nabla\phi|^{2}(x).\]
Inserting this into \eqref{eq:ent-C-eqn} gives
 \[ \Ent(\mu_{t}|m) \leq t \Ent(\mu_{1}|m) +(1-t)\Ent(\mu_{0}|m) - \frac K 2 t(1-t) \int_{M}|\nabla \phi|^{2}(x) \eta_{0}dm.\]
 Furthermore, as $d(x,F_{1}(x)) = |\nabla \phi|(x)$ and because the transport plan $(\Id, F_{1})_{*}\mu_{0}$ is optimal by Theorem \ref{theo:polar-fact-riem-mflds}, we have that 
 \[ \addlabel{eq:dW-intermsof-phi}d^{W}(\mu_{0},\mu_{1})^{2}= \int_{M} d(x,F_{1}(x)) d\mu_{0}(x) = \int_{M} |\nabla \phi|^{2}(x) d\mu_{0}(x)= \int_{M} |\nabla \phi|^{2}(x)\eta_{0}(x) dm(x) \]  and thus 
  \[ \Ent(\mu_{t}|m) \leq t \Ent(\mu_{1}|m) +(1-t)\Ent(\mu_{0}|m) - \frac K 2 t(1-t) d^{W}(\mu_{0},\mu_{1})^{2},\] as desired.

\emph{Assuming that $\Ent(\cdot|m)$ is weakly a.c. displacement $K$-convex:} Fix $x \in M$ and $v \in T_{x}M$. We will show that $\Ric(v,v) \geq K |v|_{x}$. Choose a smooth positive function $\eta_{0}$ with $\eta_{0} > 0$ in $B_{\delta}(x)$ and supported in $B_{2\delta}(x)$ for a $\delta>0$ which we will fix later. Furthermore, we will assume $\int_{M} \eta_{0}(x) dm = 1$. Thus, letting $\mu_{0} := \eta_{0} m$, this gives $\mu_{0}\in \cP(M)$ and $\mu_{0} \ll m$. Now, take a smooth function $\phi: M \to\RR$ with $\nabla \phi (x) = -v$ and $\Delta \phi (x) = 0$ (for example, in normal coordinates around $x$ such that $v = \partial_{1}$, multiplying $-x_{1}$ by a smooth cutoff function that vanishes away from $x$ satisfies this requirement). For $\epsilon>0$, small enough $\epsilon \phi$ is $\frac{d^{2}}{2}$-concave by \cite[Theorem 13.5]{Villani:OptTspt}. This implies that if we define $F_{t}^{\epsilon}(x) : = \exp_{x}(-\epsilon\nabla \phi(x))$, Theorem \ref{theo:polar-fact-riem-mflds} gives that $\mu_{t} : = (F_{t}^{\epsilon})_{*}\mu_{0}$ is the unique optimal transport between $\mu_{0}$ and $\mu_{1} := (F_{1}^{\epsilon})_{*} \mu_{0}$. By making $\epsilon$ smaller if necessary, we have that $\mu_{1} \ll m$ (this follows from Theorem \ref{theo:polar-fact-riem-mflds}, because we know that the original $\mu_{t} \ll m$ for $t\in [0,1)$). 
Because we have assumed that $\Ent(\mu_{t}|m)$ is $K$-convex, it is not hard to see that this implies that 
\[
\left (\Ent(\mu_{t}|m) + \frac {t^{2}}{2} K d^{W}(\mu_{0},\mu_{1})^{2} \right)'' \geq 0
\] if it is differentiable in $t$ (which it is, in this case). Using \eqref{eq:dW-intermsof-phi}, we have that thus
\[  \frac{d^{2}}{dt^{2}} \left[ \int_{M} \left(\frac{t^{2}}{2} K \epsilon^{2} |\nabla\phi|^{2}(x) - \cC(x,t)\right) \eta_{0}(x) dm(x) \right]'' \geq 0 \] Because everything is smooth, we can differentiate under the integral and then letting $\delta$ and then $\epsilon$ tend to zero, this implies that 
\[\addlabel{eq:ddotC-compare-K} - \ddot \cC(x,0) \geq \epsilon^{2} K |\nabla \phi|^{2}(x) .\]
However, for any fixed $\epsilon,\delta$, we have that, as in the above proof
\[ \ddot \cC(x,0) = n \frac{\ddot \cD(0) \cD(0) - \dot\cD(0)^{2}}{\cD(t)^{2}} .\]
Now, reminding ourselves of the definition of $\cD$ and $U$, we see that, by \eqref{eq:diff-cJ-U} and \eqref{eq:cD-cJ-reln} 
\[ n \frac{\dot \cD(0)}{\cD(0)} = \frac{\dot \cJ(0)}{\cJ(0)} = \tr U(0) = - \tr \nabla(\nabla \phi))(x) = -\Delta \phi(x) = 0. \]
This, combined with the fact that we have equality in \eqref{eq:trU-diff-rel} at $t=0$ (again because $\tr U(0) = 0$) gives that 
\[\ddot C(x,0) = -\epsilon^{2}\Ric_{x}(-\nabla \phi,-\nabla\phi) \]
so 
\[ - \ddot C(x,0) = \epsilon^{2}\Ric_{x}(v,v) \] 
and this, combined with \eqref{eq:ddotC-compare-K}, yields the desired
\[ \Ric_{x}(v,v) \geq K |v|^{2} . \qedhere\] 

\end{proof}

\section{Gromov--Hausdorff Convergence}\label{sect:GH-conv}

\subsection{Definition and Basic Properties} One of the interesting properties of our new notion of Ricci curvature bounds for metric spaces, as in Definition \ref{defi:k-convex-ent} is that it is preserved under a reasonably weak notion of convergence of metric measure spaces, called \emph{Gromov--Hausdorff convergence}.

\begin{defi}
For $X,Y$ compact metric spaces and $\epsilon >0$, we say that a (not necessarily continuous) map $f:X \to Y$ is an \emph{$\epsilon$-isometry} if the following properties hold \begin{enumerate}
\item for $x,x' \in X$, $|d_X(x,x') - d_Y(f(x),f(x'))| \leq \epsilon$ 
\item for all $y \in Y$ there is $x\in X$ with $d_Y(y,f(x)) \leq \epsilon$. 
\end{enumerate}
\end{defi}
\begin{defi}
A sequence of compact metric measure spaces $(X_n,d_{n},\mu_{n})$ \emph{converges in the (measured) Gromov--Hausdorff topology} to a compact metric space $(X,d,\mu)$, written \[(X_{n},d_{n},\mu_{n}) \xrightarrow{GH} (X,d,\mu)\] if there is a sequence $\epsilon_n \to 0$ such that there are $\epsilon_{n}$-isometries, $f_n: X_{n} \to X$ such that $(f_{n})_{*}\mu_{n}\to \mu$ (in the weak* topology on $\cP(X)$). 
\end{defi}

We will say that $(X_{n},d_{n}) \to (X,d)$ in the \emph{(measureless) Gromov--Hausdorff topology} if the above holds without the condition that the $\epsilon_{n}$ isometries push any measures forward. We note that if $(X_{n},d_{n})$ are (compact) geodesic spaces, then so is $(X,d)$. To see this, recall that by Lemma \ref{lemm:geo-sp-iff-midpts}, a necessary and sufficient condition for a space to be a geodesic space is that midpoints exist. Taking $x,x'\in X$, there are $x_{n},x_{n}' \in X_{n}$ such that 
\[ d_{} (f_{n}(x_{n}), x) , d_{}(f_{n}(x_{n}'),x') \leq \epsilon_{n} .\] 
Because $(X_{n},d_{n})$ is a geodesic space, by assumption, Lemma \ref{lemm:geo-sp-iff-midpts} shows that we can find a midpoint $z_{n} \in X_{n}$ such that $d_{n}(z_{n},x_{n}) = d_{n}(z_{n},x_{n}') = \frac 12 d_{n}(x_{n},x_{n}')$. Because $f_{n}$ is an $\epsilon_{n}$ isometry, we then have that 
\[ |d_{}(f_{n}(x_{n}),f_{n}(z_{n})) - d_{{n} } (x_{n},z_{n}) | \leq \epsilon_{n}\] Now, by extracting a subsequence we may assume that $f_{n}(z_{n})$ converges to $z\in X$ (because $X$ is compact), and thus, passing to the limit in the above gives 
\[ d_{}(x,z) = \lim_{n\to\infty} d_{{n}} (x_{n},z_{n}) = \lim_{n\to\infty} \frac 12 d_{n}(x_{n},x_{n}') = \frac 12 d(x,x') .\]
A similar consideration for $z$ and $x'$ shows that $z$ is a midpoint for $x$ and $x'$, so $(X,d)$ is a geodesic space.

We quickly give a few examples of Gromov--Hausdorff convergence (without proof as most of these examples are geometrically believable but actually proving them would be quite tedious and not particularly instructive). For $(X,d)$ a metric space of bounded diameter, we have the convergence \[(X,\lambda d) \xrightarrow{GH} (\{*\},d_{*})\] as $\lambda\to 0$ (where $(\{*\},d_{*})$ is the metric space with one point). 
\begin{figure}[h!]
\centering
\begin{tikzpicture}[scale = 1.2]
\clip (-5.5,-.7) rectangle (6,2);

\begin{scope}[shift = {(-4,0)}]
	\begin{scope}
		\clip (-1,0) rectangle (1,0);
		\draw (-.98,.01) -- (0,2);
		\draw (.98,.01) -- (0,2);
	\end{scope}
	\draw (.98,0) arc (0:180:.98);
	\begin{scope}
		\clip (-1,0) rectangle (1,2);
		\draw [dashed] (0,-.04) ellipse (.98 and .5);
	\end{scope}
	\begin{scope}
		\clip (-1,0) rectangle (1,-2);
		\draw (0,-.04) ellipse (.98 and .5);
	\end{scope}
\end{scope}

\begin{scope}[shift = {(-1.5,0)}]
	\begin{scope}
		\clip (-1,1) rectangle (1,0);
		\draw (-.98,.01) -- (0,2);
		\draw (.98,.01) -- (0,2);
	\end{scope}
	\draw (.49,1) arc (30:150:.565);
	
	
	\begin{scope}
		\clip (-1,0) rectangle (1,2);
		\draw [dashed] (0,-.04) ellipse (.98 and .5);
	\end{scope}
	\begin{scope}
		\clip (-1,0) rectangle (1,-2);
		\draw (0,-.04) ellipse (.98 and .5);
	\end{scope}
\end{scope}

\begin{scope}[shift = {(1,0)}]
	\begin{scope}
		\clip (-1,1.5) rectangle (1,0);
		\draw (-.98,.01) -- (0,2);
		\draw (.98,.01) -- (0,2);
	\end{scope}
	\draw (.245,1.5) arc (30:150:.283);

	\begin{scope}
		\clip (-1,0) rectangle (1,2);
		\draw [dashed] (0,-.04) ellipse (.98 and .5);
	\end{scope}
	\begin{scope}
		\clip (-1,0) rectangle (1,-2);
		\draw (0,-.04) ellipse (.98 and .5);
	\end{scope}
\end{scope}

\draw (2.7,.4) node {\scalebox{1.2}{$\xrightarrow{   \ \  GH \ \ }$}};

\begin{scope}[shift = {(4.5,0)}]
	\draw (-.98,.01) -- (0,2);
	\draw (.98,.01) -- (0,2);
	\begin{scope}
		\clip (-1,0) rectangle (1,2);
		\draw [dashed] (0,-.04) ellipse (.98 and .5);
	\end{scope}
	\begin{scope}
		\clip (-1,0) rectangle (1,-2);
		\draw (0,-.04) ellipse (.98 and .5);
	\end{scope}
\end{scope}
\end{tikzpicture}
\caption{ One of the simplest examples of a singularity developing under Gromov--Hausdorff convergence. The metrics on the spaces are all the induced length metric from the ambient $\RR^{3}$ length structure, i.e.\ the distance between two points is the infimum of the lengths of curves in the surfaces between two points (as measured in $\RR^{3}$ with the standard metric). }
\label{fig:gh-conv-to-cone}
\end{figure}
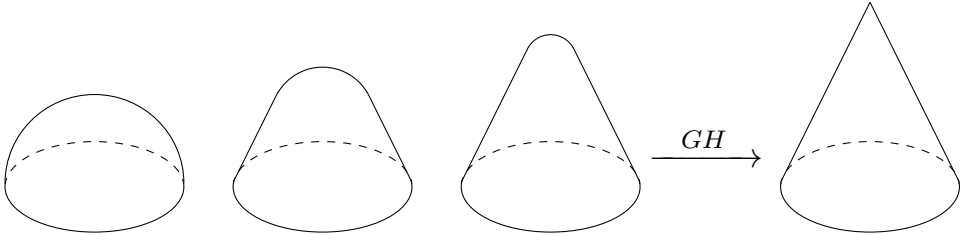
A slightly more exciting example is given in Figure \ref{fig:gh-conv-to-cone}, in which we see hemispheres embedded in $\RR^{3}$ developing a conical singularity under (measureless) Gromov--Hausdorff convergence. It also turns out that topology need not be preserved under Gromov--Hausdorff convergence. Figure \ref{fig:gh-conv-small-handle} gives an example of a small handle shrinking away on a sphere. All of the original spaces are topologically tori, but the limit space is not. Finally, we mention an example in which the dimension decreases in the limit. Recall that the Hopf fibration can be thought of a Riemannian submersion, $S^{3}(1) \to S^{2}(\frac 12)$ (where $S^{2}(\frac 12)$ is the standard 2-sphere of radius $\frac 12$. For $\epsilon >0$, we define a new metric on $S^{3}$ by declaring that (what were originally) unit vectors tangent to the Hopf fibers have length $\epsilon$, and unit vectors perpendicular to the Hopf fiber have length $1$. This results in a $1$-parameter family of metrics on $S^{3}$, giving rise to what is known as Berger spheres; cf.\ \cite[Example 11]{Petersen:RiemGeo}. It turns out that as $\epsilon \to 0$, $(S^{3},g_{\epsilon}) \xrightarrow{GH} S^{2}(\frac 12)$, so we see that the dimension can decrease under convergence.\footnote{A simpler example of decreasing dimension under Gromov--Hausdorff convergence is given, for example, by the squares $[0,1] \times [0,\epsilon] \subset \RR^{2}$ with the induced length metric converging to the interval $[0,1]$ as $\epsilon\to 0$. Dimension can also increase under convergence, as is illustrated by a finer and finer grid in $[0,1]^{2}$ (with the induced length metric) converging to $[0,1]^{2}$. However, with control over the Ricci curvature of the sequence of spaces, the dimension of the limiting space cannot increase (this is discussed in \cite{CheegerColding:Ricc1}). }

\begin{figure}[h!]
\centering
\begin{tikzpicture}[scale = 1.6]

%
\clip (-3.6,-1.3) rectangle (5,1.5);

\begin{scope}[shift = {(-2.5,0)} ]
	\draw (0,0) circle (1);
	\filldraw [white, rotate around = {-20 : (.18,.98)} ] (.175,1.01) rectangle (.27,.96);
	\filldraw [white, rotate around = {50 : (.65,.76)} ] (.64,.76) rectangle (.66,.68);

	\begin{scope}
		\clip (-1,0) rectangle (1,2);
		\draw [dashed] (0,0) ellipse (1 and .5);
	\end{scope}
	\begin{scope}
		\clip (-1,0) rectangle (1,-2);
		\draw (0,0) ellipse (1 and .5);
	\end{scope}	
	\draw (.1,.75)
			to [out=80, in = 180] (.5,1.2)
			to [out = 0, in = 90] (.8,.9)
			to [out = 270, in =30 ] (.45,.55);
	\draw (.18,.72)
			to [out=80, in = 180] (.5,1.1)
			to [out = 0, in = 90] (.7,.9)
			to [out = 270, in =30 ] (.4,.6);	
\end{scope}

%

\begin{scope}[]
	\draw (0,0) circle (1);

	\begin{scope}
		\clip (-1,0) rectangle (1,2);
		\draw [dashed] (0,0) ellipse (1 and .5);
	\end{scope}
	\begin{scope}
			\clip (-1,0) rectangle (1,-2);
			\draw (0,0) ellipse (1 and .5);
		\end{scope}	
		
		\draw (.1,.75)
				to [out=90, in = 180] (.3,.9)
				to [out = 0, in = 90] (.5,.7)
				to [out = 270, in =40 ] (.45,.55);
			\draw (.18,.72)
				to [out=90, in = 180] (.3,.84)
				to [out = 0, in =40 ] (.4,.58);
\end{scope}

\draw (1.7,.2) node {\scalebox{1.2}{$\xrightarrow{   \ \  GH \ \ }$}};
\begin{scope}[shift ={(3.4,0)}]
	\draw (0,0) circle (1);
	\begin{scope}
		\clip (-1,0) rectangle (1,2);
		\draw [dashed] (0,0) ellipse (1 and .5);
	\end{scope}
	\begin{scope}
		\clip (-1,0) rectangle (1,-2);
		\draw (0,0) ellipse (1 and .5);
	\end{scope}	
\end{scope}
\end{tikzpicture}
\caption{Spheres with a small handle converging to a sphere with no handle in the Gromov--Hausdorff topology. This shows that the topologies of the limit space may be different from the spaces in the sequence. }
\label{fig:gh-conv-small-handle}
\end{figure}
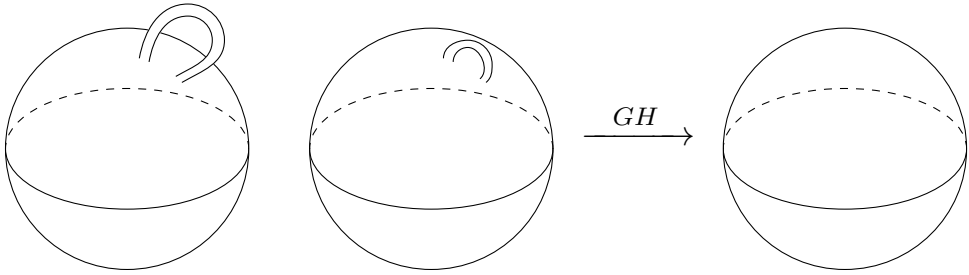

We will not dwell on further properties or examples of Gromov--Hausdorff convergence. A quick introduction to Gromov--Hausdorff convergence can be found in \cite{Cristina:GHNotes} and a more detailed exposition can be the found in the book by Burago, Burago, Ivanov \cite{BBI:MetGeo}. Gromov's book, \cite{Gromov:MetStruct} is a more difficult read, but contains numerous fascinating examples and remarks and we recommend it highly. 

\subsection{Stability of Weak Displacement Convexity}

One important property of our new definition of lower Ricci curvature bounds is that they are preserved under Gromov--Hausdorff convergence. We will give an overview of the proof, relying on several results from \cite{LottVillani:RicciViaTspt}, those concerning Gromov--Hausdorff convergence.

\begin{theo}\label{theo:stab-gh-conv}
If $(X_{n},d_{n},\mu_{n}) \xrightarrow{GH} (X,d,\mu)$ and $(X_{n},d_{n},\mu_{n})$ all have $\Ent(\cdot|\mu_{n})$ weakly a.c. $K$-convex, then the same holds for $(X,d,\mu)$. 
\end{theo}
We remark that we have already seen this phenomena in our above examples. For example, in Figure \ref{fig:gh-conv-to-cone}, the sequence all has nonnegative Ricci, showing that the cone has generalized nonnegative Ricci. Similarly (cf.\ \cite[Section 4.3]{Petersen:RiemGeo}), one can show that for $K < 4$ there is $\epsilon$ small enough so that $\Ric_{(S^{3},g_{\epsilon})} \geq K$ and clearly $S^{2}(\frac 12)$ has $\Ric \geq 4$. 
\begin{proof}
Let $f_{n}:X_{n}\to X$ be $\epsilon_{n}$-isometries with $\epsilon_{n}\to 0$.  Take $\nu_{0},\nu_{1} \in \cP(X)$ with $\nu_{0},\nu_{1} \ll \mu$. We can write $\nu_{i} = \rho_{i}(x) \mu$. Without loss of generality, by \cite[Lemma 3.24]{LottVillani:RicciViaTspt}, we may use an approximation argument to assume that the $\rho_{i}$ are continuous. Thus, we have that 
\[ 1 = \int_{X} \rho_{i}(x) d\mu(x) = \lim_{n\to\infty} \int_{X} \rho_{i}(x) d[ (f_{n})_{*}\mu_{n}](x) = \lim_{n\to\infty} \int_{X_{n}} \rho_{i}(f_{n}(x)) d\mu_{n}(x).\]
This implies that for large enough $n$, defining functions
\[ \tilde \rho_{i}^{(n)} : = \frac{ \rho_{i}(f_{n}(x)) }{\int_{X_{n}} \rho_{i}(f_{n}(x)) d\mu_{n} (x)} ,\]
gives $\tilde \nu_{i}^{(n)} := \tilde \rho_{i}^{(n)}(x) \mu_{n} \in \cP(X_{n})$. Thus, by assumption, there is a geodesic $\tilde \nu_{t}^{(n)}$ in $\cP(X_{n})$ between $\tilde\nu_{0}^{(n)}$ and $\tilde\nu_{1}^{(n)}$ such that 
\[\addlabel{eq:pre-limit-disp-convexity} \Ent(\tilde \nu_{t}^{(n)} | \mu_{n} ) \leq t \Ent( \tilde\nu_{1}^{(n)} | \mu_{n}) + (1 -t) \Ent( \tilde\nu_{0}^{(n)} | \mu_{n}) + \frac K 2 t(1-t) d^{W}(\tilde\nu_{0}^{(n)},\tilde\nu_{1}^{(n)})^{2}.  \] Now, to complete the proof, we will pass to the limit as $n\to\infty$. First, we need that $(\cP(X_{n}),d^{W}) \xrightarrow{GH} (\cP(X) ,d^{W})$, which allows us to find a geodesic in $\cP(X)$ which is the ``limit'' if $\tilde\nu_{t}^{(n)}$ in the appropriate sense. 
\begin{lemm}\label{lemm:ghX-means-ghPX}
The induced maps $(f_{n})_{*}:\cP(X_{n}) \to \cP(X)$ are $\tilde\epsilon_{n}$ isometries where
\[ \tilde \epsilon_{n} : = 4\epsilon_{n} + \sqrt{ 3\epsilon_{n} (2\diam X) + 3\epsilon_{n}}.\]
\end{lemm}
This is proven in \cite[Proposition 4.1]{LottVillani:RicciViaTspt}. As a result of this, up to extracting a subsequence, the (possibly discontinuous) maps $t\mapsto (f_{n})_{*}\tilde\nu_{t}^{(n)}$ tend uniformly (this notion makes sense even if the maps are not continuous) to a geodesic $\nu_{t}\in \cP(X)$ between $\nu_{0}$ and $\nu_{1}$. The proof of this is basically the same as Arzel\'a Ascoli, and is a generalization of the proof of the ``isometry lemma'' in \cite{Gromov:PolyGrow} on page 66. Notice that this is exactly why have assumed \emph{weak} displacement convexity. It is not clear that we can find geodesics in $\cP(X_{n})$ which converge to \emph{any} geodesic between $\nu_{0}$ and $\nu_{1}$ as above. There are examples of Gromov--Hausdorff convergence of length spaces where not all geodesics in the limiting space come are limits of geodesics in the original spaces.\footnote{A standard example of Gromov--Hausdorff convergence in which not all of the geodesics in the limit space are limits of geodesics in the converging spaces is as follows (this is \cite[Example 27.17]{Villani:OptTspt}, but I do not know of its real origin). Consider the metric spaces $(\RR^{n},d^{\ell^{p}})$, where $d^{\ell^{p}}$ is induced by the norm $\|x\|_{\ell^{p}} =\left( \sum_{i=1}^{n} |x|^{p}\right)^{1/p}$. It is not hard to see that as $p\to\infty$ $(\RR^{n},d^{\ell^{p}}) \xrightarrow{GH} (\RR^{n},d^{\ell^{\infty}})$ where $d^{\ell^{\infty}}$ is induced by $\|x\| = \sup_{1\leq i \leq n} |x_{i}|$. When $p<\infty$, geodesics are simply straight lines between the two endpoints, but for $p=\infty$ ``taxi cab'' paths, i.e.\ straight lines in one coordinate and then another and so on are \emph{also} geodesics and it is not hard to see that the majority of geodesics for $p=\infty$ are not limits of $p<\infty$ geodesics. I do not know an explicit example where we have the case we are worried about, i.e.\ $(X_{i},d_{i}) \xrightarrow{GH} (X,d)$ but there are geodesics in $\cP(X)$ that are not limits of geodesics in $\cP(X_{i})$, but it seems likely that with some work the previous example could be used to give such an example. } 

Thus, given that $(f_{n})_{*}\tilde \nu_{t}^{(n)} \to \nu_{t}$ uniformly, we would like to take the limit in \eqref{eq:pre-limit-disp-convexity}. By Lemma \ref{lemm:ghX-means-ghPX}, $d^{W}(\tilde \nu_{0}^{(n)}, \tilde\nu_{1}^{(n)}) \to d^{W}(\nu_{0},\nu_{1})$. Furthermore, for $i=0,1$
\begin{align*} \Ent(\tilde \nu_{i}^{(n)}| \mu_{n}) & = \int_{X_{n}} \tilde \rho_{i}^{(n)}(x) \log\tilde\rho_{i}^{(n)}(x) d\mu_{n}(x)\\
& = \int_{X_{n}} \frac{\rho_{i}(f_{n}(x))}{\int_{X_{n}} \rho_{i}(f_{n}(x)) d\mu_{n} (x)}  \log\left( \frac{\rho_{i}(f_{n}(x))}{\int_{X_{n}} \rho_{i}(f_{n}(x)) d\mu_{n} (x)} \right) d\mu_{n}(x)\\
& =  \int_{X} \frac{\rho_{i}(x)}{\int_{X_{n}} \rho_{i}(f_{n}(x)) d\mu_{n} (x)}  \log\left( \frac{\rho_{i}(x)}{\int_{X_{n}} \rho_{i}(f_{n}(x)) d\mu_{n} (x)} \right) d[(f_{n})_{*}\mu_{n}](x)\\
& \xrightarrow{n\to\infty}   \int_{X} \frac{\rho_{i}(x)}{1}  \log\left( \frac{\rho_{i}(x)}{1} \right) d\mu(x)\\
& = \Ent(\nu_{i}| \mu)
 \end{align*}
 This, combined with \eqref{eq:pre-limit-disp-convexity} shows that 
 \[ 
 \liminf_{n\to\infty } \Ent(\tilde \nu_{t}^{(n)} |\mu_{n}) \leq t \Ent(\nu_{1}|\mu) + (1-t) \Ent(\nu_{0}|\mu) + \frac K 2 t (1-t) d^{W}(\nu_{0},\nu_{1})^{2}.
 \]
 Finally, to conclude the desired inequality, it is enough to show that 
 \[ \addlabel{eq:ent-lsc-push}\Ent(\nu_{t} |\mu) \leq \liminf_{n\to\infty} \Ent(\tilde \nu_{t}^{(n)}|\mu_{n}) \] by general functional analysis arguments combined with the Legendre transform representation of the entropy functional. We give a quick overview of this argument, which is a vastly simplified version of the proof of \cite[Theorem B.33]{LottVillani:RicciViaTspt}. The proof is somewhat obtuse and could certainly be skipped without much harm. For a compact metric space $(Y,d)$ with probability measure $m$, we claim that for any measure $\rho m$ which is absolutely continuous with respect to $m$, we have the representation
 \[ \addlabel{eq:ent-alt-rep} \Ent(\rho m | m) = \sup_{ \varphi \in C(Y)}\left( \int_{Y} (\rho\varphi - e^{\varphi-1} ) dm\right) . \]
 To see this, first, notice that 
 \[ a\ln a \geq ab - e^{b-1} \] for $a,b\in \RR$, which can be seen by maximizing the right hand side in $b$, so letting $a = \rho(x)$ and $b = \varphi(x)$ and integrating with respect to $m$ gives that  
  \[ \Ent(\rho m | m) \geq \sup_{\varphi\in C(Y)} \left( \int_{Y} (\rho\varphi - e^{\varphi-1} ) dm  \right).\]
  On the other hand, define, for $M \in \NN$
  \[ \rho_{M}(y) := \begin{cases} M & \rho(x) > M \\
  \rho(y) & \rho\in\left[\frac 1M, M\right]\\
  \frac 1 M & \rho(y) < \frac  1 M \end{cases}\]
  and let
  \[ \overline\varphi_{M}(x) = \log \rho_{M} + 1.\] 
  Finally, it is standard that there are continuous $\varphi_{M,k} \in C(X)$ which have 
  $\varphi_{M,k} (y) \to \log \rho_{M}(y) + 1$ for $m$-a.e. $y \in Y$. Thus, by dominated convergence
  \[ \lim_{M\to\infty }\lim_{k\to\infty} \int_{Y} (\rho \varphi_{M,k} - e^{\varphi_{M,k} -1}) dm = \lim_{M\to\infty} \int_{Y} (\rho(\log \rho_{M} + 1) - \rho_{M}) dm = \int_{Y} \rho \log \rho dm,\] for $\epsilon > 0$, we see that we can take large enough $M$ and then $k$ such that 
  \[  \int_{Y} (\rho \varphi_{M,k} - e^{\varphi_{M,k} -1}) dm  \geq \Ent(\rho m|m) - \epsilon. \]
  This establishes \eqref{eq:ent-alt-rep}. This seems like a rather useless representation, but in fact it allows us to show \eqref{eq:ent-lsc-push}. To do so, we must first make three observations about this new representation. First of all, notice that we can rewrite \eqref{eq:ent-alt-rep} as
  \[ \Ent(\beta | m) = \sup_{\varphi\in C(Y)} \left( \int_{Y}  \varphi d\mu  -\int_{Y} e^{\varphi-1}  dm\right) \] and if $\beta$ is not absolutely continuous with respect to $m$, then both sides are infinite (take $\varphi$ tending to $1+ \ln M\  \chi_{A}$, where $\beta(A) > 0$ but $m(A) = 0$ and then let $M\to \infty$). Secondly, notice that we also can take the supremum over $L^{\infty}$
  \[ \Ent(\beta|m) =  \sup_{\varphi\in L^{\infty}(Y)} \left( \int_{Y}  \varphi d\mu  -\int_{Y} e^{\varphi-1}  dm\right) .\] Finally, for a fixed $\varphi \in C(Y)$, the map 
  \[ (\beta,m) \mapsto  \int_{Y}  \varphi d\beta  -\int_{Y} e^{\varphi-1}  dm  \] is a continuous functional on $C(Y)^{*}\oplus C(Y)^{*}$, so it is standard that the supremum is lower semicontinuous in both variables. Thus, we can now establish \eqref{eq:ent-lsc-push}, because this shows that
\begin{align*}
 \Ent(\nu_{t}|\mu) & \leq \liminf_{n\to\infty} \Ent((f_{n})_{*}\tilde\nu_{t}^{(n)} | (f_{n})_{*}\mu_{n})\\
  & = \liminf_{n\to\infty} \sup_{\varphi \in C(X_{})}\left( \int_{X_{}} \varphi d(f_{n})_{*} \tilde\nu_{t}^{(n)} - \int_{X_{}} e^{\varphi-1} d(f_{n})_{*}\mu_{n} \right)\\
  & = \liminf_{n\to\infty}\sup_{\varphi \in L^{\infty}(X_{})}\left( \int_{X_{}} \varphi d(f_{n})_{*} \tilde\nu_{t}^{(n)} - \int_{X_{}} e^{\varphi-1} d(f_{n})_{*}\mu_{n} \right)\\
    & = \liminf_{n\to\infty}\sup_{\varphi \in L^{\infty}(X_{})}\left( \int_{X_{n}} \varphi\circ f_{n} d \tilde\nu_{t}^{(n)} - \int_{X_{n}} e^{\varphi\circ f_{n}-1} d\mu_{n} \right)\\
    & \leq  \liminf_{n\to\infty}\sup_{\varphi \in L^{\infty}(X_{n})}\left( \int_{X_{n}} \varphi d \tilde\nu_{t}^{(n)} - \int_{X_{n}} e^{\varphi-1} d\mu_{n} \right)\\
    & = \Ent(\tilde\nu_{t}^{(n)}| \mu_{n}).
    \end{align*}
    In the first inequality, we used the fact that we have shown that $\Ent(\cdot|\cdot)$ is lower semicontinuous in both variables, and in the second inequality that $\{\varphi \circ f_{n} : \varphi \in L^{\infty} (X) \} \subset L^{\infty}(X_{n})$. This completes the proof. 
\end{proof}

Thus, because we see that generalized Ricci lower bounds are preserved under Gromov--Hausdorff convergence, this theory gives a partial answer to the natural question posed by the following theorem of Gromov, which says that a set of Riemannian manifolds of a fixed dimension, with a uniform diameter upper bounds and Ricci curvature lower bounds is precompact in the Gromov--Hausdorff topology. 
\begin{theo}[Gromov's Compactness Theorem]\label{theo:gro-cpt-ricci}
If $(M_{k},d_{k},\mu_{k})$ is a sequence of compact Riemannian manifolds (where $d_{k}$ is the induced length metric and $\mu_{k} = ({\vol(M_{k})})^{-1}{\vol_{M_{k}}}$ is the normalized volume measure) with $\dim M_{k} = n$, $\Ric M_{k} \geq K$ and $\diam M_{k}\leq D$, for a fixed $n$, $K\in \RR$, $D>0$, then there is a subsequence $(M_{k_{j}},d_{k_{j}},\mu_{k_{j}})$ and a metric measure space $(X,d,\mu)$ such that $(M_{k_{j}},d_{k_{j}},\mu_{k_{j}}) \xrightarrow{GH} (X,d,\mu)$. 
\end{theo}
This is proven in, \cite[Corollary 30]{Petersen:RiemGeo}. This theorem leads to the natural question of what properties are inherited by a geodesic measure space in the closure of the set of such manifolds. We have just shown that the limit spaces have the same lower bounds for Ricci (in the weak a.c. displacement convexity sense we have developed above). Further properties of these limit manifolds can be found in Section \ref{subsect:GH-lim-bd-ricci}. 

\subsection{Approximations of $\cP_{0}$} \label{subsect:approx-cP0} The notion of Gromov--Hausdorff convergence allows us to make finite dimensional approximations of $(\cP_{0},d^{W})$. In this section, we will work with measureless Gromov--Hausdorff convergence (i.e., drop the the associated measures, along with the requirement that $(f_{n})_{*}\mu_{n}\to \mu$), and then later we will use our results to discuss measured Gromov--Hausdorff convergence.\footnote{We do not sacrifice much generality by considering measureless Gromov--Hausdorff convergence first, as Prokhorov's theorem guarantees that for any sequence of $\mu_{n}$ associated to the $(X_{n},d_{n})$, some subsequence of $(f_{n})_{*} \mu_{n}$ converges.}

We construct a sequence of metric spaces converging to $(\cG_{0},d^{L^{2}})$ by associating to each partition \[\{t_i\}_{i=1}^N  = \{0=t_0<t_1<\dots<t_N<t_{N+1}=1\}\] a metric space $(X(\{t_i \}),d_{t_i})$ and a map \[f^{\{t_i\}} : (X(\{t_i\}),d_{t_i}) \to (\cG_0,d^{L^{2}})\] which is an isometry onto its image. To do this, we define the set $X(\{t_i\}_{i=1}^N)$ to be\footnote{Notice that $X(\{t_{i}\})$ does not actually depend on the actual partition, but instead just on the number of elements in the partition. However, we label it this way because it will rapidly get confusing when we have different partitions in one expression and hopefully this should make things slightly simpler.} \begin{equation}\label{eq:def-sigma-N}\Sigma_N := \{ (x_1,\dots,x_N)\in [0,1]^N : 0 = x_0 < x_1 < \dots < x_N < x_{N+1} = 1\},\end{equation} and then define a map 
\begin{equation}\label{eq:defi-eps-iso-f} f^{\{t_i\}}(x_1,\dots,x_N) = \sum_{i=1}^N x_i \chi_{(t_i,t_{i+1}]} \in \cG_0 \end{equation} Notice that
\[ \Vert f^{\{t_i\}}(x_1,\dots,x_N) - f^{\{t_i\}}(y_1,\dots,y_N) \Vert_{L^2}^2 = \sum_{i=1}^N |x_i-y_i|^2|t_{i+1}-t_i|,\] so because of this we define 
\[ d^{\{t_i\}}((x),(y))^{2} :={\sum_{i=1}^N |x_i-y_i|^2|t_{i+1}-t_i|}\] and clearly this turns $f^{t_i}$ into an isometry onto its image. 

\begin{prop}\label{prop:gh-conv-to-g0}
For any sequence of partitions of $[0,1]$, $\{t_i^{(k)}\}_{i=1}^{N_k}$ such that the longest interval in the $k$-th partition shrinks to zero \[\delta_k : = \sup_{0\leq i \leq N_{k}} (t_{i+1}^{(k)}-t_i^{(k)}) \to 0,\] we have the Gromov--Hausdorff convergence 
\[ (X(\{t_i^{(k)}\}_{i=1}^{N_k}, d_{t_i^{(k)}}) \xrightarrow{GH} (\cG_0,d^{L^2}) .\] 
\end{prop}
\begin{proof}
To show Gromov--Hausdorff convergence, it is sufficient to show that there are $\epsilon_k\to 0$ such that the $f^{t_i^{(k)}}$ are $\epsilon_k$ isometries, i.e.\ for all $(x),(y) \in X(\{t_i^{(k)}\})$,
\[ |\  \Vert f^{t_i^{(k)}}((x)) - f^{t_i^{(k)}}((y))\Vert_{L^2} - d^{t^{(k)}_i}((x),(y)) | < \epsilon_k \]
and
for any $g \in \cG_0$ there is $(x) \in X(\{t_i^{(k)}\})$ such that
\[ \Vert g - f^{t_i^{(k)}}((x))\Vert_{L^2} < \epsilon_k.\]  
The first property is clearly satisfied because $f^{t_i^{(k)}}$ is an isometry onto its image, so we will establish the second property. For any $g \in \cG_0$, we have that
\begin{align*}
\Vert g - f^{t_i^{(k)}}(g_{t_1},\dots,g_{t_N})\Vert_{L^2}^2 & = \sum_{i=0}^N \int_{t_i^{(k)}}^{t_{i+1}^{(k)}} (g(s) - g(t_i^{(k)}))^2 ds \\
& \leq \sum_{i=0}^N \int_{t_i^{(k)}}^{t_{i+1}^{(k)}} (g(t_{i+1}^{(k)}) - g(t_i^{(k)}))^2 ds\\
& = \sum_{i=0}^N (t_{i+1}^{(k)}-t_i^{(k)})  (g(t_{i+1}^{(k)}) - g(t_i^{(k)}))^2\\
& \leq \left( \sup_{0 \leq i \leq N} (t_{i+1}^{(k)}-t_i^{(k)}) \right) \sum_{i=0}^N (g(t_{i+1}^{(k)})-g(t_i^{(k)}))^2\\
& \leq \left( \sup_{0 \leq i \leq N} (t_{i+1}^{(k)}-t_i^{(k)}) \right) \sum_{i=0}^N (g(t_{i+1}^{(k)})-g(t_i^{(k)})) \\
&= \left( \sup_{0 \leq i \leq N} (t_{i+1}^{(k)}-t_i^{(k)}) \right)\\
 & = \delta_k
\end{align*}
Thus, $\epsilon_k = 2 \sqrt{\delta_k}$ satisfies the above requirements. 
\end{proof}

We will often refer to the ``standard'' partition, defined by
\begin{equation} \label{eq:std-part-defi} T^{\text{std}}_{N} = \left\{ 0= t_0 < \frac{1}{2^N} < \frac{2}{2^N}<\dots<\frac{2^{N}-1}{2^N}<t_{2^N} = 1\right\}.\end{equation} Writing $(X_{N}^{\text{std}},d^{N})$ for the spaces associated with $T_{N}$ as in the above, Proposition \ref{prop:gh-conv-to-g0} shows that $(X_N^{\text{std}},d^N) \xrightarrow{GH} (\cG_0,L^2)$.

It is instructive to relate the form of the $\epsilon_{k}$-isometries in Proposition \ref{prop:gh-conv-to-g0}, i.e.\ the maps \[f^{t_{i}^{(k)}}:X(\{t_{i}^{(k)}\}) \to \cG_{0}\] to the composition \[\Psi \circ f^{t_{i}^{(k)}} : X(\{t_{i}^{(k)}\}) \to  \cP_{0}\] where $\Psi$ is the map coming from Proposition \ref{prop:G0-P0-iso}, because we clearly have that $\Psi\circ f^{t_{i}^{(k)}}$ are $\epsilon_{k}$-isometries. It is not hard to see that
\[ \Psi \circ f^{t_{i}^{(k)}}(x_{1},\dots,x_{N}) = \sum_{i=0}^{N}(t_{i+1}-t_{i}) \delta_{x_{i}}. \] 

This somewhat motivates the following proposition, which we will not prove, as we will not make use of it later.
\begin{prop}
Letting $[0,1]^{n}/\fS_{n}$ denote $n$-tuples of points in $[0,1]$, quotiented out by permutations and $d_{n}$ the quotient metric coming from the euclidean metric on $[0,1]^{n}$, then the maps 
\[ f_{n}: ([0,1]^{n}/\fS_{n},d_{n}) \to( \cP_{0},d^{W}) \] \[ \{x_{1},\dots,x_{n}\} \mapsto \frac{1}{n} \sum_{i=1}^{n} \delta_{x_{i}}\] are $\epsilon_{n}$-isometries with $\epsilon_{n}\to 0$, so
\[ ([0,1]^{n}/\fS_{n},d_{n}) \xrightarrow{GH} ( \cP_{0},d^{W}). \]
\end{prop}

Work relating this notion and ideas discussed in the next section can be found in \cite{AndresRenesse:PartApproxWass}.

\section{Measures on $\cP_{0}$ and Associated Ricci Curvature Bounds} \label{sect:meas-cP0-ricci-curv}

From Proposition \ref{prop:P0-vanish-alex-curv}, we know that $(\cP_{0},d^{W})$ (or equivalently $(\cG_{0},d^{L^{2}})$) has nonnegative (actually vanishing) sectional curvature (in the generalized sense). By analogy with the classical Riemannian case, we expect that this implies that it has nonnegative Ricci curvature in our new generalized notion. However, to make sense of this, we must choose a reference measure, and there is not really an obviously good choice. For example, Petrunin has shown in \cite{Petrunin:AlexMeetsLVS} that if a space with nonnegative sectional curvature in the sense of Alexandrov has finite Hausdorff dimension, then the space with the appropriately normalized Hausdorff measure has nonnegative Ricci curvature.\footnote{Actually, in this paper, Petrunin proves a stronger result, saying that if the space is $m$-dimensional with nonnegative sectional curvature, then the space has nonnegative generalized $m$-dimensional Ricci curvature, a notion which is discussed both in \cite{LottVillani:RicciViaTspt} and in \cite{Sturm:MetMeas2}. In particular, the version we give in the text follows from the ``Main Theorem'' in \cite{Petrunin:AlexMeetsLVS} and \cite[Proposition 1.6]{Sturm:MetMeas2}.} However, it seems that there is no useful notion of infinite dimensional Hausdorff measure, to try to generalize this result to the $(\cG_{0},d^{L^{2}})$ case directly. On the other hand, von Renesse and Sturm, in \cite{RenesseSturm:EntMeasDiff}, have constructed a reasonable candidate measure, which we will refer to as the entropic measure. However, as we will see, it suprisingly does not admit lower Ricci bounds at all.

\subsection{Analytic Consequences} We are not just interested in a measure with lower Ricci bounds on $(\cG_{0},d^{L^{2}})$ for its own sake (although, it is certainly a natural question that we feel is worth studying for its own merit). One reason why we might be interested in this is that there are useful analytic consequences of having lower Ricci bounds, including log Sobolev, Talagrand and Poincar\'e inequalities. We will briefly discuss some of these inequalities. The following theorems are proven \cite{LottVillani:RicciViaTspt} in the setting of generalized lower Ricci bounds.
\begin{theo}[\cite{LottVillani:RicciViaTspt} Corollary 6.12]\label{theo:log-sob}
For $(X,d,\nu)$ with generalized $\Ric \geq K$, if $f$, a Lipschitz function on $X$ satisfies $\int_{X}f^{2}d\nu = 1$, then 
\begin{enumerate}
\item If $K>0$ then
\begin{equation}\label{eq:Kpos-log-sob}
\int_{X} f^{2}\log(f^{2}) d\nu \leq \frac 2 K \int_{X} |\nabla f|^{2}d\nu.
\end{equation}
\item  If $K \leq 0$ then 
\begin{equation}
\int_{X} f^{2}\log(f^{2}) d\nu \leq 2 \diam(X) \sqrt{ \int_{X} |\nabla f|^{2} d\nu} - \frac 12 K \diam(X)^{2}.
\end{equation}
\end{enumerate}
Note that in the above, we have defined \[ |\nabla f|^{2}(x) := \limsup_{y\to x} \frac{|f(x) - f(y)|}{d(x,y)} \] which clearly agrees with the standard notion of the norm of the gradient on a manifold and for differentiable $f$.
\end{theo}

As a consequence of \eqref{eq:Kpos-log-sob}, it is possible to derive a Poincar\'e inequality for geodesic measure spaces with generalized positive Ricci curvature.
\begin{theo}[\cite{LottVillani:RicciViaTspt} Theorem 6.18]\label{theo:poinc}
If $(X,d,\nu)$ has $\Ric \geq K$ for some $K > 0$, for all Lipschitz $f$ on $X$ with $\int_{X} f d\nu = 0$ we have that 
\begin{equation}
\int_{X} h^{2}d\nu \leq \frac 1 K \int_{X} |\nabla h|^{2} d\nu.
\end{equation}
\end{theo}
Log-Sobolev and Poincar\'e inequalities have proven useful in a wide variety of applications. For example, \cite{Villani:TopOptTspt03} describes applications of log-Sobolev inequalities to functional analysis on manifolds, information theory, hypercontractivity estimates, concentration of measure, and entropy production in dissipative equations. Furthermore, Poincar\'e inequalities have become replacements for ``spectral gap'' inequalities when there is no Laplacian available (\cite[Chapter 30]{Villani:OptTspt}). This gives some indication that lower Ricci bounds would have interesting implications for analysis on $(X,d,\nu)$.

\subsection{Failure of Ricci Bounds for the Entropic Measure}

 Von Renesse and Sturm construct a measure $\QQ^{\beta}_{0}$ on $\cG_{0}$ in \cite[Proposition 3.4]{RenesseSturm:EntMeasDiff}. First, we will describe the construction and then explain why one might expect that $(\cG_{0},d^{L^{2}},\QQ^{\beta}_{0})$ has nonnegative generalized Ricci curvature (or at least $\Ric \geq K$ for some $K \in \RR$), which we will then show fails to be true.

 \begin{prop}\label{prop:exist-entrop-meas}
 For $\beta >0$ there is a (unique) probability measure $\QQ^{\beta}_{0} \in \cP(\cG_{0})$ which we will call the \emph{entropic measure} (but is also known as the \emph{Dirichlet process}) such that for each partition of $[0,1]$ \[ 0=t_{0}< t_{1}<\dots<t_{N}<t_{N+1} = 1\] and for all bounded measurable functions $u:[0,1]^{N} \to \RR$, we have that 
 \begin{multline}\label{eq:def-ent-meas}
 {\int_{\cG_{0}} u(g(t_{1}),\dots, g(t_{N}) )d \QQ^{\beta}_{0}}(g) \\
 = \frac{\Gamma(\beta)}{\prod_{i=0}^{N} \Gamma(\beta(t_{i+1}-t_{i}))} \int_{\Sigma_{N}} u(x_{1},\dots,x_{N}) \prod_{i=0}^{N} (x_{i+1} - x_{i})^{\beta(t_{i+1}-t_{i})-1} dx_{1}\cdots dx_{N}
 \end{multline}
 where we define $\Sigma_{N}$ as in \eqref{eq:def-sigma-N} by
 \begin{equation*}\Sigma_N := \{ (x_1,\dots,x_N)\in [0,1]^N : 0 = x_0 < x_1 < \dots < x_N < x_{N+1} = 1\},\end{equation*} and $\Gamma(s) = \int_{0}^{\infty} t^{s-1}e^{-t}dt$ is the Gamma function. 
 
 Furthermore, for each partition of $[0,1]$, $\{t_{i}\}_{i=1}^{N}$, defining a measure $\mu_{\{t_{i}\}}$ on $(X(\{t_{i}\},d^{\{t^{i}\}})$ (these spaces are defined in Section \ref{subsect:approx-cP0}) by requiring that for $u \in C(X(\{t_{i}\}))$
 \begin{multline}
 \int_{X(\{t_i\})} u(x_1,\dots, x_N) d\mu_{\{t_i\}}\\ = \frac{\Gamma(\beta) }{\prod_{i=0}^N\Gamma(\beta(t_{i+1} - t_i)) }\int_{\Sigma_N} u(x_1,\dots,x_N) \prod_{i=0}^N (x_{i+1}-x_i)^{\beta(t_{i+1}-t_i) -1} dx_1\cdots dx_N
 \end{multline}
 and writing $\mu_{N}$ as the measure associated to the standard partition, as defined in \eqref{eq:std-part-defi}, we have the measured Gromov--Hausdorff convergence
 \begin{equation}
 (X_{N}^{\text{std}},d^{N},\mu_{N}) \xrightarrow{GH} (\cG_{0},d^{L^{2}},\QQ^{\beta}_{0}).
 \end{equation}
 \end{prop}
\begin{proof}
Equation \eqref{eq:def-ent-meas} certainly does not define how to integrate against all continuous functions on $\cG_{0}$. To show that $\QQ^{\beta}_{0}$ is well defined and unique, von Renesse and Sturm, in \cite{RenesseSturm:EntMeasDiff}, make use of the Kolmogorov extension theorem. We give a more direct proof by first showing that there can be at most one measure satisfying \eqref{eq:def-ent-meas} and then showing that the pushforward of the $\mu_{N}$ forms a Cauchy sequence in $(\cP(\cG_{0}),d^{W})$ (note that here, in an abuse of notation, $d^{W}$ is Wasserstein distance on $\cP(\cG_{0})$), with its limit satisfying \eqref{eq:def-ent-meas}. This will complete the proof of the proposition, because we have already shown the (measureless) Gromov--Hausdorff convergence $(X_{N},d^{N})\xrightarrow{GH} (\cG_{0},d^{L^{2}})$ in Proposition \ref{prop:gh-conv-to-g0}.
\begin{lemm}\label{lemm:entmeas-is-unique}
There can be at most one measure satisfying \eqref{eq:def-ent-meas}.
\end{lemm}
\begin{proof}
In the proof of Proposition \ref{prop:gh-conv-to-g0}, we showed that for $g\in \cG_{0}$
\[ \| g - f^{N}(g({1/2^{N}}),\dots,g({(2^{N}-1)/2^{N}})) \|_{L^{2}}^{2} \leq \frac {1}{2^{N}} \]
where $f^{N}$ is the map defined in \eqref{eq:defi-eps-iso-f}, associated to the standard partition. This shows that for any continuous $u\in C(\cG_{0})$, $u\circ f^{N} \to u$ uniformly on $\cG_{0}$. Thus, if $\QQ^{\beta}_{0}$ exists
\begin{equation*}
\int_{\cG_{0}} u(g) d\QQ^{\beta}_{0}(g)  = \lim_{N\to\infty}  \int_{\cG_{0}} u\circ f^{N}(g) d\QQ^{\beta}_{0}(g),
\end{equation*}
and assuming that $\QQ^{\beta}_{0}$ satisfies \eqref{eq:def-ent-meas}, the right hand side of this equation is fixed for any $N$, and thus because the limit exists, the limit is uniquely specified.
\end{proof}

\begin{prop}\label{prop:push-q-cauchy-seq}
The push-forward measures associated to the standard partitions \[\{(f^N)_*\mu_N\}_{N=1,2,\dots}\] form a Cauchy sequence in $(\cP(\cG_0),d^W)$ (where $d^W$ is the Wasserstein distance on $\cP(\cG_{0})$). 
\end{prop} 
Before proving this, we recall the following integral identity, which plays a crucial role in the existence of $\QQ^{\beta}_{0}$. We follow the proof given in \cite[Section 1.5]{Lebedev:SpecialFunctions}.
\begin{lemm} [Euler's Beta Integral Identity] \label{lemm:Euler-beta-integral} For real numbers $a < b$ and $\alpha,\beta >0$, we have the following identity
\[ \int_a^b (t-a)^{\alpha-1} (b-t)^{\beta -1} dt = \frac{\Gamma(\alpha) \Gamma(\beta )}{\Gamma(\alpha+\beta)} (a-b)^{\alpha+\beta -1} . \] 
\end{lemm}
\begin{proof}[Proof of Euler's Beta Integral]
By using the change of coordinates \[  t = a + (a-b) \hat t,\] it is sufficient to show that 
\begin{equation} \int_{0}^{1} t^{\alpha-1} (1-t)^{\beta-1} dt = \frac{\Gamma(\alpha)\Gamma(\beta)}{\Gamma(\alpha+\beta)}.\end{equation}
To show this, first recall that one standard representation of the $\Gamma$ function is
\[ \Gamma(x) = \int_{0}^{\infty} t^{x-1}e^{-t}dt\] and making the change of coordinates $t\to rt$ for $r >0$, we have that 
\begin{equation}\label{eq:gamma-scaled}
\frac{\Gamma(x)}{r^{x}} = \int_{0}^{\infty} t^{x-1} e^{-rt} dt. 
\end{equation}
We introduce the variable $u= \frac{t}{1-t}$, and see that 
\begin{equation}\label{eq:beta-u-sub-for-t}
\int_{0}^{1} t^{\alpha-1} (1-t)^{\beta-1} dt = \int_{0}^{\infty} \frac{u^{\alpha-1}}{(1+u)^{\alpha+\beta}} du
\end{equation}
Using \eqref{eq:gamma-scaled} with $r=1+u$ in \eqref{eq:beta-u-sub-for-t} gives
\begin{align*}
\int_{0}^{1} t^{\alpha-1} (1-t)^{\beta-1} dt & = \frac{1}{\Gamma(\alpha+\beta)}\int_{0}^{\infty} \left( u^{\alpha-1} \int_{0}^{\infty} t^{\alpha+\beta-1} e^{-(1+u) t} dt  \right) du\\
& = \frac{1}{\Gamma(\alpha+\beta)} \int_{0}^{\infty} \left( t^{\alpha+\beta-1} e^{-t}\int_{0}^{\infty} u^{\alpha-1} e^{-ut} du \right) dt\\
& = \frac{1}{\Gamma(\alpha+\beta)} \int_{0}^{\infty} \left( t^{\beta-1} e^{-t}\int_{0}^{\infty} u^{\alpha-1} e^{-u} du \right) dt\\
& = \frac{\Gamma(\alpha)\Gamma(\beta)}{\Gamma(\alpha+\beta)} 
\end{align*}
as desired. 
\end{proof}

The Euler beta integral shows the following consistency relation: if $u(x_{1},\dots,x_{N})$ in \eqref{eq:def-ent-meas} does not depend on $x_{k}$, then 
\begin{align*}
\int_{\cG_{0}}  & u(g(t_{1}),\dots,g(t_{N})) d\QQ^{\beta}_{0}  \\
& := \frac{\Gamma(\beta)}{\prod_{i=0}^{N} \Gamma(\beta(t_{i+1}-t_{i}))} \int_{\Sigma_{N}} u(x_{1},\dots,x_{N}) \prod_{i=0}^{N} (x_{i+1} - x_{i})^{\beta(t_{i+1}-t_{i})-1} dx_{1}\cdots dx_{N}\\
 & = \frac{\Gamma(\beta)}{\prod_{i=0}^{N} \Gamma(\beta(t_{i+1}-t_{i}))} \int_{\Sigma_{N}'} \int_{x_{k-1}}^{x_{k+1 }}u(x_{1},\dots,x_{N}) \\
 & \qquad \qquad \qquad \qquad  \qquad  \times\prod_{i=0}^{N} (x_{i+1} - x_{i})^{\beta(t_{i+1}-t_{i})-1} dx_{k} dx_{1}\cdots \hat{dx_{k}} \cdots dx_{N}\\
  & = \frac{\Gamma(\beta)}{\prod_{i=0}^{N} \Gamma(\beta(t_{i+1}-t_{i}))} \int_{\Sigma_{N}'} u(x_{1},\dots,x_{N})  \prod_{i\neq k-1,k} (x_{i+1} - x_{i})^{\beta(t_{i+1}-t_{i})-1}  \\
 &\qquad \qquad \times \int_{x_{k-1}}^{x_{k+1 }} (x_{k}-x_{k-1})^{\beta(t_{k}-t_{k-1}) - 1} (x_{k+1}-x_{k})^{\beta(t_{k+1}-t_{k}) - 1}dx_{k} dx_{1}\cdots \hat{dx_{k}} \cdots dx_{N}\\
  & = \frac{\Gamma(\beta)}{\Gamma(\beta(t_{k+1}-t_{k-1})) \prod_{i\neq k-1,k} \Gamma(\beta(t_{i+1}-t_{i}))} \int_{\Sigma_{N}'} u(x_{1},\dots,x_{N})  \\
  & \qquad \qquad \qquad\qquad  \qquad \times \prod_{i\neq k-1,k}(x_{i+1} - x_{i})^{\beta(t_{i+1}-t_{i})-1} dx_{1}\cdots dx_{N}\\
  & = \int_{\cG_{0}} \tilde u(g(t_{1}),\dots, g(t_{k-1}),  g(t_{k+1}),\dots,g(t_{N})) d\QQ^{\beta}_{0}(g)
\end{align*}
where $\Sigma_{N}' := \{0<x_{1}<\dots<x_{k-1}<x_{k+1}<\dots<x_{N}<1\}$ and $\tilde  u$ is $u$ with the $x_{k}$ variable not written (because we $u$ did not depend on it anyways). We used the Euler beta integral in the second to last line. 

\begin{proof}[Proof of Proposition \ref{prop:push-q-cauchy-seq}]
We bound $d^W((f^N)_* \mu_N, (f^{N+1})_*\mu_{N+1})$ as follows. Let $\pi_N \in \cP(\cG_0\times \cG_0)$ be defined by 
\[ \int_{\cG_0\times \cG_0} u(g,h) d\pi_N (g,h) = \int_{ X(T_{N+1})}u(f^N(\tilde x), f^{N+1} (x) )  d\mu_{N+1}(x)\] where $\tilde x = (x_2,x_4,x_6,\dots,x_{2^N-2})$ and $u \in C(\cG_{0} \times \cG_{0})$. We claim that this has marginals $(f^N)_* \mu_N$ and $(f^{N+1})_*\mu_{N+1}$. The second claim is obvious and the first follows from Euler's beta integral, stated in Lemma \ref{lemm:Euler-beta-integral}. Thus,
\begin{align*} d^W((f^N)_* \mu_N, (f^{N+1})_*\mu_{N+1})^2 &   \leq \int_{\cG_0\times\cG_0} \Vert g - h\Vert_{L^2}^2 d\pi_N(g,h) \\
& = \int_{X_{T_{N+1}}}\Vert f^N(\tilde x) - f^{N+1}(x)\Vert_{L^2}^2 d\mu_{N+1}(x).\end{align*}
We have that\footnote{This alternatively follows from the calculation proving the Gromov--Hausdorff convergence of $X(T_N^{\text{std}})$ to $\cG_0$ with $g = f^{N+1}(x)$ compared to $f^N(g_{1/2^N},\dots, g_{(2^N-1)/2^N})$. }
\begin{align*}
\Vert f^N(\tilde x) - f^{N+1}(x)\Vert_{L^2}^2 & = \frac{1}{2^{N+1}} \sum_{i=0}^{2^{N-1}-1} (x_{2i+1}-x_{2i})^2\\
& \leq \frac{1}{2^{N+1}} \sum_{i=0}^{2^{N-1}-1} (x_{2i+1}-x_{2i})\\
&  \leq \frac{1}{2^{N+1}} \sum_{i=0}^{2^N-1} (x_{i+1}-x_{i})\\
& \leq \frac{1}{2^{N+1}}
\end{align*}
Inserting this into the above shows that $(f^N)_*\mu_N$ is a Cauchy sequence as claimed. 
\end{proof}
Thus $(f^N)_* \mu_N \to q^\beta_0$ for some measure, which we suggestively write $q^\beta_0 \in \cP(\cG_0)$. Now, it remains to show that $q^{\beta}_{0}$ satisfies \eqref{eq:def-ent-meas}. 

\begin{prop}
The measure $q^{\beta}_{0}$ satisfies \eqref{eq:def-ent-meas}. More precisely,
\begin{multline}
 {\int_{\cG_{0}} u(g(t_{1}),\dots, g(t_{M}) )d q^{\beta}_{0}}(g) \\
 = \frac{\Gamma(\beta)}{\prod_{i=0}^{M} \Gamma(\beta(t_{i+1}-t_{i}))} \int_{\Sigma_{M}} u(x_{1},\dots,x_{N}) \prod_{i=0}^{M} (x_{i+1} - x_{i})^{\beta(t_{i+1}-t_{i})-1} dx_{1}\cdots dx_{M}
 \end{multline}
 for all partitions $\{t_{i}\}_{i=1}^{M}$ and continuous functions $u : \RR^{M} \to \RR$. Thus, $\QQ^{\beta}_{0}$ exists and is equal to $q^{\beta}_{0}$.
\end{prop}

The proof of this is somewhat subtle. The main idea is that in order to show that $q^{\beta}_{0}$ satisfies \eqref{eq:def-ent-meas}, for some fixed partition, it is enough to show it after we have added more points to the partition. This follows from Euler's beta integral, because the additional terms simply integrate out to give the desired integrand, assuming that $u$ does not depend on these coordinates. Thus for a given partition $0< t_{1}<\dots<t_{M}$, we can consider the standard partition $T_{N}^{\text{std}}$, for a very large $N$, and adjust the appropriate points so that they agree with the points of the partition we care about, $\{t_{i}\}$. We then show that the pushforward measure associated to this partition converges to $q_{0}^{\beta}$ weakly as $N\to\infty$, which is done by constructing a rather complicated transport plan between this pushforward measure, and that associated to the standard partition (the intuition behind the plan is that we have changed only $M$ of the elements of the partition, which is very small compared to $N$, the total number of elements, so in the transport plan, we should consider the unchanged partition elements to be the same in each, while a small number of elements that have changed do not affect things much). Because the adjusted partition satisfies \eqref{eq:def-ent-meas} for all $N$, it is not hard to show that it satisfies it in the limit, showing the desired result. 

\begin{proof} By Lemma \ref{lemm:entmeas-is-unique}, it is enough to show that for any partition $0<t_1<\dots<t_M<1$,
\[ q^\beta_0 (g_{t_1} \in dx_1,\dots, g_{t_M} \in dx_M) = \frac{\Gamma(\beta)}{\prod_{i=0}^M \Gamma(\beta(t_{i+1}-t_i)) } \prod_{i=0}^M (x_{i+1}-x_i)^{\beta(t_{i+1}-t_i) - 1}dx_1\cdots dx_M.\]
To show this, first fix the partition $0<t_1<\dots<t_M<1$ and let $\delta = \sup_{0\leq i \leq N} (t_{i+1}-t_i)$ and take any $N$ so that $1/2^N < \delta/2$. By this choice, for each $t_i$, there is a unique $j(i)$ such that \[ t_i \in [j(i)/2^N, (j(i)+1)/2^N).\] Clearly, $j(i)$ is strictly increasing and $j(i+1) \geq j(i)  +2$. Now, we modify the ``standard'' partition $T_N^{\text{std}}$ by replacing $j(i)/2^N$ by $t_i$ for each $i \in\{1,\dots,M\}$. We call this new partition $R_N$. We define a coupling $\pi$ between $(f^{R_N})_*\mu_{R_N}$ and $(f^{T_N^{\text{std}}})_*\mu_{T_N^{\text{std}}}$ by 
\[ \pi = (f^{R_N}\times f^{T_N^{\text{std}}})_* \tilde \pi\] for $\tilde \pi \in \cP(X(\{R_N\})\times X(\{T_N^{\text{std}}\})$ given by 
\begin{align*}
 & \int_{X(\{R_N\})\times X(\{T_N\})}  u(x,y) d \tilde \pi(x,y)\\
 &  = C \int_{\tilde X} u(x',x)  \prod_{i \not\in S} (x_{i+1}- x_i)^{\beta/2^N -1} \\
 & \qquad \qquad \qquad \qquad \times \prod_{i=0}^M[(y_i - x_{j(i)})^{\beta(t_i - j(i)/2^N)-1}(x_{j(i)+1} - y_i)^{\beta((j(i)+1)/2^N - t_i) -1}] dxdy
 \end{align*}
 where $S = \{ j(i) : 1\leq i \leq M\}$ and $x'$ is $x$ with the $x_{j(i)}$ terms replaced by $y_{i}$. Furthermore, we have written the normalizing constant as $C$ where
 \[ C = \frac{\Gamma(\beta)}{\prod_{i \not\in S} \Gamma(\beta/2^N) \prod_{i=1}^M [ \Gamma(\beta(t_i-j(i)/2^N)) \Gamma(\beta((j(i)+1)/2^N-t_i))]}\]
 and we are integrating over the set
 \[ \tilde X = \{ (x_1,\dots,x_N, y_1,\dots, y_M): (x) \in X(R_N), y_i \in (x_{j(i)},x_{j(i)+1})\}\subset [0,1]^{N+M}.\]  
 Again, by Euler's beta integral given in Lemma \ref{lemm:Euler-beta-integral}, it is not hard to see that $\pi$ has the desired moments. Given this, we have the bound
 \[ d^W((f^{R_N})_*\mu_{R_N},(f^{T^{\text{std}}_N})_*\mu_{T^{\text{std}}_N})^2 \leq \int_{\cG_0\times \cG_0} \Vert g - h \Vert_{L^2}^2 d\pi(g,h) = \int_{\tilde X} \Vert f^{R_N}(x') - f^{T_N}(x)\Vert_{L^2}^2 d\tilde \pi(x,y). \]
 As before, we compute
 \begin{align*}
\Vert f^{R_N}(x') - f^{T_N}(x)\Vert_{L^2}^2 & = \sum_{i=1}^M [(y_i - x_{j(i)})^2(t_i - j(i)/2) + (x_{j(i)+1}-y_i)^2 ((j(i)+1)/2-t_i)]\\
& \leq \frac {1}{2^{2N}} \sum_{i=1}^M [(t_i - j(i)/2) + ((j(i)+1)/2-t_i)]\\
& = \frac{M}{2^{2N+1}}.
 \end{align*}
 This gives 
 \[ d^W((f^{R_N})_*\mu_{R_N},(f^{T_N})_*\mu_{T_N})^2 \leq \frac{M}{2^{2N+1}}\] and thus $(f^{R_N})_*\mu_{R_N} \to q^\beta_0$ weakly as $N\to \infty$. Finally, by Euler's beta integral, it is clear that for the original partition $0<t_1<\dots<t_M<1$, we have that  
 \begin{multline*} (f^{R_N})_*\mu_{R_N} (g_{t_1} \in dx_1,\dots,g_{t_M} \in dx_M)\\ = \frac{\Gamma(\beta)}{\prod_{i=0}^M \Gamma(\beta(t_{i+1}-t_i)) } \prod_{i=0}^M (x_{i+1}-x_i)^{\beta(t_{i+1}-t_i) - 1}dx_1\cdots dx_M.\end{multline*}
Passing to the limit, this proves that $q^\beta_0$ satisfies \eqref{eq:def-ent-meas} and thus by Lemma \ref{lemm:entmeas-is-unique}, $\QQ^{\beta}_{0}$ exists and $q^{\beta}_{0} = \QQ^{\beta}_{0}$. 
\end{proof} 

Thus, we see that $\QQ^{\beta}_{0}$ exists and is uniquely specified by \eqref{eq:def-ent-meas}, and that we have the convergence 
\begin{equation*}
 (X_{N}^{\text{std}},d^{N},\mu_{N}) \xrightarrow{GH} (\cG_{0},d^{L^{2}},\QQ^{\beta}_{0}).
 \end{equation*}

\end{proof} 

Now that we have constructed $\QQ^{\beta}_{0}$, we will give some indication why one might expect that $(\cG_{0},d^{L^{2}},\QQ^{\beta}_{0})$ admits some Ricci lower bound. The first is that in \cite{RenesseSturm:EntMeasDiff}, von Renesse and Sturm give a heuristic argument that the measure $\PP^{\beta}_{0} := \Psi_{*}\QQ^{\beta}_{0} \in \cP(\cP_{0})$ (which makes $\Psi:(\cG_{0},d^{L^{2}},\QQ^{\beta}_{0}) \to (\cP_{0},d^{W},\PP^{\beta}_{0})$ into a metric measure isomorphism, so Ricci lower bounds on one space are equivalent to bounds on the other space) can formally be written in the form
\[ d\PP^{\beta}_{0}(\mu) = \frac{1}{Z_{\beta}} e^{-\beta \Ent(\mu|m)} d\PP_{0} \]
where $d\PP_{0}$ is some ``uniform'' measure on $\cP_{0}$. Because we can think of $\Ent(\mu|m)$ as a convex function on $\cP_{0}$ by Proposition \ref{coro:ricci-curv-01}, as well a fact which we have not discussed, which is that multiplying a measure by $e^{-V}$ for some convex function $V$ can only \emph{improve} Ricci bounds; cf. \cite{Sturm:MetMeas1} Proposition 4.14. Thus, from this formal point of view, we expect that $(\cG_{0},d^{L^{2}},\QQ^{\beta}_{0})$ admits some Ricci lower bound.

The second reason that one might expect such lower bounds is that von Renesse and Sturm have constructed a symmetric Dirichlet form in \cite{RenesseSturm:EntMeasDiff}, given as the closure in $L^{2}(\cG_{0})$, $(\EE,D(\EE))$, of the quadratic form 
\[ 
\EE(F) : = \int_{\cG_{0}} |\DD F(g)|^{2}_{L^{2}([0,1])} d\QQ^{\beta}_{0}(g)
\]
with domain 
\[\left\{F(g) = \varphi \left(\bangle{f_{1},g}_{L^{2}([0,1])}, \dots, \bangle{f_{m},g}_{L^{2}([0,1])} \right) : m \geq 1, \varphi \in C^{1}_{b}(\RR^{m}), f_{k} \in L^{2}([0,1]) \right\}\]
and where $\DD F(g)$ is the $L^{2}$-Fr\'echet derivative of $F$ at $g$, which for $F$ in the domain described above is 
\[ 
\DD F(g) (x) = \sum_{i=1}^{m} \partial_{i}\varphi \left(\bangle{f_{1},g}_{L^{2}([0,1])}, \dots, \bangle{f_{m},g}_{L^{2}([0,1])} \right) f_{i}(x).
\]
The existence of such a Dirichlet form is interesting for various reasons (e.g.\ see \cite{FOT:DirichletForms}) but in our case, it is relevant because D\"oring and Stannat have shown that $\EE$ satisfies a Poincar\'e inequality
 \begin{theo}[\cite{Stannat:logsobWass} Theorem 1.2]
The Dirichlet form constructed in \cite{RenesseSturm:EntMeasDiff}, $\EE$ satisfies a Poincar\'e inequality with constant less than $\frac 1\beta$, i.e.\ for all $F \in D(\EE)$
 \[ \Var_{\QQ^{\beta}_{0}}(F) \leq  \frac 1 \beta \EE(F). \]
 \end{theo}
 as well as a log-Sobolev inequality
 \begin{theo}[\cite{Stannat:logsobWass} Theorem 1.4] 
 There exists a constant $C$ (independent of $\beta$) such that for $F \in D(\EE)$
 \[ \int_{\cG_{0}} F(g)^{2} \log\left(\frac{F(g)^{2}}{ \| F \|^{2}_{L^{2}(\QQ^{\beta}_{0})}} \right) d\QQ^{\beta}_{0}(g)\leq \frac C \beta \EE(F).\]
 \end{theo}
 
Thus, while Theorems \ref{theo:log-sob} and \ref{theo:poinc} certainly do not show that such inequalities imply lower Ricci bounds, they do suggest that $(\cG_{0},d^{L^{2}},\QQ^{\beta}_{0})$ has some Ricci lower bounds (in fact, examining these theorems, they suggest the tantalizingly simple result that the Ricci curvature is bounded from below by $\beta$). Upon trying to prove this, the first indication that this might not be true is from the fact that the spaces $(X_{N}^{\text{std}},d^{N},\mu_{N})$ can be seen not to have any generalized Ricci lower bounds, by considering the function $V$, defined by $\mu_{N} = e^{-V} \Leb_{[0,1]^{2^{N}-1}}$. It is then a relatively straightforward calculation to show that $\Hess V$ cannot be bounded from below as a quadratic form on $X_{N}^{\text{std}}$, and by examining \cite{Sturm:MetMeas1} Theorem 4.9, this shows there are no generalized Ricci lower bounds on $(X^{\text{std}}_{N},d^{N},\mu_{N})$. This does not prove that $(\cG_{0},d^{L^{2}},\QQ^{\beta}_{0})$ does not have lower Ricci bounds.

 However, we have managed to show that there are no lower Ricci bounds using intuition based on the fact that a Hilbert space (in our case $L^{2}([0,1])$) equipped with a measure (in our case $\QQ^{\beta}_{0}$) having nonnegative Ricci curvature must be log-concave (and in fact this is also sufficient). Further details about this can be found in \cite[Chapter 9]{AGS:GradFlow}. These ideas allow us to prove

\begin{theo}\label{theo:QQbeta-no-ric-geqK}
There is no $K \in \RR, \beta>0$ such that $(\cG_{0},d^{L^{2}},\QQ^{\beta}_{0})$ has generalized $\Ric \geq K$. 
\end{theo}
\begin{proof}
Suppose otherwise, so there is some $K$ such that $\Ric \geq K$. Let, for $s \in (0,1)$
\[ A_{s}: = \{g \in \cG_{0} : g(s) > 1/2 \}\]
\[ B := \{g \in\cG_{0} : g(s) > 0 \}.\]
It is clear that
\begin{equation} \label{eq:AB-convex-comb-C-def} t A_{s} + (1-t) B = \{ g\in \cG_{0} : g(s) > t/2\} := C_{s}(t). \end{equation} 
Notice that 
\[ \QQ^{\beta}_{0}(C_{s}(t)) = \frac{\Gamma(\beta)}{\Gamma(\beta s) \Gamma(\beta (1-s))} \int_{t/2}^{1} x^{\beta s - 1}(1-x)^{\beta (1-s) -1} dx.\]
In particular, $\QQ^{\beta}_{0}(B) = 1$, and for $s \in (0,1)$ and $t \in[0,1]$, $\QQ^{\beta}_{0} (C_{s}(t)) > 0$. Thus, define 
\[ \mu({s}) := \frac{1}{\QQ^{\beta}_{0}(A_{s})} \chi_{A_{s}} \QQ^{\beta}_{0} \in \cP(\cG_{0}) \]
and by assumption, there is a geodesic $\mu(s)_{t}$ between $\mu(s)$ and $\QQ^{\beta}_{0}$ such that $\Ent(\mu(s)_{t}|\QQ^{\beta}_{0})$ is $K$-convex, as in the definition of generalized $\Ric \geq K$.  Because \[\Ent(\mu(s) | \QQ^{\beta}) = \int_{\cG_{0}}  \frac{1}{\QQ^{\beta}_{0}(A_{s})} \chi_{A_{s}}  \log\left( \frac{1}{\QQ^{\beta}_{0}(A_{s})} \chi_{A_{s}} \right) d\QQ^{\beta}_{0} = - \log(\QQ^{\beta}_{0}(A_{s})) < \infty\]
(and clearly $\Ent(\QQ^{\beta}_{0},\QQ^{\beta}_{0}) = 0$) so by assumptions of $K$-convexity, we must have that \[\Ent(\mu(s)_{t}|\QQ^{\beta}_{0}) <\infty,\] in particular implying that $\mu(s)_{t} \ll \QQ^{\beta}_{0}$. Thus, we can write $\mu(s)_{t} = \rho(s)_{t} \QQ^{\beta}_{0}$, and we have that by \eqref{eq:AB-convex-comb-C-def}, we see that $\mu(s)_{t}$ is supported on $C_{s}(t)$\footnote{This follows from the fact that optimal transport maps mass along geodesics of the underlying space. This is intuitively obvious, as if not, we could move along a geodesic between the endpoints, reducing the total distance traveled, and it follows rigorously from \cite[Proposition 2.10]{LottVillani:RicciViaTspt}. Thus, the support of $\mu(s)_{t}$ is contained in $\cup_{\gamma} \gamma(t)$ where the union is over all geodesics $\gamma:[0,1] \to \cG_{0}$ with $\gamma(0) \in A_{s}$ and $\gamma(1) \in B$. Because $\cG_{0}$ is a totally geodesic subset of a Hilbert space, we have that this union is just $C_{s}(t)$. }, implying that
\begin{align*} \Ent(\mu(s)_{t}| \QQ^{\beta}_{0}) & = \int_{\cG_{0}} \rho(s)_{t}\log\rho(s)_{t} d\QQ^{\beta}_{0}\\
& = \int_{\cG_{0}} \log \rho(s)_{t} d\mu(s)_{t}\\
& = \int_{\cG_{0}} \log\left( \frac{d\mu(s)_{t}}{d\QQ^{\beta}_{0}|_{C_{s}(t)}} \frac{1}{\QQ^{\beta}_{0}(C_{s}(t)) }\right) d\mu(s)_{t}\\
& = \int_{\cG_{0}} \log\left( \frac{d\mu(s)_{t}}{d\QQ^{\beta}_{0}|_{C_{s}(t)}}\right) d\mu(s)_{t} -  \int_{\cG_{0}} \log\left({\QQ^{\beta}_{0}(C_{s}(t)) }\right) d\mu(s)_{t} \\
& = \Ent(\mu(s)_{t}|\QQ^{\beta}_{0}|_{C_{s}(t)}) - \log\left({\QQ^{\beta}_{0}(C_{s}(t)) }\right)\\
& \geq  - \log\left({\QQ^{\beta}_{0}(C_{s}(t)) }\right)
\end{align*}
where $\QQ^{\beta}_{0}|_{C_{s}(t)} = \frac{1}{\QQ^{\beta}_{0}(C_{s}(t))} \chi_{C_{s}(t)} \QQ^{\beta}_{0}$.
Combining this with the assumed $K$-convexity of the entropy functional along the path $\mu(s)_{t}$, we thus have that 
\begin{align*}
- \log \left(\QQ^{\beta}_{0}(C_{s}(t))\right) & \leq \Ent( \mu(s)_{t}|\QQ^{\beta}_{0})\\
&  \leq (1-t)\Ent(\mu(s)|\QQ^{\beta}_{0}) + t\underbrace{\Ent(\QQ^{\beta}_{0}|\QQ^{\beta}_{0})}_{=0} - \frac K 2 t(1-t) d^{W}(\mu(s),\QQ^{\beta}_{0})^{2}.
\end{align*}

This implies that, because $\Ent(\mu(s)|\QQ^{\beta}_{0}) = -\log\left( \QQ^{\beta}_{0}(A_{s})\right)$
\begin{equation}\label{eq:log-convex-gen-K} \log\left(\QQ^{\beta}_{0} (C_{s}(t)) \right) \geq (1-t) \log \left( \QQ^{\beta}_{0}(A_{s})\right) + \frac K 2 t(1-t) d^{W}(\mu(s),\QQ^{\beta}_{0})^{2} .\end{equation}
Because $\diam(\cG_{0}) = 1$, we must have that $d^{W}(\mu(s),\QQ^{\beta}) \in[0,1]$, so rearranging \eqref{eq:log-convex-gen-K} and using this observation, we have that 
\begin{equation}
 \log \left( \frac{\QQ^{\beta}_{0}(C_{s}(t))}{(\QQ^{\beta}_{0}(A_{s}))^{1-t}} \right) \geq \frac K 2 t(1-t),
 \end{equation}
 implying that for all $s,t \in (0,1)$
 \begin{equation}\label{eq:final-cond-from-log-conc}
  \frac{\QQ^{\beta}_{0}(C_{s}(t))}{(\QQ^{\beta}_{0}(A_{s}))^{1-t}} \geq \exp\left({\frac K 2 t(1-t)}\right).
 \end{equation}
 We will show that for a fixed $t\in(0,1)$ 
 \begin{equation} \label{eq:lim-final-cond-from-log-conc}
  \lim_{s\to 0}   \frac{\QQ^{\beta}_{0}(C_{s}(t))}{(\QQ^{\beta}_{0}(A_{s}))^{1-t}} = 0,
 \end{equation}
  contradicting \eqref{eq:final-cond-from-log-conc}. To see this, note that 
 \begin{equation}
  \frac{\QQ^{\beta}_{0}(C_{s}(t))}{(\QQ^{\beta}_{0}(A_{s}))^{1-t}}  = \left( \frac{\Gamma(\beta)}{\Gamma(\beta s)\Gamma(\beta (1-s))}\right)^{t} \frac{\int_{t/2}^{1} x^{\beta s - 1}(1-x)^{\beta(1-s) -1} dx}{\left( \int_{1/2}^{1} x^{\beta s -1} (1-x)^{\beta(1-s) -1} \right)^{1-t}}.
 \end{equation}
 It is not hard to see that because we have fixed $t\in(0,1)$, all of the terms have finite and nonzero limits, except for $\Gamma(\beta s)$ which approaches $\infty$, showing \eqref{eq:lim-final-cond-from-log-conc}, and thus showing that there cannot be any generalized Ricci lower bounds on $(\cG_{0},d^{L^{2}},\QQ^{\beta}_{0})$. 
\end{proof}

\subsection{Other Measures}

The failure of the entropic measure to admit generalized Ricci bounds leads one to the following 
\begin{quest}
Are there any reasonable measures on $(\cG_{0},d^{L^{2}})$ with lower Ricci bounds? Here, ``reasonable'' is certainly open to interpretation, but one would at least want a measure supported on all of $\cG_{0}$.
\end{quest}
 By repeating the proof of Theorem \ref{theo:QQbeta-no-ric-geqK}, and appropriate modification of Theorems 9.4.10 and 9.4.11 in \cite{AGS:GradFlow} (the stated theorems are given for $K=0$, but the proof for general $K$ goes through with basically no change) we have the following necessary and sufficient condition

\begin{theo}
For a probability measure $\MM \in \cP(\cG_{0})$, the space $(\cG_{0},d^{L^{2}},\MM)$ has generalized $\Ric \geq K$ if and only if $\MM$ is $K$-log-concave, in the sense that for open sets $A,B \subset \cG_{0}$ 
\begin{equation}\label{eq:K-log-concave}
\log \left( \MM((1-t) A + tB) \right) \geq (1-t) \log \left( \MM(A)\right)  + t \log\left( \MM (B)\right) + \frac K 2 t(1-t) d^{W}(\MM|_{A},\MM|_{B})^{2}
\end{equation}
where $\MM|_{A} : = \frac{1}{\MM(A)} \chi_{A} \MM$, and similarly for $\MM|_{B}$, and $d^{W}$ is the Wasserstein distance on $\cP(\cG_{0})$. In fact, it is enough to check \eqref{eq:K-log-concave} for the balls $A = B_{r}(f)$, $B = B_{r}(g)$ for all $0<r<r_{0}$ for any $r_{0}>0$ and $f,g \in \cG_{0}$. 
\end{theo}

This seems to be a difficult condition to verify, but as in Theorem \ref{theo:QQbeta-no-ric-geqK}, it seems possible that it could rule out Ricci lower bounds for a given $\MM$. Finally, we mention that Theorem \ref{theo:stab-gh-conv} gives a possible method for constructing measures $\MM \in \cP(\cG_{0})$ with lower Ricci bounds, by constructing measures on $X_{N}$ with lower Ricci bounds, whose pushforwards converge to $\MM$. This is certainly a good way to construct {a} measure with lower Ricci bounds, because for  any sequence of measures $\MM_{N}$ on $X_{N}^{\text{std}}$, by Prokhorov's theorem, we can extract a subsequence such that $(X_{N'}^{\text{std}},d_{N},\MM_{N'}) \xrightarrow{GH} (\cG_{0},d^{L^{2}},\MM)$ for some $\MM \in \cP(\cG_{0})$. However, it is quite tricky to choose the $\MM_{N}$ appropriately, so that the limit $\MM$ is useable in any way. For example, it is easy to show that by taking normalized Lebesgue measure on $X_{N}^{\text{std}}$ we have nonnegative Ricci curvature bounds, so some subsequence converges to a measure $\MM \in \cP(\cG_{0})$ with nonnegative Ricci curvature. However, we cannot find a reasonably nice description of this measure, and have not been able to prove any interesting properties about it, other than the fact that it has nonnegative Ricci curvature. For example, we are unable determine the support, which is quite unsettling, as if the support of $\MM$ is not all of $\cG_{0}$, then this could be a highly trivial statement (for example $\delta_{1} \in \cP(\cG_{0})$ where $1$ is the function $t\mapsto 1$ has nonnegative Ricci curvature, but this statement means absolutely nothing, as the only measure in $\cP(\cG_{0})$ which is absolutely continuous with respect to $\delta_{1}$ is $\delta_{1}$ {itself}.)

As a final remark, we note that further exploration of this topic could be aided by the following description of geodesics in $\cP(X)$ where $X$ is a Hilbert space.\footnote{We have not  discussed noncompact base spaces in our treatment of Wasserstein space, but the reader should be reassured that most things discussed above are true, and the biggest changes is that we must restrict to measures with finite second moment, $\int_{X} d(x,x_{0})^{2} d\mu(x) < \infty$ for all $x_{0}\in X$ and the weak* topology is replaced by saying that $\mu_{k}\to \mu$ if it converges in the dual topology of continuous bounded functions on $X$, i.e.\ $C_{b}(X)^{*}$ and furthermore that $\int_{X} d(x_{0},x)^{2} d\mu_{k}(x) \to \int_{X} d(x_{0},x)^{2} d\mu(x)$. Anyways, as the reader will soon see, we will not actually be interested in geodesics in $\cP(L^{2}([0,1]))$, just of those in the totally geodesic subset $\cP(\cG_{0})$, which certainly fits in our framework discussed previously. See the conclusion (Section \ref{sect:conclusion}) for further bibliographic references. } We explain the notation in the subsequent paragraph.
\begin{theo}[\cite{AGS:GradFlow} Theorem 6.2.10] 
For $X$ a separable Hilbert space, suppose that $\mu \in \cP_{2}^{r}(X)$, and $\nu \in \cP_{2}(X)$ with bounded support, then there exists a locally Lipschitz, ${d^{2}}/{2}$-concave, maximal Kantorovich potential $\varphi$ such that the unique geodesic in $\cP_{2}(X)$ between $\mu$ and $\nu$ is given by $(r_{t})_{*} \mu$, where \[ r_{t}(x) = x - t \nabla \varphi(x) .\]
\end{theo}
In the above, $\cP_{2}(X)$ is the set of probability measures $\mu$ with $\int_{X} |x-x_{0}|^{2} d\mu(x) < \infty$ for some (and thus all) $x_{0}$, $\cP_{2}^{r}(X)$ is the set of ``regular measures'' with finite second moments, where a measure $\mu$ is ``regular'' if $\mu(B) = 0$ for for all Borel set $B$ with the property that $\gamma(B) = 0$ for any measure $\gamma \in \cP(X)$ such that \[L_{*} \gamma = \frac{1}{\sqrt{2\pi\sigma }} e^{-|t-m|^{2}/2\sigma^{2}} dt\] for all functionals $L \in X^{*}$. (Such measures $\gamma$ are called Gaussian measures, and such sets $B$ are called Gaussian null sets). We furthermore note that $\nabla\varphi$ is the ``Gateaux differential of $\varphi$,'' defined by 
\[ \varphi(x+y) = \varphi(x) +  \bangle{\nabla\varphi(x), y} + o(\|y\|) .\]
Finally, the requirement that $\varphi$ is a maximal Kantorovich potential means that $(\varphi,\varphi^{c})$ is a maximizing pair, achieving the value of
\[ \sup_{(\varphi, \psi) \in C_{b}(X)\times C_{b}(X)} \left\{ \int_{X} \varphi(x) d\mu(x) + \int_{X} \psi(y) d\nu(y) \right\}. \]

Thus, because $\cG_{0} \subset L^{2}([0,1])$ is a totally geodesic subset, this theorem gives a characterization of some of the geodesics in $\cP(\cG_{0})$. We note, however, that it seems rather difficult to check the regularity of a measure in $\cP(\cG_{0})$, in the above sense, so this theorem may not actually be of much help in verifying displacement convexity for a chosen reference measure $\MM$. 


\section{Conclusion}\label{sect:conclusion}

In this section, we briefly mention in what ways the results presented in this paper are simplifications of the full theory. There is no need to restrict to compact base spaces, as there is a perfectly good notion of Wasserstein distance and optimal transport on the space of probability measures on a (complete) geodesic space $X$ with finite second moments (i.e.\ $\int_{X} d(x_{0},x)^{2} d\mu(x) < \infty$ for some $x_{0}\in X$). This allows us to extend the definition of generalized Ricci bounds to complete geodesic spaces, which may not be compact. In fact, this is preserved under the weaker notion of pointed Gromov--Hausdorff convergence, which is better suited to the noncompact setting. In another direction, it turns out that it is beneficial to consider $n$-dimensional manifolds equipped with measures of the form $e^{-V} m$ (where $m$ is the normalized volume measure), for some function $V$. In this case, it is advantageous to consider the $N$-Ricci tensor for $N \in [1,\infty]$
\[Ric_{N} = \begin{cases}
\Ric + \Hess V & N =\infty\\
\Ric + \Hess V - \frac{1}{N-n} (dV\otimes dV)  & N \in (n,\infty)\\
\Ric + \Hess V - \infty (dV \otimes dV) & N = n\\
-\infty & N < n
\end{cases} \]
In the $N=\infty$ case, it is not hard to generalize the result of Theorem \ref{theo:riem-mfld-equiv} and show that weak a.c. displacement $K$-convexity of $\Ent(\cdot|e^{-V} m)$ is equivalent to $\Ric_{\infty} \geq K$. However, it turns out that there are displacement convexity characterizations of $\Ric_{N} \geq K$ as well, which we have not discussed. Lott and Villani, in \cite{LottVillani:RicciViaTspt} show that for $N < \infty$, $\Ric_{N} \geq 0$ is equivalent to displacement $0$-convexity of the R\'enyi functional, defined by 
\[ S(\rho m | m ) := - \int_{X}\rho^{-1/N}dm.\]
In \cite{Sturm:MetMeas2}, Sturm discusses this notion, as well as discussing a notion of generalized $\Ric_{N} \geq K$ for arbitrary $N \in [1,\infty]$ and $K \in \RR$, which is more involved. Sturm proves that under this assumption, the metric measure space satisfies the Bishop-Gromov volume inequality (cf.\ \cite[Theorem 2.3]{Sturm:MetMeas2}), that such bounds are preserved under Gromov--Hausdorff convergence, as well as discussing analytic consequences of such bounds (and other, weaker variants). We note that in terms of relationship with $\cP_{0}$, these notions are not useful, however, because such generalized bounds imply that the underlying space has Hausdorff dimension less than $N$ (cf.\ \cite[Corollary 2.5]{Sturm:MetMeas2}). Discussion of all of these results can be found in great generality in Villani's book \cite{Villani:OptTspt}.

Finally, we note that in \cite{Sturm:MultEnt}, Sturm has given a definition of the entropic measure over a general compact Riemannian manifold, denoted $\PP^{\beta}\in \cP(\cP(M))$. We have not discussed this measure, because our discussion of $\QQ^{\beta}_{0}\in\cP(\cG_{0})$ has heavily relied on the existence of the isometry $\Psi : \cG_{0}\to \cP_{0}$. Such an isometry does not exist in higher dimensions\footnote{There is still a homeomorphism from $\cG\to \cP(M)$ given by $g\mapsto g_{*}m$, where $\cG$ is the set of functions on $M$ of the form $\exp_{x}(-\nabla\varphi)$ for $\varphi$ a ${d^{2}}/{2}$-concave $\varphi$, but this is not an isomorphism for dimension $n>1$.}, and thus we cannot make use of the Hilbert space setting that proved crucial in, e.g.\ the proof of Theorem \ref{theo:QQbeta-no-ric-geqK}. However, it seems possible that if one carefully examined the proof of Theorem \ref{theo:QQbeta-no-ric-geqK}, it might be possible to translate it back into a proof entirely on $\cP_{0}$, which could then be generalized to show that
\begin{conj}
For $M$ a compact Riemannian manifold, and $\beta >0$, using the higher dimensional entropic measure $\PP^{\beta} \in \cP(\cP(M))$, as constructed in \cite{Sturm:MultEnt} as a background measure, $(\cP(M),d^{W},\PP^{\beta})$ does not admit generalized $\Ric \geq K$ for any $K \in \RR$. 
\end{conj}

As in the one dimensional setting, the question of a natural ``volume measure'' in $\cP(\cP(M))$ with lower Ricci bounds certainly deserves further invesigation. 

\appendix

\section{Nonsmooth Change of Variables}\label{app:non-smooth-cov}

In this section, we briefly recall some facts about changing variables under a function that is only differentiable almost everywhere. This is particularly important to the theory of optimal transport on Riemannian manifolds, because by Theorem  \ref{theo:polar-fact-riem-mflds}, we know that for $\mu,\nu \ll \vol_{M}$, the unique geodesic between them is given by $(F_{t})_{*}\mu$ where $F_{t} (x) = \exp_{x}(-t\nabla \phi(x))$ for $\phi$ a $\frac{d^{2}}{2}$-concave function. Thus, in general $F_{t}$ is differentiable almost everywhere, and it will be important to understand the density of the pushforward measure $(F_{t})_{*}\mu$. Our discussion below will mostly deal with functions on $\RR^{n}$, but by working locally, can be easily extended to manifolds. 

To begin with, we recall that by the area formula, proven in e.g.\ \cite[Section 3.2.3]{Federer:GMT}. For $f:\RR^{n}\to\RR^{n}$ a Lipschitz function, and $h:\RR^{n} \to [0,\infty]$, a nonnegative Borel function, we have that
\begin{equation}\label{eq:area-formula}
\int_{\RR^{n}} h(x) |\det \nabla f(x)| dx = \int_{\RR^{n}} \sum_{x\in f^{-1}(y)} u(x) dy.
\end{equation}
The following lemma and proof comes from \cite[Section 5.5]{AGS:GradFlow}. 
\begin{lemm}\label{lemm:change-of-var}
Suppose that $\rho\in L^{1}(\RR^{n})$ a nonnegative function with $\int_{\RR^{n}} \rho(x) dx = 1$. Furthermore, suppose $f:\RR^{n}\to\RR^{n}$ is a Lipchitz function such that there is a Borel set $A \subset \RR^{n}$ such that $\{\rho>0\} \setminus A$ has $\Leb$-measure zero and $f|_{A}$ is injective. Then $f_{*} (\rho\Leb) \ll \Leb$ if and only if $|\det \nabla f | > 0$ $\Leb$-a.e. on $A$, and in this case
\begin{equation}
f_{*}(\rho \Leb) =\left[ \frac{\rho}{|\det \nabla f| } \circ (f|_{A})^{-1}\right] \Bigg|_{f(A)} \Leb.
\end{equation}
\end{lemm}
\begin{proof}
If $|\det \nabla f(x)| > 0$ for a.e. $x\in \RR^{d}$, then from the area formula \eqref{eq:area-formula} with 
\[ 
h =\frac{ \rho \chi_{(f|_{A})^{-1}(B) \cap A} }{|\det \nabla f|}
\]
for $B\subset\RR^{d}$ an arbitrary Borel set, then
\begin{align*} f_{*}(\rho \Leb) (B) & = \int_{(f|_{A})^{-1}(B)} \rho(x) dx \\
& = \int_{(f|_{A})^{-1}(B) \cap A} \rho(x) dx\\
& = \int_{\RR^{n}} h(x) |\det \nabla f(x) | dx\\
& = \int_{\RR^{n}} \frac{\rho  \chi_{(f|_{A})^{-1}(B) \cap A}}{ |\det\nabla f|}\circ (f|_{A})^{-1}(y) dy\\
& = \int_{B \cap f(A)} \frac{\rho }{ |\det\nabla f|}\circ (f|_{A})^{-1}(y) dy,
\end{align*}
as desired. 

Conversely, if there is a Borel set $B$ with $\Leb(B) > 0$ but $|\det \nabla f(x)| = 0$ for $x \in B$, then the area formula \eqref{eq:area-formula}, with $h = \chi_{B}$, gives 
\[ 0 = \int_{B} |\det \nabla f(x)| dx = \int_{\RR^{n}} \chi_{B}\circ (f|_{A})^{-1}(y) dy = \Leb(f|_{A}(B)). \] However, we have that 
\[ f_{*} (\rho\Leb) (f|_{A}(B)) = \int_{f^{-1}(f|_{A}(B))} \rho(x) dx  \geq \int_{f^{-1}(f|_{A}(A \cap B))} \rho(x) dx = \int_{A\cap B} \rho (x) dx > 0  \] because $\Leb(A\cap B \cap \{\rho > 0\} ) = \Leb(B\cap \{\rho>0\}) > 0 $. This shows that $f_{*}(\rho\Leb)$ is not absolutely continuous with respect to Lebesgue measure. 
\end{proof}

This easily gives the following (\cite{Villani:OptTspt}, Chapter 1, ``change of variables formula'')
\begin{lemm}\label{lemm:change-of-var-mfld}
For $M$ a $n$-dimensional (compact) Riemannian manifold, with $m = \frac{\vol_{M}}{\vol(M)}$ the normalized volume measure, if $\mu_{0},\mu_{1} \ll m$ with $\mu_{i} = \rho_{i} m$ and $F : M \to M$ is Lipschitz and almost everywhere injective (i.e.\ when restricted to a set of full measure is injective) with $F_{*}\mu_{0} = \mu_{1}$, then 
\[ \rho_{0} (x) = \rho_{1}(F(x)) \Jac F. \]
\end{lemm}

\section{A Brief Overview of Riemannian Geometry and Curvature}\label{app:riem-mfld-intro}

\subsection{Riemannian Metrics and Connections}
In this section, we give an incredibly brief overview of the basics of Riemannian geometry needed to understand the ideas presented in this text. The material presented here is certainly not enough for an appropriate introduction to the vast subject of Riemannian geometry. The interested reader could consult any of \cite{Petersen:RiemGeo}, \cite{GHL:RiemGeo}, \cite{Milnor:MorseTheory}, or \cite{Lee:RiemGeo}, for a proper introduction. Suppose that $M$ is a smooth manifold.\footnote{For an introduction to smooth manifolds, see \cite{Lee:SmMflds}.} A \emph{Riemannian metric} $g$ is a (smooth) symmetric nondegenerate bilinear form on the tangent space $TM$. By this, we mean that for $p\in M$, $g_{p}$ is an inner product on the vector space $T_{p}M$ and for $X$ and $Y$ any smooth vector fields defined in the neighborhood of $p$, the function $g(X,Y)$ defined by $g(X,Y)|_{q} = g_{q}(X_{q},Y_{q})$ is smooth in some neighborhood of $p$. We write $(M,g)$ for a smooth manifold $M$ equipped with a metric $g$. The Riemannian metric allows us to define a preferred notion of directional derivative on $(M,g)$, called the \emph{Levi-Civita connection}. This is a special choice of \emph{affine connection}, which is a map (writing $\Gamma(TM)$ for smooth vector  fields) 
\[ \nabla : \Gamma(TM) \times \Gamma(TM) \to \Gamma(TM), \qquad (X,Y) \mapsto \nabla_{X} Y\] 
obeying the following properties:
\begin{enumerate}
\item It is linear over $C^{\infty}$ (i.e.\ tensorial) in the first slot. That is, for $f,g \in C^{\infty}(M)$, and $X,Y,Z \in \Gamma(TM)$
\[ \nabla_{f X + gZ } Y = f\nabla_{X}Y + g \nabla_{Z} Y .\]
\item It obeys the Leibniz rule and is linear over $\RR$ in the second slot. That is, for $f \in C^{\infty}$, $\alpha,\beta \in \RR$ and $X,Y, Z \in \Gamma(TM)$ 
\[  \nabla_{X} (\alpha Y + \beta Z) = \alpha \nabla_{X}Y +\beta \nabla_{X}Z\]
\[ \nabla_{X} (fY) = f \nabla_{X}Y + X(f) Y.\]
\end{enumerate}
The Levi-Civita connection is the unique affine connection which is torsion free and compatible with the metric $g$. Being \emph{torsion free} means that for $X,Y\in\Gamma(TM)$ 
\[ \addlabel{eq:torsion-free} \nabla_{X}Y - \nabla_{Y}X = [X,Y] \]
and \emph{compatibility with the metric} means that for $X,Y,Z\in \Gamma(TM)$
\[\addlabel{eq:met-compat} X(g(Y,Z)) = g(\nabla_{X}Y,Z) + g(Y,\nabla_{X}Z).  \]
To see that the Levi-Civita connection is unique, it is possible to use \eqref{eq:torsion-free} and \eqref{eq:met-compat} to get an expression in local coordinates for what the connection must be, but on the other hand, it is possible to see that this gives a connection satisfying \eqref{eq:torsion-free} and \eqref{eq:met-compat}. 

A connection also allows us to define the notion of the \emph{covariant derivative} along a curve $\gamma$ which is a $\RR$-linear map $V\to DV/dt$ for $V$ a vector field along $\gamma$ satisfying the Leibniz rule 
\[ \frac{D}{dt} (fV) = \pa ft V + f \frac {DV}{dt}\] and it is induced by the connection: if $V$ is extendible to $\tilde V$, a vector field in a neighborhood of $\gamma$, then
\[ \frac{DV}{dt} = \nabla_{\gamma'} V.\]  
For a curve $\gamma(t)$ and a vector $V \in T_{\gamma(0)}M$, it turns out that there is a unique vector field along $\gamma$, $\tilde V$ such that 
\[ \frac{D\tilde V}{dt} = 0,\] and $\tilde V(0) = V$. In this case, we say that $V$ has been \emph{parallely transported} along $\gamma$.

Covariant differentiation also allows us to give a nice description of \emph{geodesics} which are locally length minimizing curves (notice that this is slightly different than our definition of geodesics in a metric space, which we required to be globally length minimizing). It can be shown that a curve $\gamma:[0,1] \to M$ is a geodesic if and only if 
\[ \nabla_{\dot\gamma} \dot\gamma = 0 .\] Assuming that $(M,g)$ is a compact manifold (in the topology determined by the induced length metric coming from the definition of the length of a smooth curve being $L(\gamma) : = \int_{0}^{1} |\dot\gamma|_{g} dt$), one can show that geodesics exist for all time, and are uniquely determined by their initial position and velocity. This allows us to define the \emph{exponential map} $\exp_{p}(X)$, to be $\gamma(1)$, where $\gamma(t)$ is the unique geodesic starting at $p\in M$ with initial tangent vector $X \in T_{p}M$. All of the following works (sometimes with slight modification, because in some non-compact manifolds, geodesics need not exist for all time) for non-compact manifolds, but since all we consider in the above is compact manifolds, we will assume that $(M,g)$ is compact. 

\subsection{Curvature} The Levi-Civita connection allows us to then define the Riemann curvature tensor
\[ \addlabel{eq:riem-tensor-def} \Riem(X,Y)Z = \nabla_{X}\nabla_{Y}Z - \nabla_{Y}\nabla_{X} Z - \nabla_{[X,Y]}Z.  \] 
This tensor is very complicated, so it is beneficial to consider the following tensor and scalar, the \emph{Ricci tensor} defined as (for $X,Y\in T_{p}M$)
\[ \addlabel{eq:ric-curv-defi} \Ric(X,Y) : = \sum_{i=1}^{n} g(\Riem(X,E_{i})E_{i},Y)\]
and \emph{scalar curvature} 
\[ \addlabel{eq:scal-curv-defi}S: = \sum_{i,j=1}^{n} g(\Riem(E_{j},E_{i})E_{i},E_{j})\]
where $E_{i}$ is an orthonormal basis for the tangent space at the point of interest (it is not hard to see that this is independent of the choice of orthonormal basis). Finally, for any two $g$-orthogonal vectors $X,Y \in T_{p}M$, we define the sectional curvature of the plane spanned by $X$ and $Y$ to be 
\[\addlabel{eq:sec-curv-defi} \sec(X,Y) : =  g(\Riem(X,Y)Y,X). \]
It turns out that these are much more manageable things to study and allow for simpler geometric interpretation than the full Ricci tensor. We give some geometric interpretations of sectional curvature bounds and Ricci curvature bounds in Sections \ref{sect:synth-ricci-curv} and \ref{sect:synth-sect-curv}, additionally discussing Ricci bounds in Appendix \ref{app:mfd-lower-ric}. We will often write $\Ric \geq K$ or $\sec \geq K$ to mean $\Ric(X,Y) \geq K g(X,Y)$ or $\sec(X,Y) \geq K g(X,Y)$ respectively. 

\subsection{Jacobi Fields} Examining the proof of Lemma \ref{lemm:diff-opt-map}, it is apparent that in the theory of optimal transport, a type of vector fields known as Jacobi fields are the link between geometrical properties and Ricci curvature. The classical reference for these is Milnor's book \cite{Milnor:MorseTheory}\footnote{N.B.\ our conventions about the sign of the curvature tensor are different from Milnor's.}, but any of the books mentioned in the introduction to this appendix should also more than suffice as references. A \emph{Jacobi field} arises from a \emph{variation through geodesics}, which is a smooth map
\begin{equation}\label{eq:var-thru-geo}
\alpha: (-\epsilon,\epsilon)\times[0,1] \to M
\end{equation}
so that for a fixed $s\in (-\epsilon,\epsilon)$, $t\mapsto \alpha(s,t)$ is a geodesic. From the map $\alpha$, we get two covariant derivatives $D/ds$ and $D/dt$ for vector fields along $\alpha$, i.e.\ $V: (-\epsilon,\epsilon)\times[0,1] \to TM$ such that $V_{s,t} \in T_{\alpha(s,t)}M$. Using the definition of Riemannian curvature, \eqref{eq:riem-tensor-def}, and computing in local coordinates, it is possible to show (cf. \cite[Lemma 9.2]{Milnor:MorseTheory}) that for a vector field $V$ along $\alpha$
\begin{equation}\label{eq:ric-comm-cov-der}
\frac{D}{ds} \frac{D}{dt} V -\frac{D}{dt} \frac{D}{ds} V = \Riem \left(\pa \alpha s,\pa\alpha t\right) V. 
\end{equation}
Setting $V = \partial\alpha/\partial t$, the tangent vector along the geodesics, \eqref{eq:ric-comm-cov-der} gives that 
\begin{equation}
\frac{D}{ds} \frac{D}{dt} \pa\alpha t -\frac{D}{dt} \frac{D}{ds} \pa\alpha t = \Riem \left(\pa \alpha s,\pa\alpha t\right) \pa\alpha t. 
\end{equation}
Because $\alpha$ is a variation through geodesics, we have that 
\[  \frac{D}{dt} \pa\alpha t = 0 \]
and by the torsion free property of the connection, it is not hard to see that we have
\[ \frac{D}{ds} \pa \alpha t = \frac{D}{dt}\pa \alpha s. \] 
Combining these equations 
\[ \frac{D}{dt}\frac{D}{dt} \pa \alpha s + \Riem \left(\pa \alpha s,\pa\alpha t\right) \pa\alpha t = 0. \]
Evaluating this at $s = 0$ and letting $J = \partial\alpha/\partial s |_{s=0}$ be the vector field along the geodesic $\alpha(0,t)$, we have
\begin{equation}\label{eq:jacobi-eqns}
 \frac{D}{dt}\frac{D}{dt} J + \Riem \left(J,\pa\alpha t\right) \pa\alpha t = 0. 
\end{equation}
This is known as \emph{Jacobi's equation} and any $J$ satisfying it is called a \emph{Jacobi field}. Writing \eqref{eq:jacobi-eqns} in local coordinates, it is possible to show that this is a linear second order ODE, and thus it has a unique solution existing as long as $\gamma$ is defined, as long as we specify $J(0)$ and $\dot J(0)$ (where $\dot J = DJ/dt$). It is possible to show that any $J$ along a geodesic $\gamma$ satisfying \eqref{eq:jacobi-eqns} comes from some variation through geodesics (see \cite[Lemma 14.4]{Milnor:MorseTheory}). 

The reason which we are interested in Jacobi fields is that they allow us to compute the derivative of the exponential map, as follows. Notice that $\exp_{p}:T_{p} \to M$, so the derivative is a map $d\exp_{p}:T(T_{p}M) \to TM$. Because $T_{p}M$ is a vector space, for $V \in T_{p}M$, $T_{V}(T_{p}M)$ is naturally identified with $T_{p}M$. For $W \in T_{V}( T_{p}M)$, we have that
\[ d\exp_{p}(W) = \frac{d}{ds}\Big|_{s=0}\exp_{p}(V + sW) . \] However, notice that defining
\[ \addlabel{eq:defi-jacobi-alpha}\alpha(s,t) = \exp_{p}(t(V + sW)),\]
this is a variation through geodesics, so 
\[J(t) = \pa\alpha s(0,t)\] 
is a Jacobi field along $\gamma(t) = \exp_{p}(tV)$ with $J(1) = d\exp_{p}(W)$. Furthermore, notice that because it is a Jacobi field, $J(t)$ is uniquely described by its initial value and covariant derivative. Examining the definition of $\alpha$, \eqref{eq:defi-jacobi-alpha}, it is easy to see that $J(0) = 0$ and $\dot J(0) = W$ (regarded as an element of $T_{p}M$). Thus $d\exp_{p}(W) = J(1)$, where $J$ is the unique Jacobi field with these initial conditions. A similar calculation, in a more specialized case is done in the proof of Lemma \ref{lemm:diff-opt-map}. There, we calculate the derivative of a map of the form $y \mapsto \exp_{y}( X(y) )$ where $X(y)$ is a vector field on $M$, which follows in roughly the same manner as the above.

\section{Manifolds with Lower Ricci Bounds}\label{app:mfd-lower-ric}

In this section we give an overview of consequences of Ricci lower bounds in the classical sense, for Riemannian manifolds. Recall that if $(M,g)$ is a Riemannian manifold, we say that it has \emph{Ricci curvature bounded below by $K \in \RR$} (which we will write $\Ric \geq K$) if $\Ric(\xi,\xi) \geq K |\xi|^2$ for all $\xi \in TM$.

\subsection{Topological and Geometric Consequences} It is natural to ask what sort of topological conditions must be met so that a manifold admits a metric of Ricci bounded below by $K$. If $K > 0$, we have a classical theorem giving strong topological and geometric restrictions on the manifold
\begin{theo}[Bonnet-Myers]
For a $n$-dimensional Riemannian manifold $(M,g)$, if $\Ric \geq (n-1)K > 0$ then $M$ has a finite fundamental group, i.e.\ $|\pi_1(M)| < \infty$, and \[ \diam(M,g) \leq \frac{\pi}{ \sqrt K}.\]
\end{theo}
The standard proof uses Jacobi fields to study minimizing geodesics, and can be found in \cite[Chapter 4]{Petersen:RiemGeo}. If $K = 0$, then the example $(\RR^n, d_{\RR^n})$ shows that $(M,g)$ need not be compact, but there is the following theorem which 	shows that this is in some sense essentially the only non-compact (simply connected) example. 
\begin{theo}[Cheeger--Gromoll]\label{theo:cheeger-gromoll-splitting}
If $(M,g)$ is a Riemannian manifold with $\Ric \geq 0$, then the universal cover, equipped with the pullback metric $(\tilde M,\tilde g)$ is isometric to the product $(N,g_N) \times (\RR^q, g_{\RR^q})$ for some compact $(N,g_N)$ and $q \geq 0$. 
\end{theo}
Even negative Ricci lower bounds place topological restrictions on the underlying manifold, as the following theorem  shows
\begin{theo}[Gallot--Gromov]
For $(M,g)$ a Riemannian manifold of dimension $n$ and $\diam(M) \leq D$ with $\Ric \geq (n-1)K$, then there is a constant $C = C\left(n ,K D^2\right)$ such that \[ b_1(M) \leq C,\] where $b_1(M)$ is the first Betti number, i.e.\ the dimension of the homology group $H_1(M,\RR)$. Moreover, there is $\epsilon(n) > 0$ such that if $KD^2 \geq -\epsilon(n)$ then \[ b_1(M) \leq n.\]
\end{theo}
Thus, we see that we have a good amount of topological control on manifolds admitting metrics with ``almost nonnegative'' Ricci curvature. Proofs of the above two propositions can be found in \cite[Chapter 9]{Petersen:RiemGeo}. A good overview of these theorems and more information about Ricci lower bounds can be found in \cite{Wei:Ricci}.

We now give the first indication that Ricci curvature bounds are linked to volume measurements. It turns out that lower Ricci bounds give control over the rate at which the volume of balls change with the radius compared to the rate in the standard spaces of constant curvature. 
\begin{theo}[Bishop--Gromov Inequality] 
If $(M,g)$ satisfies $\Ric \geq (n-1) K$ then for all $p\in M$, the map
\[R \mapsto \frac{\vol_M (B_R(p))}{\vol_{S_{K}^{n}} (B_R)} \]
is nonincreasing for $R \in (0,\infty)$, where $\vol_{S_{K}^{n}} (B_R)$ is the volume of a ball of radius $R$ in the simply connected space form of constant sectional curvature $K$.
\end{theo}

This plays a crucial role in the proof of Theorem \ref{theo:gro-cpt-ricci}, which says that the set of manifolds with fixed dimension, upper bounds on diameter and lower bounds on Ricci curvature is precompact in the Gromov--Hausdorff topology. This is because it allows us to uniformly bound the number of disjoint $\epsilon$-balls for such manifolds, which it turns out is a sufficient condition of precompactness, by a Theorem proven by Gromov (cf.\ \cite[Theorem 7.4.15]{BBI:MetGeo} for a easy to follow proof). 

\subsection{Gromov--Hausdorff Limits of Bounded Ricci Manifolds}\label{subsect:GH-lim-bd-ricci}

Theorem \ref{theo:gro-cpt-ricci} thus leads to the natural question of whether or not we can identify any properties of metric spaces which are the limit of manifolds with lower Ricci bounds. Theorem \ref{theo:stab-gh-conv}, which says that generalized Ricci bounds are preserved under Gromov--Hausdorff convergence, combined with Theorem \ref{theo:riem-mfld-equiv}, which says that Riemannian manifolds with lower Ricci bounds also have the same generalized lower bounds, show that the limiting metric spaces have generalized Ricci lower bounds. However, as we will explain below, it turns out that not all spaces with generalized Ricci lower bounds are Gromov--Hausdorff limits of manifolds with Ricci lower bounds. As such, we expect that limits of manifolds with lower Ricci bounds enjoy a richer structure than spaces with generalized lower Ricci bounds. The study of such spaces was basically started by Cheeger and Colding, in their papers \cite{CheegerColding:Ricc1,CheegerColding:Ricc2,CheegerColding:Ricc3}. Most of their results are somewhat out of the range of this exposition, but we mention a few. They showed that if the Hausdorff dimension of the limiting space is the same as the limit manifolds, then except for a set of small codimension, the limiting space is bi-H\"older equivalent to a smooth manifold. Furthermore, in this case, they showed that the isometry group of the limit space is a Lie group. Even if the dimension of the limit decreases (they show that it cannot increase), they are able to define a Laplacian and show that eigenvalues and eigenfunctions pass to the limit under the Gromov--Hausdorff convergence in some sense. 

One thing they proved is that the limit spaces of manifolds with nonnegative Ricci curvature obey a similar sort of splitting theorem as in Theorem \ref{theo:cheeger-gromoll-splitting}. We do not give the details, but only remark that this shows that $(\RR^{n},\|\cdot\|)$ with $\|\cdot\|$ any Banach norm which is not the standard euclidean norm, cannot be a limit of Ricci nonnegative manifolds (because this space cannot split). However, it can be seen that this space has nonnegative generalized Ricci curvature, so as we claimed above, not all spaces with generalized Ricci lower bounds are limits of manifolds with Ricci lower bounds.

\subsection{Upper Ricci Bounds}

To finish this section, we make a few remarks on upper Ricci bounds. It is not unreasonable to wonder if a notion of displacement \emph{concavity} is a fruitful notion, possibly generalizing upper Ricci bounds. In fact it is really not, for a few reasons. First of all, Proposition \ref{prop:ent-01} shows that $\Ent(\cdot| \Leb)$ is not displacement concave on $\cP([0,1])$, even though in this case, $\Ric\equiv 0$. Secondly, it seems unlikely that the notion of displacement concavity will be preserved under Gromov--Hausdorff convergence, because examining the proof of Theorem \ref{theo:stab-gh-conv}, we see that we would need equality in \eqref{eq:ent-lsc-push}, i.e.
\[ \Ent(\nu_{t} |\mu) = \lim_{n\to\infty} \Ent(\tilde \nu_{t}^{(n)}|\mu_{n}) ,\] which from the proof seems unlikely to be true. Finally, we remark that the theory of lower Ricci bounds is far richer and more interesting than that of upper Ricci bounds.\footnote{Two sided bounds on Ricci curvature \emph{are} studied, however; cf. the remarks in Chapter 6 of \cite{Wei:Ricci}.} As an example of this, we mention three theorems of Lohkamp. 
\begin{theo}[\cite{Lohkamp:MetNegRic}, Corollary B]
Every smooth compact manifold $M^{n}$ of dimension $n\geq 3$ admits a metric with negative Ricci curvature, $\Ric < 0$. 
\end{theo}
\begin{theo}[\cite{Lohkamp:DiscontGeoExpans}, Theorem 2] For a fixed smooth  manifold $M^{n}$, any $V \in (0,\infty)$, and any sequence of positive real numbers $0=\lambda_{0} < \lambda_{1} \leq \lambda_{2}\leq \dots$, there is a sequence of metrics $g_{m}$ on $M$ such that
\begin{enumerate}
\item The $k$-th eigenvalue of the Laplacian, $\Delta_{g_{m}}$ is $\lambda_{k}$ for $k\leq m$.
\item The volumes are constant, i.e.\ $\vol(M,g_{m}) = V$.
\item The Ricci curvature is bounded from above by $-m^{2}$, i.e.\ $\Ric_{g_{m}} < -m^{2}$.
\end{enumerate}
\end{theo}
\begin{theo}[\cite{Lohkamp:Curv-h-Princ} Theorem B]
For a fixed $\alpha \in \RR$ and smooth compact manifold $M$, the set of Riemannian metrics on $M$ with $\Ric < \alpha$ is dense in the set of Riemannian metrics on $M$ in the $C^{0}$ topology, and in the Gromov--Hausdorff topology. That is, for any metric on $M$, $g$, we can find $g_{m}$ with $\Ric_{g_{m}} < \alpha$ and where the components of $g_{m}$ converge in the sup-norm to those of $g$, and additionally metrics $g_{k}$ such that $\Ric_{g_{k}} < \alpha$ and \[ (M,g_{k}) \xrightarrow{GH} (M,g) .\]
\end{theo}

The first two theorems shows that the existence of a Ricci-negative metric gives no topological information in direct contrast with the Ricci-positive case. The second theorem, puts further restrictions on the implications of upper Ricci bounds on the eigenvalues of the Laplacian. However, the third is probably the most striking, showing that that Ricci upper bounds cannot be preserved under Gromov--Hausdorff convergence (in contrast to lower bounds).

\nocite{*} 
\bibliography{bibliography}
\bibliographystyle{siam}
\end{document}